\documentclass{amsart}

\usepackage{mathrsfs}
\usepackage{stmaryrd}
\usepackage{enumerate}
\usepackage{tikz-cd}
\usepackage[all]{xy}
\usepackage{aliascnt}
\usepackage[colorlinks=true, linkcolor=blue, citecolor=blue]{hyperref}
\usepackage{fullpage}

\newtheorem{theorem}{Theorem}

\newaliascnt{lemma}{theorem}
\newtheorem{lemma}[lemma]{Lemma}
\aliascntresetthe{lemma}

\newaliascnt{corollary}{theorem}
\newtheorem{corollary}[corollary]{Corollary}
\aliascntresetthe{corollary}

\newaliascnt{proposition}{theorem}
\newtheorem{proposition}[proposition]{Proposition}
\aliascntresetthe{proposition}

\newaliascnt{conjecture}{theorem}
\newtheorem{conjecture}[conjecture]{Conjecture}
\aliascntresetthe{conjecture}

\newaliascnt{question}{theorem}
\newtheorem{question}[question]{Question}
\aliascntresetthe{question}

\theoremstyle{definition}

\newaliascnt{definition}{theorem}
\newtheorem{definition}[definition]{Definition}
\aliascntresetthe{definition}

\newaliascnt{remark}{theorem}
\newtheorem{remark}[remark]{Remark}
\aliascntresetthe{remark}

\newaliascnt{example}{theorem}
\newtheorem{example}[example]{Example}
\aliascntresetthe{example}

\newaliascnt{notation}{theorem}
\newtheorem{notation}[notation]{Notation}
\aliascntresetthe{notation}

\newtheorem*{acknowledgements}{Acknowledgements}

\newcommand{\Z}{\mathbb{Z}}
\newcommand{\Q}{\mathbb{Q}}
\newcommand{\N}{\mathbb{N}}
\newcommand{\R}{\mathbb{R}}

\newcommand{\bL}{\mathbb{L}}
\newcommand{\bG}{\mathbb{G}}
\newcommand{\bA}{\mathbb{A}}

\newcommand{\cX}{\mathcal{X}}
\newcommand{\cU}{\mathcal{U}}
\newcommand{\cY}{\mathcal{Y}}
\newcommand{\cZ}{\mathcal{Z}}
\newcommand{\cF}{\mathcal{F}}
\newcommand{\cO}{\mathcal{O}}
\newcommand{\cC}{\mathcal{C}}
\newcommand{\cD}{\mathcal{D}}

\newcommand{\sL}{\mathscr{L}}
\newcommand{\sM}{\mathscr{M}}
\newcommand{\sJ}{\!\mathscr{J}}
\newcommand{\sR}{\mathscr{R}}

\newcommand{\Var}{\mathbf{Var}}
\newcommand{\Stack}{\mathbf{Stack}}
\newcommand{\Vari}{\mathbf{Var}}

\newcommand{\red}{\mathrm{red}}
\newcommand{\sm}{\mathrm{sm}}
\newcommand{\gp}{\mathrm{gp}}
\newcommand{\Gor}{\mathrm{Gor}}
\newcommand{\diff}{\mathrm{d}}

\DeclareMathOperator{\e}{e}
\DeclareMathOperator{\trop}{trop}
\DeclareMathOperator{\Hom}{Hom}
\DeclareMathOperator{\Spec}{Spec}
\DeclareMathOperator{\ord}{ord}
\DeclareMathOperator{\ordjac}{ordjac}
\DeclareMathOperator{\rk}{rk}
\DeclareMathOperator{\sep}{sep}
\DeclareMathOperator{\id}{id}
\DeclareMathOperator{\cok}{cok}
\DeclareMathOperator{\tor}{tor}
\DeclareMathOperator{\ext}{Ext}
\DeclareMathOperator{\sat}{sat}

\newcommand{\SL}{\mathrm{SL}}
\newcommand{\bw}{\mathbf{w}}
\DeclareMathOperator{\Gr}{Gr}

\tikzset{cong/.style={draw=none,edge node={node [sloped, allow upside down, auto=false]{$\cong$}}},
         Isom/.style={above,every to/.append style={edge node={node [sloped, allow upside down, auto=false]{$\sim$}}}}}

\newcounter{ResumeEnumerate}

\title{Stringy invariants and toric Artin stacks}

\author{Matthew Satriano and Jeremy Usatine}

\thanks{MS was partially supported by a Discovery Grant from the National Science and Engineering Research Council of Canada.}

\address{Matthew Satriano, Department of Pure Mathematics, University of Waterloo}
\email{msatriano@uwaterloo.ca}

\address{Jeremy Usatine, Department of Mathematics, Brown University}
\email{jeremy{\_}usatine@brown.edu}

\begin{document}

\begin{abstract}
We propose a conjectural framework for computing Gorenstein measures and stringy Hodge numbers in terms of motivic integration over arcs of smooth Artin stacks, and we verify this framework in the case of fantastacks, which are certain toric Artin stacks that provide (non-separated) resolutions of singularities for toric varieties. Specifically, let $\mathcal{X}$ be a smooth Artin stack admitting a good moduli space $\pi: \mathcal{X} \to X$, and assume that $X$ is a variety with log-terminal singularities, $\pi$ induces an isomorphism over a nonempty open subset of $X$, and the exceptional locus of $\pi$ has codimension at least 2. We conjecture a change of variables formula relating the motivic measure for $\mathcal{X}$ to the Gorenstein measure for $X$ and functions measuring the degree to which $\pi$ is non-separated. We also conjecture that if the stabilizers of $\mathcal{X}$ are special groups in the sense of Serre, then almost all arcs of $X$ lift to arcs of $\mathcal{X}$, and we explain how in this case (assuming a finiteness hypothesis satisfied by fantastacks), our conjectures imply a formula for the stringy Hodge numbers of $X$ in terms of a certain motivic integral over the arcs of $\mathcal{X}$. We prove these conjectures in the case where $\mathcal{X}$ is a fantastack.
\end{abstract}

\maketitle

\numberwithin{theorem}{section}
\numberwithin{lemma}{section}
\numberwithin{corollary}{section}
\numberwithin{proposition}{section}
\numberwithin{conjecture}{section}
\numberwithin{question}{section}
\numberwithin{remark}{section}
\numberwithin{definition}{section}
\numberwithin{example}{section}
\numberwithin{notation}{section}

\setcounter{tocdepth}{1}
\tableofcontents

\section{Introduction}

Let $X$ be a variety with log-terminal singularities. Motivated by mirror symmetry for singular Calabi-Yau varieties, Batyrev introduced \emph{stringy Hodge numbers} for $X$ in \cite{Batyrev98}, which 
are defined in terms of a resolution of singularities. In particular, if $X$ admits a crepant resolution $Y \to X$ by a smooth projective variety $Y$, then the stringy Hodge numbers of $X$ are equal to the usual Hodge numbers of $Y$. In \cite{DenefLoeser} Denef and Loeser defined the \emph{Gorenstein measure} $\mu^\Gor_X$ on the arc scheme $\sL(X)$ of $X$ and used it to prove a McKay correspondence that refines the McKay correspondence conjectured by Reid in \cite{Reid} and proved by Batyrev in \cite{Batyrev99}. The measure $\mu^\Gor_X$ takes values in a modified Grothendieck ring of varieties $\widehat{\sM}_k[\bL^{1/m}]$ and is a refinement of the stringy Hodge numbers of $X$. If $X$ admits a crepant resolution $Y \to X$, then $\mu^\Gor_X$ is essentially equivalent to the usual motivic measure $\mu_Y$ on $\sL(Y)$ as introduced by Kontsevich in \cite{Kontsevich}.

A major open question asks whether or not the stringy Hodge numbers of projective varieties are nonnegative, as conjectured by Batyrev in \cite[Conjecture 3.10]{Batyrev98}. A stronger conjecture predicts that stringy Hodge numbers of projective varieties are equal to the dimensions of some kind of cohomology groups. In \cite{Yasuda2004}, these conjectures were proved in the case where $X$ has quotient singularites. Yasuda showed that in that case, if $\cX$ is the canonical smooth Deligne--Mumford stack over $X$, then the stringy Hodge numbers of $X$ are equal to the orbifold Hodge numbers of $\cX$ in the sense of Chen and Ruan \cite{ChenRuan}. To prove this result, Yasuda introduced a notion of motivic integration (further developed in \cite{Yasuda2006, Yasuda2019}) for Deligne--Mumford stacks and proved a formula expressing $\mu^\Gor_X$ in terms of certain motivic integrals over arcs of $\cX$. When $X$ is projective, those integrals over arcs of $\cX$ compute the orbifold Hodge numbers of $\cX$.

In this paper, we initiate a similar program for varieties with singularities that are worse than quotient singularities. Such varieties never arise as the coarse space of a smooth Deligne--Mumford stack, so one is instead forced to consider Artin stacks. A major technicality is that such stacks are not separated. This leads us to define new functions $\sep_{\pi, \cC}$, discussed below, which measure the degree to which an Artin stack is not separated. These functions play a key role in our theory.

The class of varieties we consider are those $X$ occurring as the good moduli space (in the sense of \cite{Alper}) of a smooth Artin stack $\cX$; varieties of this form arise naturally in the context of GIT. We require that the map $\pi\colon\cX \to X$ induces an isomorphism over a nonempty open subset of $X$ and that the exceptional locus of $\pi$ has codimension at least 2. In other words, we want $\cX$ to be a ``small'' resolution of $X$. We conjecture a relationship between $\mu^\Gor_X$ and a motivic measure $\mu_\cX$ on the arc stack $\sL(\cX)$ of $\cX$. This relationship involves integrating $\sep_{\pi, \cC}: \sL(X) \to \N$ against $\mu^\Gor_X$. This function $\sep_{\pi, \cC}$ counts the number of arcs of $\cX$ (in some auxiliary measurable subset $\cC$), up to isomorphism, above each arc of $X$, and can therefore be thought of as an invariant which measures the non-separatedness of $\pi$. We emphasize that this conjectural relationship is not ``built into’’ our definition of $\mu_\cX$. In fact, our notion of $\mu_\cX$ is straightforward: it is more-or-less Kontsevich’s original motivic measure, except various notions for schemes are replaced with the obvious analogs for Artin stacks. When the stabilizers of $\cX$ are special groups in the sense of Serre\footnote{$G$ is special if every $G$-torsor is Zariski locally trivial.} and $\sL(\pi): \sL(\cX) \to \sL(X)$ has finite fibers outside a set of measure 0, our conjectures imply a formula expressing the stringy Hodge numbers of $X$ in terms of a certain motivic integral over $\sL(\cX)$.

We prove that our conjectures hold when $X$ is a toric variety and $\cX$ is a fantastack, i.e., a type of smooth toric Artin stack in the sense of \cite{GeraschenkoSatriano1, GeraschenkoSatriano2}. Fantastacks are a broad class of toric stacks that allow one to simultaneously have any specified toric variety $X$ as a good moduli space while also obtaining stabilizers with arbitrarily large dimension. An important special case of fantastacks (and their products with algebraic tori) is the so-called canonical stack $\cX$ over a toric variety $X$. When $X$ has quotient singularities, $\cX$ is the canonical smooth Deligne-Mumford stack over $X$; when $X$ has worse singularities, the good moduli space of $\cX$ is still $X$, but $\cX$ is an Artin stack that is not Deligne-Mumford.

\subsection{Conventions}

Throughout this paper, $k$ will be an algebraically closed field with characteristic 0. All Artin stacks will be assumed to have affine (geometric) stabilizers, and all toric varieties will be assumed to be normal. For any stack $\cX$ over $k$, we will let $|\cX|$ denote the topological space associated to $\cX$, and for any $k$-algebra $R$, we will let $\overline{\cX}(R)$ denote the set of isomorphism classes of the category $\cX(R)$.

\subsection{Conjectures}

Our first conjecture predicts a relationship between $\mu^\Gor_X$ and $\mu_\cX$. As mentioned above, our formula involves integrals weighted by functions $\sep_{\pi, \cC}$ that measure the degree to which $\pi$ is not separated. We refer the reader to \autoref{sectionMotivicIntegrationQuotientStacks} for precise definitions of the arc stack $\sL(\cX)$ and its motivic measure $\mu_\cX$, and to \autoref{subsectionNonSeparatednessFunctions} for the definition of $\sep_{\pi, \cC}$ and its integral $\int_C \sep_{\pi, \cC} \diff\mu^\Gor_X$.

\begin{conjecture}
\label{conjectureGorensteinMeasureAndStackMeasure}
Let $\cX$ be a smooth irreducible Artin stack over $k$ admitting a good moduli space $\pi: \cX \to X$, where $X$ is a separated $k$-scheme and has log-terminal singularities. Assume that $\pi$ induces an isomorphism over a nonempty open subset of $X$, and that the exceptional locus of $\pi$ has codimension at least $2$.

If $\cC \subset |\sL(\cX)|$ is a measurable subset such that $\sep_{\pi, \cC}: \sL(X) \to \N \cup \{\infty\}$ is finite outside a set of measure $0$, then $\sep_{\pi, \cC}: \sL(X) \to \N \cup \{\infty\}$ has measurable fibers, and for any measurable subset $C \subset \sL(X)$, the set $\cC \cap \sL(\pi)^{-1}(C) \subset |\sL(\cX)|$ is measurable and satisfies
\[
	\mu_\cX(\cC \cap \sL(\pi)^{-1}(C)) = \int_C \sep_{\pi, \cC} \diff\mu^\Gor_X \in \widehat{\sM}_k[\bL^{1/m}],
\]
where $m \in \Z_{>0}$ is such that $mK_X$ is Cartier.
\end{conjecture}

\autoref{conjectureGorensteinMeasureAndStackMeasure} predicts that for the purpose of computing $\mu^\Gor_X$, the stack $\cX$ behaves like a crepant resolution of $X$, except we need to correct by $\sep_{\pi, \cC}$ to account for the fact that $\cX$ is not separated over $X$. Set
\[
	\sep_\pi = \sep_{\pi, |\sL(\cX)|}.
\]
Notice that \autoref{conjectureGorensteinMeasureAndStackMeasure} implies, in particular, that the motivic measure $\mu_\cX$ ``does not see'' how $\mu^\Gor_X$ behaves on the set $\sep_\pi^{-1}(0) \subset \sL(X)$. This set can have nonzero measure because $\pi: \cX \to X$ does not necessarily satisfy the ``strict valuative criterion'', i.e., there may exist arcs of $X$ (even outside a set of measure 0) that do not lift to arcs of $\cX$. Thus in general, we cannot use this conjecture to compute the total Gorenstein measure $\mu^\Gor_X(\sL(X))$, which specializes to the stringy Hodge numbers of $X$. This issue already occurs in the case where $\cX$ is a Deligne--Mumford stack. For this reason, Yasuda uses a notion of ``twisted arcs'' of $\cX$ instead of usual arcs of $\cX$, and this is why the inertia of $\cX$ and orbifold Hodge numbers appear in Yasuda's setting. We take a different approach, emphasizing a setting in which the next conjecture predicts that almost all arcs of $X$ lift to arcs of $\cX$.

\begin{conjecture}
\label{conjectureSpecialStabilizersLiftArcs}
Let $\cX$ be a finite type Artin stack over $k$ admitting a good moduli space $\pi: \cX \to X$. Assume $X$ is an irreducible $k$-scheme and that $\pi$ induces an isomorphism over a nonempty open subset of $X$. If the stabilizers of $\cX$ are all special groups, then $\sep_\pi^{-1}(0) \subset \sL(X)$ is measurable and
\[
	\mu_X( \sep_\pi^{-1}(0) ) = 0,
\]
where we note that $\mu_X$ is the usual (non-Gorenstein) motivic measure on $\sL(X)$.
\end{conjecture}

\begin{remark}
All special groups are connected, so if $\cX$ is a Deligne--Mumford stack whose stabilizers are special groups, then the stabilizers of $\cX$ are all trivial. Thus \autoref{conjectureSpecialStabilizersLiftArcs} highlights a setting that is ``orthogonal'' to the setting considered by Yasuda.
\end{remark}

Our next question is motivated by the fact that if $\sep_\pi$ is finite outside a set of measure 0, we may then consider the special case of \autoref{conjectureGorensteinMeasureAndStackMeasure} where $\cC = |\sL(\cX)|$.

\begin{question}
\label{questionMuXsepInfinity}
Let $\cX$ be a finite type Artin stack over $k$ admitting a good moduli space $\pi: \cX \to X$. Assume $X$ is an irreducible $k$-scheme and that $\pi$ induces an isomorphism over a nonempty open subset of $X$. When is
\[
	\mu_X(\sep_\pi^{-1}(\infty)) = 0
\]
satisfied?
\end{question}

We now give an application of this framework to computing stringy Hodge numbers. In  \autoref{subsectionNonSeparatednessFunctions}, we introduce the function $\sep_\cX = 1/(\sep_\pi \circ \sL(\pi)): |\sL(\cX)| \to \Q_{\geq 0} \cup \{\infty\}$. We think of its integral $\int_{\sL(\cX)} \sep_\cX \diff\mu_\cX$ as a kind of motivic class of $\sL(\cX)$ corrected by $\sep_\cX$ to account for the fact that $\cX$ is not separated. We refer the reader to \autoref{subsectionNonSeparatednessFunctions} for the precise definitions of $\int_{\sL(\cX)} \sep_\cX \diff\mu_\cX$ and the ring $\widehat{\sM_k \otimes_\Z \Q}$.  The next proposition is then immediate.

\begin{proposition}
\label{propositionConsequenceOfConjectures}
With hypotheses as in \autoref{conjectureGorensteinMeasureAndStackMeasure}, if the stablizers of $\cX$ are special groups and $\mu_X(\sep_\pi^{-1}(\infty)) = 0$, then \autoref{conjectureGorensteinMeasureAndStackMeasure} and \autoref{conjectureSpecialStabilizersLiftArcs} imply that the fibers of $\sep_\cX: |\sL(\cX)| \to \Q_{\geq 0}$ are measurable and
\[
	\mu^\Gor_X(\sL(X)) = \int_{\sL(\cX)} \sep_\cX \diff\mu_\cX \in \widehat{\sM_k \otimes_\Z \Q}.
\]
\end{proposition}

Since the stringy Hodge--Deligne invariant of $X$ is a specialization of the image of $\mu^\Gor_X(\sL(X))$ in $(\widehat{\sM_k \otimes_\Z \Q})[\bL^{1/m}] \supset \widehat{\sM_k \otimes_\Z \Q}$, \autoref{propositionConsequenceOfConjectures} provides a conjectural formula for the stringy Hodge numbers of $X$ (when the stringy Hodge numbers exist, i.e., when the stringy Hodge--Deligne invariant is a polynomial).

We envision a few potential applications of this framework. Noting that the good moduli space map $\pi: \cX \to X$ is intrinsic to the stack $\cX$ and therefore so is the integral $\int_{\sL(\cX)} \sep_\cX \diff\mu_\cX$, we hope that a cohomological interpretation of $\int_{\sL(\cX)} \sep_\cX \diff\mu_\cX$ will lead to progress on Batyrev’s conjecture on the non-negativity of stringy Hodge numbers. We also hope that, by considering \autoref{propositionConsequenceOfConjectures} as a kind of McKay correspondence, our conjectures will lead to new representation-theoretic statements for positive dimensional algebraic groups.

\begin{remark}
The hypothesis $\mu_X(\sep_\pi^{-1}(\infty)) = 0$ in \autoref{propositionConsequenceOfConjectures} allows us to make a canonical choice for $\cC$ in \autoref{conjectureGorensteinMeasureAndStackMeasure}, specifically the choice $\cC = |\sL(\cX)|$. We hope that even when this hypothesis does not hold, one can still (after an appropriate generalization of the notion of arc) make a canonical choice for $\cC$. This is a subject of the authors' ongoing research.
\end{remark}

\subsection{Main results}

Our first main result is that \autoref{conjectureGorensteinMeasureAndStackMeasure} holds (and $\mu_X(\sep_\pi^{-1}(\infty)) = 0$) when $\cX$ is a fantastack and $\cC = |\sL(\cX)|$. In particular, our framework applies to the Gorenstein measure of any toric variety $X$ with log-terminal singularities.

\begin{theorem}
\label{mainTheoremStackMeasureAndGorensteinMeasure}
\autoref{conjectureGorensteinMeasureAndStackMeasure} holds and $\mu_X(\sep_\pi^{-1}(\infty)) = 0$ when $\cX$ is a fantastack and $\cC = |\sL(\cX)|$.
\end{theorem}

\begin{remark}\label{rmk:mainTheoremStackMeasureAndGorensteinMeasure-more-general}
In fact, our techniques prove a more general result:~the conclusions of \autoref{conjectureGorensteinMeasureAndStackMeasure} hold when $\cC = |\sL(\cX)|$ and $\cX$ is a fantastack satisfying a certain combinatorial condition analogous to $\cX \to X$ being ``crepant'' (see \autoref{remarkCrepantDefinition} for more details). It is important to note here that unlike the case of Deligne-Mumford stacks, defining $K_\cX$ for Artin stacks is a subtle issue, so there is no \emph{a priori} obvious definition one can take for $\cX\to X$ to be crepant.
\end{remark}

\begin{remark}
We note that the stacks $\cX$ in \autoref{mainTheoremStackMeasureAndGorensteinMeasure} all have commutative stabilizers. In order to provide evidence that \autoref{conjectureGorensteinMeasureAndStackMeasure} should not be limited to the setting of commutative stabilizers, we also verify that it holds in examples that involve $\mathrm{SL}_2$ as a stabilizer. See \autoref{quotientsbySL2andconesoverGrassmannians}. These examples also demonstrate the flexibility in choosing the auxiliary set $\cC \subset |\sL(\cX)|$.
\end{remark}

\begin{remark}
\autoref{mainTheoremStackMeasureAndGorensteinMeasure} can be thought of as a motivic change of variables formula. We note that Balwe introduced versions of motivic integration for Artin $n$-stacks \cite{Balwe2008, Balwe2015} and proved a change of variables formula \cite[Theorem 7.2.5]{Balwe2008}. However \autoref{mainTheoremStackMeasureAndGorensteinMeasure} cannot be obtained from Balwe’s result, as the map $\pi: \cX \to X$ does not satisfy Balwe’s hypotheses: specifically $\pi$ is not ``0-truncated''.
\end{remark}

The three main steps of proving \autoref{mainTheoremStackMeasureAndGorensteinMeasure} are as follows. First, we give a combinatorial description of the fibers of the map $\sL(\pi): \sL(\cX) \to \sL(X)$. Second, we show that for sufficiently large $n$, the map of jets $\sL_n(\pi): \sL_n(\cX) \to \sL_n(X)$ has constant fibers (after taking the fibers' reduced structure) over certain combinatorially defined pieces of $\sL_n(X)$. These two steps allow us to reduce \autoref{mainTheoremStackMeasureAndGorensteinMeasure} to the final step: verifying the case where the measurable sets $C$ are certain combinatorially defined subsets of $\sL(X)$. A key ingredient in this final step is \autoref{theoremMainMotivicIntegrationForQuotientStacks} and its corollary, \autoref{corollaryMotivicMeasureOfMeasurableSet}, which show how to compute the motivic measure of the stack quotient of a variety by the action of a special group.

Our second main result is that \autoref{conjectureSpecialStabilizersLiftArcs} holds for fantastacks.

\begin{theorem}
\label{mainTheoremLiftingToUntwistedArcsSpecialStabilizers}
\autoref{conjectureSpecialStabilizersLiftArcs} holds when $\cX$ is a fantastack.
\end{theorem}

An essential ingredient in proving \autoref{mainTheoremLiftingToUntwistedArcsSpecialStabilizers} is \autoref{theoremFantastackConnectedStabilizers}, which may be of independent interest, as it provides a combinatorial criterion to check whether or not the stabilizers of a fantastack are special groups.

\begin{acknowledgements}
We thank Dan Abramovich, Dan Edidin, Jack Hall, Martin Olsson, and Karl Schwede for helpful conversations. We are also grateful to the anonymous referee for a thorough reading of our paper and many suggestions which improved the paper.
\end{acknowledgements}

\section{Preliminaries}\label{sectionPreliminaries}

In this section, we introduce notation and recall some facts about motivic integration for schemes and the Gorenstein measure, the Grothendieck ring of stacks and constructible subsets, and toric Artin stacks.

\subsection{Motivic integration for schemes}\label{subsectionPreliminariesMotivicIntegrationSchemes}

If $X$ is a $k$-scheme, for each $n \in \N$ we will let $\sL_n(X)$ denote the $n$th jet scheme of $X$, for each $n \geq m$ we will let $\theta^n_m: \sL_n(X) \to \sL_m(X)$ denote the truncation morphism, we will let $\sL(X) = \varprojlim_n\sL_n(X)$ denote the arc scheme of $X$, and for each $n \in \N$ we will let $\theta_n: \sL(X) \to \sL_n(X)$ denote the canonical morphism, which is also referred to as a truncation morphism. For any $k$-algebra $R$ and $k$-scheme $X$, the map $X(R\llbracket t \rrbracket) \to \sL(X)(R)$ is bijective by \cite[Theorem 1.1]{Bhatt}, and we will often implicitly make this identification.

We will let $K_0(\Vari_k)$ denote the Grothendieck ring of finite type $k$-schemes, for each finite type $k$-scheme $X$ we will let $\e(X) \in K_0(\Var_k)$ denote its class, we will let $\bL = \e(\bA_k^1) \in K_0(\Var_k)$ denote the class of the affine line, and for each constructible subset $C$ of a finite type $k$-scheme we will let $\e(C) \in K_0(\Var_k)$ denote its class.

We will let $\sM_k$ denote the ring obtained by inverting $\bL$ in $K_0(\Var_k)$. For each $\Theta \in \sM_k$, let $\dim(\Theta) \in \Z \cup \{-\infty\}$ denote the infimum over all $d \in \Z$ such that $\Theta$ is in the subgroup of $\sM_k$ generated by elements of the form $\e(X)\bL^{-n}$ with $\dim(X) - n \leq d$, and let $\Vert \Theta \Vert = \exp(\dim(\Theta))$. We will let $\widehat{\sM}_k$ denote the separated completion of $\sM_k$ with respect to the non-Archimedean semi-norm $\Vert \cdot \Vert$, and we will also let $\Vert \cdot \Vert$ denote the non-Archimedean norm on $\widehat{\sM}_k$. For any $m \in \Z_{>0}$, we will let $\widehat{\sM}_k[\bL^{1/m}] = \widehat{\sM}_k[t]/(t^m-\bL)$, we will let $\bL^{1/m}$ denote the image of $t$ in $\widehat{\sM}_k[\bL^{1/m}]$, and we will endow $\widehat{\sM}_k[\bL^{1/m}]$ with the topology induced by the equality
\[
	\widehat{\sM}_k[\bL^{1/m}] = \bigoplus_{\ell = 0}^{m-1} \widehat{\sM}_k \cdot (\bL^{1/m})^\ell,
\]
where each summand $\widehat{\sM}_k \cdot (\bL^{1/m})^\ell$ has the topology induced by the bijection 
\[
	\widehat{\sM}_k \to \widehat{\sM}_k \cdot (\bL^{1/m})^\ell: \Theta \mapsto \Theta \cdot (\bL^{1/m})^\ell.
\]
We note that above and throughout this paper, if $\Theta$ is an element of $K_0(\Var_k)$, $\sM_k$, or $\widehat{\sM}_k$, we slightly abuse notation by also using $\Theta$ to refer to its image under any of the ring maps $K_0(\Var_k) \to \sM_k \to \widehat{\sM}_k \to \widehat{\sM}_k[\bL^{1/m}]$.

If $X$ is an equidimensional finite type $k$-scheme and $C \subset \sL(X)$ is a cylinder, i.e., $C =(\theta_n)^{-1}(C_n)$ for some $n \in \N$ and some constructible subset $C_n \subset \sL_n(X)$, we will let $\mu_X(C) \in \widehat{\sM}_k$ denote the motivic measure of $C$, so by definition
\[
	\mu_X(C) = \lim_{n \to \infty} \e(\theta_n(C))\bL^{-(n+1)\dim X} \in \widehat{\sM}_k,
\]
where we note that each $\theta_n(C)$ is constructible (for example, by \cite[Chapter 5 Corollary 1.5.7(b)]{ChambertLoirNicaiseSebag}) and the above limit exists (for example, by \cite[Chapter 6 Theorem 2.5.1]{ChambertLoirNicaiseSebag}). The motivic measure $\mu_X$ can be extended to the class of so-called measurable subsets of $\sL(X)$, whose definition we now recall.

\begin{definition}
Let $X$ be an equidimensional finite type scheme over $k$, let $C \subset \sL(X)$, let $\varepsilon \in \R_{>0}$, let $I$ be a set, let $C^{(0)} \subset \sL(X)$ be a cylinder, and let $\{C^{(i)}\}_{i \in I}$ be a collection of cylinders in $\sL(X)$.

The data $(C^{(0)}, (C^{i})_{i \in I})$ is called a \emph{cylindrical $\varepsilon$-approximation} of $C$ if
\[
	(C \cup C^{(0)}) \setminus (C \cap C^{(0)}) \subset \bigcup_{i \in I} C^{(i)}
\]
and for all $i \in I$,
\[
	\Vert \mu_X(C^{(i)}) \Vert < \varepsilon.
\]
\end{definition}

\begin{definition}
Let $X$ be an equidimensional finite type scheme over $k$, and let $C \subset \sL(X)$. The set $C$ is called \emph{measurable} if for any $\varepsilon \in \R_{>0}$, there exists a cylindrical $\varepsilon$-approximation of $C$. The \emph{motivic measure} of a measurable subset $C \subset \sL(X)$ is defined to be the unique element $\mu_X(C) \in \widehat{\sM}_k$ such that for any $\varepsilon \in \R_{>0}$ and any cylindrical $\varepsilon$-approximation $(C^{(0)}, (C^{(i)})_{i \in I})$ of $C$, we have
\[
	\Vert \mu_X(C) - \mu_X(C^{(0)}) \Vert < \varepsilon.
\]
Such an element $\mu_X(C)$ exists by \cite[Chapter 6 Theorem 3.3.2]{ChambertLoirNicaiseSebag}.
\end{definition}

For the remainder of this subsection, let $X$ be an integral finite type separated $k$-scheme with log-terminal singularities. We will set notation relevant for the Gorenstein measure associated to $X$. We will let $K_X$ denote the canonical divisor on $X$. If $m \in \Z_{>0}$ is such that $mK_X$ is Cartier, we will let $\omega_{X,m} = \iota_*( (\Omega_{X_\sm}^{\dim X})^{\otimes m} )$ where $\iota: X_\sm \hookrightarrow X$ is the inclusion of the smooth locus of $X$, and we will let $\sJ_{X,m}$ denote the unique ideal sheaf on $X$ such that the image of $(\Omega_X^{\dim X})^{\otimes m} \to \omega_{X,m}$ is equal to $\sJ_{X,m} \omega_{X,m}$. If $C \subset \sL(X)$ is measurable, we will let $\mu^\Gor_X(C)$ denote the Gorenstein measure of $C$, so by definition
\begin{align*}
	\mu^\Gor_X(C) &= \int_C (\bL^{1/m})^{\ord_{\sJ_{X,m}}} \diff \mu_X \\
	&= \sum_{n = 0}^\infty (\bL^{1/m})^n \mu_X(\ord_{\sJ_{X,m}}^{-1}(n) \cap C) \in \widehat{\sM}_k[\bL^{1/m}],
\end{align*}
where $m \in \Z_{>0}$ is such that $mK_X$ is Cartier and $\ord_{\sJ_{X,m}}: \sL(X) \to \N \cup \{\infty\}$ is the order function of the ideal sheaf $\sJ_{X,m}$. The following proposition is easy to check using the definition of $\mu^\Gor_X$ and standard properties of $\mu_X$ given in \cite[Chapter 6 Proposition 3.4.3]{ChambertLoirNicaiseSebag}.

\begin{proposition}
\label{countableAdditivityOfGorensteinMeasure}
Let $\{C^{(i)}\}_{i \in \N}$ be a sequence of pairwise disjoint measurable subsets of $\sL(X)$ such that $C = \bigcup_{i = 0}^\infty C^{(i)}$ is measurable. Then
\[
	\lim_{i \to \infty} \mu^\Gor_X(C^{(i)}) = 0,
\]
and
\[
	\mu^\Gor_X(C) = \sum_{i =0}^\infty \mu^\Gor_X(C^{(i)}).
\]
\end{proposition}

\subsection{The Grothendieck ring of stacks and constructible subsets}

We will let $K_0(\Stack_k)$ denote the Grothendieck ring of stacks in the sense of \cite{Ekedahl}, and for each finite type Artin stack $\cX$ over $k$, we will let $\e(\cX) \in K_0(\Stack_k)$ denote the class of $\cX$. If $K_0(\Var_k)[\bL^{-1}, \{(\bL^n - 1)^{-1}\}_{n \in \Z_{>0}}]$ is the ring obtained from $K_0(\Var_k)$ by inverting $\bL$ and $(\bL^n-1)$ for all $n \in \Z_{>0}$, then the obvious ring map $K_0(\Var_k) \to K_0(\Stack_k)$ induces an isomorphism
\[
	K_0(\Var_k)[\bL^{-1}, \{(\bL^n - 1)^{-1}\}_{n \in \Z_{>0}}] \cong K_0(\Stack_k)
\]
by \cite[Theorem 1.2]{Ekedahl}. Therefore there exists a unique ring map
\[
	K_0(\Stack_k) \to \widehat{\sM}_k
\]
whose composition with $K_0(\Var_k) \to K_0(\Stack_k)$ is the usual map $K_0(\Var_k) \to \widehat{\sM}_k$. If $\Theta \in K_0(\Stack_k)$, we will slightly abuse notation by also using $\Theta$ to refer to its image under $K_0(\Stack_k) \to \widehat{\sM}_k$. By \cite[Proposition 1.1(iii) and Proposition 1.4(i)]{Ekedahl}, if $G$ is a special group over $k$, then $\e(G) \in K_0(\Stack_k)$ is a unit and for any finite type $k$-scheme $X$ with $G$-action, the class of the stack quotient is
\[
	\e([X/G]) = \e(X)\e(G)^{-1} \in K_0(\Stack_k).
\]

\begin{remark}
\label{jetSchemeOfSpecialGroupIsSpecial}
Let $G$ be an algebraic group over $k$. For each $n \in \N$, we give $\sL_n(G)$ the group structure induced by applying the functor $\sL_n$ to the group law $G \times_k G \to G$. It is easy to verify that for each $n \in \N$, we have a short exact sequence 
\[
	1 \to \mathfrak{g} \to \sL_{n+1}(G) \xrightarrow{\theta^{n+1}_n} \sL_n(G) \to 1,
\]
where $\mathfrak{g}$ is the Lie algebra of $G$. Thus by induction on $n$, the fact that $\bG_a$ is special, the fact that extensions of special groups are special, and the fact that $\sL_0(G) \cong G$, we see that if $G$ is a special group, then each jet scheme $\sL_n(G)$ is a special group.
\end{remark}

To state the next result, we recall that if $\cX$ is a finite type Artin stack over $k$, then the topological space $|\cX|$ is Noetherian, so its constructible subsets are precisely those subsets that can be written as a finite union of locally closed subsets.

\begin{proposition}
\label{existenceOfClassOfConstructibleSubsetOfStack}
Let $\cX$ be a finite type Artin stack over $k$ and let $\cC \subset |\cX|$ be a constructible subset. Then there exists a unique $\e(\cC) \in K_0(\Stack_k)$ that satisfies the following property. If $\{\cX_i\}_{i \in I}$ is a finite collection of locally closed substacks $\cX_i$ of $\cX$ such that $\cC$ is equal to the disjoint union of the $|\cX_i|$, then
\[
	\e(\cC) = \sum_{i \in I}\e(\cX_i) \in K_0(\Stack_k).
\]
\end{proposition}

\begin{proof}
The proposition holds by the exact same proof used for the analogous statement for schemes in \cite[Chapter 2 Corollary 1.3.5]{ChambertLoirNicaiseSebag}.
\end{proof}

If $\cX$ is a finite type Artin stack and $\cC \subset |\cX|$ is a constructible subset, we will let $\e(\cC)$ denote the class of $\cC$, i.e., $\e(\cC)$ is as in the statement of \autoref{existenceOfClassOfConstructibleSubsetOfStack}.

We end this subsection with a useful tool to compute the class of a stack.

\begin{definition}
Let $S$ be a scheme, let $Z$ be scheme over $S$, let $\cY$ and $\cF$ be Artin stacks over $S$, and let $\xi: \cY \to Z$ be a morphism over $S$. We say $\xi$ is a \emph{piecewise trivial fibration with fiber $\cF$} if there exists a finite cover $\{Z_i\}_{i \in I}$ of $Z$ consisting of pairwise disjoint locally closed subschemes $Z_i \subset Z$ such that for all $i \in I$,
\[
	(\cY \times_Z Z_i)_\red \cong (\cF \times_S Z_i)_\red
\]
as stacks over $(Z_i)_\red$.
\end{definition}

\begin{remark}
\label{remarkPiecewiseTrivialFibrationGivesProductOfClasses}
Let $Z$ be a finite type scheme over $k$, let $\cY$ and $\cF$ be finite type Artin stacks over $k$, and let $\xi: \cY \to Z$ be a piecewise trivial fibration with fiber $\cF$. Then by \autoref{existenceOfClassOfConstructibleSubsetOfStack},
\[
	\e(\cY) = \e(\cF)\e(Z) \in K_0(\Stack_k).
\]
\end{remark}

The next proposition is well known in the case where $\cY$ is a scheme.

\begin{proposition}
\label{piecewiseTrivialFibrationCriterion}
Let $S$ be a Noetherian scheme, let $Z$ be a finite type scheme over $S$, let $\cY$ and $\cF$ be finite type Artin stacks over $S$, and let $\xi: \cY \to Z$ be a morphism over $S$. Then $\xi$ is a piecewise trivial fibration with fiber $\cF$ if and only if for all $z \in Z$, there exists an isomorphism 
\[
	(\cY \times_Z \Spec(k(z)))_\red \cong (\cF \times_S \Spec(k(z)))_\red
\]
of stacks over $k(z)$, where $k(z)$ denotes the residue field of $z$.
\end{proposition}

\begin{proof}
If $\xi$ is a piecewise trivial fibration with fiber $\cF$, then for every $z\in Z$, there is a locally closed subset $Z'\subseteq Z$ containing $z$ for which $(\cY\times_Z Z')_{\red}\cong(\cF\times_S Z')_{\red}$ as $Z'_{\red}$-stacks. Then
\[
\begin{split}
(\cY\times_Z \Spec k(z))_{\red} &=((\cY\times_Z Z')_{\red}\times_{Z'_{\red}} \Spec k(z))_{\red}\\
&\cong((\cF\times_S Z')_{\red}\times_{Z'_{\red}} \Spec k(z))_{\red}=(\cF\times_S \Spec k(z))_{\red}.
\end{split}
\]
We now show the converse holds. Since 
\[
	(\cY_{\red}\times_Z \Spec k(z))_{\red}=(\cY\times_Z \Spec k(z))_{\red}
\]
for every $z\in Z$, we can assume $\cY$ is reduced. By Noetherian induction on $Z$, we need only find a non-empty open subset $U\subseteq Z$ for which $(\cY\times_Z U)_{\red}\cong(\cF\times_S U)_{\red}$. Let $z\in Z$ be the generic point of an irreducible component of $Z$; replacing $Z$ by an open affine neighborhood of $z$, we may further assume $Z$ is affine. Since $\cO_{Z,z}$ is a field, $\cY\times_Z \Spec k(z)$ is reduced and we hence have a surjective closed immersion
\[
\iota\colon\cY\times_Z \Spec k(z)\cong(\cF\times_S \Spec k(z))_{\red}\to\cF\times_S \Spec k(z).
\]
Now, $\Spec\cO_{Z,z}=\lim_\lambda U_\lambda$ is the inverse limit of open affine neighborhoods $U_\lambda\subseteq Z$ of $z$. Since $Z$ is affine, each map $U_\lambda\to Z$ is affine. Note also that $\cY$ is Noetherian, hence quasi-compact and quasi-separated, and that $\cF\times_S \Spec k(z)\to \Spec k(z)$ is locally of finite presentation. \cite[Proposition B.2]{Rydh2015} then shows there is some index $\lambda$ and a morphism $\iota_\lambda\colon\cY\times_Z U_\lambda\to\cF\times_S U_\lambda$ whose base change to $\Spec\cO_{Z,z}$ is $\iota$. Furthermore, since $\cF\times_S \Spec k(z)\to \Spec k(z)$ and $\xi$ are both of finite presentation, \cite[Proposition B.3]{Rydh2015} shows that after replacing $\lambda$ by a larger index if necessary, we can assume $\iota_\lambda$ is a surjective closed immersion, and hence defines an isomorphism $(\cY\times_Z U_\lambda)_{\red}\cong(\cF\times_S U_\lambda)_{\red}$.
\end{proof}

\subsection{Toric Artin stacks}
In this subsection, we briefly review the theory of toric stacks introduced in \cite{GeraschenkoSatriano1}, as well as establish some notation. Since the focus in our paper is on the toric variety $X$, and the toric stack $\cX$ is viewed as a stacky resolution of $X$, we introduce some notational changes to emphasize this focus.

\begin{definition}\label{def:stacky-fan}
 A \emph{stacky fan} is a pair $(\widetilde{\Sigma},\nu)$, where $\widetilde{\Sigma}$ is a fan on a lattice $\widetilde{N}$ and $\nu\colon\widetilde{N}\to N$ is a homomorphism to a lattice $N$ so that the cokernel $\cok\nu$ is finite.
\end{definition}

A stacky fan $(\widetilde{\Sigma},\nu)$ gives rise to a toric stack as follows. Let $X_{\widetilde\Sigma}$ be the toric variety associated to $\widetilde{\Sigma}$. Since $\cok\nu$ is finite, $\nu^*$ is injective, so we obtain an a surjective homomorphism of tori
\[
\widetilde{T}:=\Spec k[\widetilde{N}^*]\longrightarrow\Spec k[N^*]=:T.
\]
Let $G_\nu$ denote the kernel of this map. Since $\widetilde{T}$ is the torus of $X_{\widetilde\Sigma}$, we obtain a $G_\nu$-action on $X_{\widetilde\Sigma}$ via the inclusion $G_\nu\subset\widetilde{T}$.

\begin{definition}\label{def:fan->toric-stack}
With notation as in the above paragraph, if $(\widetilde{\Sigma},\nu)$ is a stacky fan, the associated \emph{toric stack} is defined to be
\[
\cX_{\widetilde{\Sigma},\nu}:=[X_{\widetilde{\Sigma}}/G_\nu].
\]
When $\widetilde{\Sigma}$ is the fan generated by the faces of a single cone $\widetilde{\sigma}$, we denote $\cX_{\widetilde{\Sigma},\nu}$ by $\cX_{\widetilde{\sigma},\nu}$.
\end{definition}

\begin{example}
If $\Sigma$ is a fan on a lattice $N$ and we let $\nu$ be the identity map, then $\cX_{\Sigma,\nu}=X_\Sigma$. Thus, every toric variety is an example of a toric stack.
\end{example}

In this paper, we concentrate in particular on fantastacks introduced in \cite[Section 4]{GeraschenkoSatriano1}. These play a particularly important role for us since they allow us to start with a toric variety $X_\Sigma$ and produce a smooth stack $\cX$ with arbitrary degree of stackyness while maintaining the property that $X$ is the good moduli space of $\cX$.
Below, we let $e_1,\dots,e_r$ be the standard basis for $\Z^r$.

\begin{definition}
\label{def:fantastack}
Let $\Sigma$ be a fan on a lattice $N$, and let $\nu\colon \Z^r\to N$ be a homomorphism with finite cokernel so that every ray of $\Sigma$ contains some $v_i:=\nu(e_i)$ and every $v_i$ lies in the support of $\Sigma$. For a cone $\sigma\in \Sigma$, let $\widetilde\sigma=\mathrm{cone}(\{e_i|v_i\in \sigma\})$. We define the fan $\widetilde\Sigma$ on $\Z^r$ as the fan generated by all the $\widetilde\sigma$. We define
\[
\cF_{\Sigma,\nu} := \cX_{\widetilde\Sigma,\nu}.
\]
Any toric stack isomorphic to some $\cF_{\Sigma,\nu}$ is called a \emph{fantastack}. When $\Sigma$ is the fan generated by the faces of a cone $\sigma$, we denote $\cF_{\Sigma,\nu}$ by $\cF_{\sigma,\nu}$.
\end{definition}

\begin{remark}
By \cite[Example 6.24]{GeraschenkoSatriano1}, cf.~\cite[Theorem 5.5]{Satriano2013}, the natural map
\[
\cF_{\Sigma,\nu}\longrightarrow X_\Sigma
\]
is a good moduli space morphism. Furthermore, fantastacks have moduli interpretations in terms of line bundles and sections, analogous to the moduli interpretation for $\mathbb{P}^n$, see \cite[Section 7]{GeraschenkoSatriano1}.
\end{remark}

The next two results will be useful later on.

\begin{proposition}
\label{elementofFaddstogetelementofP}
Let $\sigma$ be a pointed full-dimensional cone and suppose the good moduli space map $\pi\colon\cF_{\sigma,\nu}\to X_\sigma$ is an isomorphism over the torus $T$ of $X_\sigma$. Then for any $f \in F:=\widetilde{\sigma}^\vee\cap\widetilde{N}^*$, there exists some $f' \in F$ such that
\[
	f + f' \in P:=\sigma^\vee\cap N^*.
\]
In particular, if $\psi\colon F\to\N\cup\{\infty\}$ is a morphism of monoids and $\psi(P)\subset\N$, then $\psi(F)\subset\N$.
\end{proposition}
\begin{proof}
Let $v_i=\beta(e_i)$ for $1\leq i\leq r$. Since $\pi$ is an isomorphism over $T$, each $v_i\neq0$. As $\sigma$ is pointed, there exists some $p \in P$ such that $\langle v_i, p \rangle  > 0$ for all $i$. Viewing $p$ as an element of $F$ via the inclusion $P\subset F$, we have $\langle e_i, p \rangle > 0$.

Let $f_1, \dots, f_r$ be the basis of $\widetilde{M}$ dual to $e_1, \dots, e_r$. Since the $f_i$ are generators of $F$, it suffices to prove the proposition for each $f_i$. Note that
\[
	  \langle e_1, p \rangle f_1 + \dots + \langle e_r, p \rangle f_r = p \in P.
\]
Since $\langle e_i, p \rangle > 0$, we see
\[
f'_i:=(\langle e_i,p\rangle-1)f_i+\sum_{j\neq i}\langle e_j,p\rangle f_j\in F
\]
and that $f_i+f'_i\in P$.
\end{proof}

\begin{proposition}
\label{fiberabovewisfiniteset}
Keep the notation and hypotheses of \autoref{elementofFaddstogetelementofP} and let $\beta\colon\widetilde\sigma\cap\widetilde{N}\to\sigma\cap N$ be the induced map. If $w \in \sigma \cap N$, then $\beta^{-1}(w)$ is a finite set.
\end{proposition}

\begin{proof}
Let $f_1, \dots, f_r$ be the minimal generators of the monoid $F$. By \autoref{elementofFaddstogetelementofP}, there exist $f'_1, \dots, f'_r$ such that $f_i + f'_i \in P$ for all $i \in \{1, \dots, r\}$. For any $\widetilde{w} \in \beta^{-1}(w)$,
\[
	\langle \widetilde{w}, f_i \rangle \leq \langle \widetilde{w}, f_i + f'_i \rangle = \langle w, f_i+f'_i\rangle,
\]
so there are only finitely many possible values for each $\langle \widetilde{w}, f_i \rangle$. Thus $\beta^{-1}(w)$ is a finite set.
\end{proof}

We end this section by discussing canonical stacks as defined in \cite[Section 5]{GeraschenkoSatriano1}.

\begin{definition}
\label{def:canonical-stack}
If $\Sigma$ is a fan on a lattice $N$, let $v_1,\dots,v_r\in N$ be the first lattice points on the rays of $\Sigma$, let $\nu\colon\Z^r\to N$ be the map $\nu(e_i):=v_i$, and let $\widetilde\Sigma$ be as in \autoref{def:fantastack}. If $N'$ is a direct complement of the support of $\Sigma$ and $\nu'\colon\Z^r\oplus N'\to N$ is given by $\nu'(v,n')=\nu(v)+n'$, then $\cX_{\widetilde\Sigma,\nu'}$ is the \emph{canonical stack} of $X_\Sigma$.
\end{definition}

\begin{remark}
With notation as in \autoref{def:canonical-stack}, if the support of $\Sigma$ is $N$, the canonical stack of $X_\Sigma$ is the fantastack $\cF_{\widetilde{\Sigma},\nu}$.
\end{remark}


The next proposition, which is straightforward from the definition, says that canonical stacks are compatible with open immersions. This will be useful for us, as this proposition will allow us to reduce most of our work to the case of affine toric varieties defined by a $d$-dimensional cone in $N_\R$.

\begin{proposition}
\label{canonicalFantastackOpenInvariantSubvariet}
Let $\Sigma$ be a fan consisting of pointed rational cones in $N_\R$, let $\sigma$ be a cone in $\Sigma$, let $X(\Sigma)$ and $X(\sigma)$ be the $T$-toric varieties associated to $\Sigma$ and $\sigma$, respectively, and let $\iota: X(\sigma) \hookrightarrow X(\Sigma)$ be the open inclusion. If $\cX(\Sigma)$ and $\cX(\sigma)$ are the canonical stacks over $X(\Sigma)$ and $X(\sigma)$, respectively, and $\pi(\Sigma): \cX(\Sigma) \to X(\Sigma)$ and $\pi(\sigma): \cX(\sigma) \to X(\sigma)$ are the canonical maps, then there exists a map $\cX(\sigma) \to \cX(\Sigma)$ such that
\begin{center}
\begin{tikzcd}
\cX(\sigma) \arrow{r}{\pi(\sigma)} \arrow[d] & X(\sigma) \arrow{d}{\iota}\\
\cX(\Sigma) \arrow{r}{\pi(\Sigma)} & X(\Sigma) 
\end{tikzcd}
\end{center}
is a fiber product diagram.
\end{proposition}

For the remainder of this subsection, let $\sigma$ be a $d$-dimensional pointed rational cone in $N_\R$, let $X$ be the affine $T$-toric variety associated to $\sigma$, let $\cX$ be the canonical stack over $X$, and let $\pi: \cX \to X$ be the canonical map. At points later in this paper, we will refer to the list \autoref{notationforcanonicalfantastackforfulldimensionalcone} when we want to set the following notation, and we also set that notation for the remainder of this subsection.

\begin{notation}\label{notationforcanonicalfantastackforfulldimensionalcone}
\begin{itemize}

\item Let $M = N^*$.

\item Let $\widetilde{N}$ be the free abelian group with generators indexed by the rays of $\sigma$.

\item Let $\widetilde{M} = \widetilde{N}^*$.

\item Let $\langle \cdot, \cdot \rangle$ denote both pairings $N \otimes_\Z M \to \Z$ and $\widetilde{N} \otimes_\Z \widetilde{M} \to \Z$.

\item Let $\widetilde{T} = \Spec(k[\widetilde{M}])$ be the algebraic torus with co-character lattice $\widetilde{N}$.

\item Let $\widetilde{\sigma}$ be the positive orthant of $\widetilde{N}_\R$, i.e., $\widetilde{\sigma}$ is the positive span of those generators of $\widetilde{N}$ that are indexed by the rays of $\sigma$.

\item Let $\widetilde{X}$ be the affine $\widetilde{T}$-toric variety associated to $\widetilde{\sigma}$.

\item Let $\beta: \widetilde{\sigma} \cap \widetilde{N} \to \sigma \cap N$ be the monoid map taking the generator of $\widetilde{N}$ indexed by a ray of $\sigma$ to the first lattice point of that ray.

\item Let $\widetilde{\pi}: \widetilde{X} \to X$ be toric map associated to $\beta^\gp: \widetilde{N} \to N$.

\item Let $P = \sigma^\vee \cap M$. Note that $X = \Spec(k[P])$.

\item Let $F = \widetilde{\sigma}^\vee \cap \widetilde{M}$. Note that $\widetilde{X} = \Spec(k[F])$.

\item Identify $P$ with its image under the injection $P \hookrightarrow F$ given by dualizing $\beta$. Note that $P \hookrightarrow F$ is injective because $\sigma$ is full dimensional.

\item Let $A = F^\gp/P^\gp = \widetilde{M}/ M$.

\item Let $G = \Spec(k[A])$ be the kernel of the algebraic group homomorphism $\widetilde{T} \to T$ obtained by restricting $\widetilde{\pi}$, and let $G$ act on $\widetilde{X}$ by restricting the toric action of $\widetilde{T}$ on $\widetilde{X}$.

\end{itemize}
\end{notation}

By definition the canonical stack $\cX$ is equal to the stack quotient $[\widetilde{X}/G]$, and the morphism $\widetilde{\pi}: \widetilde{X} \to X$ is the composition $\widetilde{X} \to [\widetilde{X} / G] =  \cX \xrightarrow{\pi} X$. 

\vspace{1em}

We note that because our focus is on singular varieties instead of on stacks, we simplify our exposition by focusing on canonical stacks over toric varieties instead of all fantastacks. The expositional advantage is that canonical stacks depend only on the toric variety and not on additional data as is the case for other fantastacks. We end this section with the next two remarks, which explain why we have not lost any generality by making this expositional simplification, as well as discuss a generalization of \autoref{mainTheoremStackMeasureAndGorensteinMeasure}.

\begin{remark}
\label{remarkCrepantDefinition}
For \autoref{mainTheoremStackMeasureAndGorensteinMeasure}, it is sufficient to consider canonical stacks as these are precisely the fantastacks satisfying the hypotheses of \autoref{conjectureGorensteinMeasureAndStackMeasure}. Nonetheless, we note that with only superficial modifications to our techniques, one can actually prove a more general statement than \autoref{mainTheoremStackMeasureAndGorensteinMeasure}, which we explain here.

With notation as in \autoref{def:canonical-stack}, let $\cX=\cF_{\Sigma,\nu}$ be a fantastack. Assume $X=X_\Sigma$ is $\Q$-Gorenstein so for each maximal cone $\sigma\in\Sigma$, there exists $q_\sigma\in N^*$ and $m_\sigma\in\Z_{>0}$ such that the set $$\mathcal{H}_\sigma:=\{v\in \sigma\cap N\mid\langle q_\sigma,v\rangle=m_\sigma\}$$ contains the first lattice point of every ray of $\sigma$. We say the good moduli space map $\pi\colon\cX\to X$ is \emph{combinatorially crepant} if $\nu(e_i)\in\bigcup_\sigma\mathcal{H}_\sigma$ for every $i\in\{1,\dots,r\}$.

For example, the canonical stack is combinatorially crepant over $X$. Since \autoref{QGorensteinsumofgeneratorsisinvariant} holds for all fantastacks that are combinatorially crepant over their good moduli space, the conclusions of \autoref{conjectureGorensteinMeasureAndStackMeasure} hold for any fantastack that is combinatorially crepant over its good moduli space.
\end{remark}

\begin{remark}
\label{remarkOtherFantastacks}
If $\cF_{\Sigma, \nu}$ is a fantastack over $X$, then $\cF_{\Sigma, \nu} \to X$ is an isomorphism over a nonempty open subset of $X$ if and only if $\nu$ does not send any standard basis vector to 0. Since \autoref{elementofFaddstogetelementofP} holds for every fantastack satisfying the hypotheses of \autoref{conjectureSpecialStabilizersLiftArcs}, our proofs show that \autoref{mainTheoremLiftingToUntwistedArcsSpecialStabilizers} holds for any fantastack as well.
\end{remark}

\section{Motivic integration for stacks}\label{sectionMotivicIntegrationQuotientStacks}

For the remainder of this paper, by a \emph{quotient stack} over $k$, we will mean an Artin stack over $k$ that is isomorphic to the stack quotient of a $k$-scheme by the action of a linear algebraic group over $k$. 


\begin{remark}
\label{quotientStackIsQuotientBySpecialGroup}
Let $G$ be a linear algebraic group over $k$ acting on a $k$-scheme $\widetilde{X}$, and let $G \hookrightarrow G'$ be an inclusion of $G$ as a closed subgroup of a linear algebraic group $G'$ over $k$. Then we have an isomorphism
\[
	[\widetilde{X} / G] \cong [(\widetilde{X} \times^G G')/G'],
\]
where $\widetilde{X} \times^G G'$ is the $k$-scheme with $G'$-action obtained from $\widetilde{X}$ by pushout along $G \hookrightarrow G'$. Thus any quotient stack is isomorphic to a stack quotient of a scheme by a general linear group, which in particular, is a special group.
\end{remark}

In this section, we define a notion of motivic integration for quotient stacks. On the one hand, our definition is straightforward: it is more-or-less identical to motivic integration for schemes, but in various places, we need to replace notions for schemes with the obvious analogs for Artin stacks; in particular, our motivic integration for quotient stacks does not depend on a choice of presentation for the stack as a quotient. On the other hand, our notion allows explicit computations in terms of motivic integration for schemes, as long as one first writes the stack as a stack quotient of a scheme by a \emph{special} group.

\begin{definition}
Let $\cX$ be an Artin stack over $k$, and let $n \in \N$. The $n$th \emph{jet stack} of $\cX$, denoted $\sL_n(\cX)$, is the Weil restriction of $\cX \otimes_k k[t]/(t^{n+1})$ with respect to the morphism $\Spec(k[t]/(t^{n+1})) \to \Spec(k)$.
\end{definition}

\begin{remark}
Each jet stack $\sL_n(\cX)$ is an Artin stack by \cite[Theorem 3.7(iii)]{Rydh2011}.
\end{remark}

\begin{remark}
Each jet stack $\sL_n(\cX)$ has affine (geometric) stabilizers by the following argument. Let $y: \Spec(k') \to \sL_n(\cX)$ be a geometric point corresponding to $\psi_n: \Spec(k'[t]/(t^{n+1})) \to \cX$. Because $\cX$ has affine (geometric) stabilizers, the reduction of the stabilizer of $\psi_n$ is affine, so the stabilizer of $\psi_n$ is affine. Thus the stabilizer of $y$, which is the Weil restriction of the stabilizer of $\psi_n$, is affine.
\end{remark}

The morphisms $k[t]/(t^{n+1}) \to k[t]/(t^{m+1})$, when $n \geq m$, induce \emph{truncation morphisms} $\theta^n_m: \sL_n(\cX) \to \sL_m(\cX)$ for any Artin stack $\cX$ over $k$. Like in the case of schemes, we use these truncation morphisms to define arcs of $\cX$ and a stack parametrizing them.

\begin{definition}
Let $\cX$ be an Artin stack over $k$. The \emph{arc stack} of $\cX$ is the inverse limit $\sL(\cX) = \varprojlim_n \sL_n(\cX)$, where the inverse limit is taken with respect to the truncation morphisms $\theta^n_m: \sL_n(\cX) \to \sL_m(\cX)$.
\end{definition}

\begin{remark}
The name arc \emph{stack} is justified by the fact that $\sL(\cX)$ is indeed a stack. See for example \cite[Proposition 2.1.9]{Talpo}. Since $\sL(\cX)$ is a stack as opposed to an Artin stack, we use the symbol $|\sL(\cX)|$ to denote equivalence classes of points but do not define a topology on this set.
\end{remark}

\begin{remark}
Let $\cX$ be an Artin stack over $k$, and let $k'$ be a field extension of $k$. The truncation morphism $k'\llbracket t \rrbracket \to k'[t]/(t^{n+1})$ induces a functor $\cX(k'\llbracket t \rrbracket) \to \cX(k'[t]/(t^{n+1})) = \sL_n(\cX)(k')$ for each $n \in \N$, and these functors induce a functor $\cX(k'\llbracket t \rrbracket) \to \sL(\cX)(k')$. Since $\cX$ is an Artin stack, the functor $\cX(k'\llbracket t \rrbracket) \to \sL(\cX)(k')$ is an equivalence of categories, e.g.~by Artin's criterion for algebraicity. Throughout this paper, we will often implicitly make this identification.
\end{remark}

We will let each $\theta_n: \sL(\cX) \to \sL_n(\cX)$ denote the canonical morphism, and we will also call these \emph{truncation morphisms}.

We will eventually define a notion of measurable subsets of $|\sL(\cX)|$ and a motivic measure $\mu_\cX$ that assigns an element of $\widehat{\sM}_k$ to each of these measurable subsets. We begin with an important special case of measurable subsets. Note that when $\cX$ is finite type over $k$, so is each $\sL_n(\cX)$ by \cite[Proposition 3.8(xv)]{Rydh2011}.

\begin{definition}
Let $\cX$ be a finite type Artin stack over $k$, and let $\cC \subset |\sL(\cX)|$. We call the subset $\cC$ a \emph{cylinder} if there exists some $n \in \N$ and a constructible subset $\cC_n \subset |\sL_n(\cX)|$ such that $\cC = (\theta_n)^{-1}(\cC_n)$. 
\end{definition}

The next theorem, which we will prove later in this section, allows us to define a motivic integration for quotient stacks that is closely related to motivic integration for schemes.

\begin{theorem}
\label{theoremMainMotivicIntegrationForQuotientStacks}
Let $\cX$ be an equidimensional finite type quotient stack over $k$, and let $\cC \subset |\sL(\cX)|$ be a cylinder. Then the set $\theta_n(\cC) \subset |\sL_n(\cX)|$ is constructible for each $n \in \N$, and the sequence 
\[
	\{\e(\theta_n(\cC)) \bL^{-(n+1)\dim\cX}\}_{n \in \N} \subset \widehat{\sM}_k
\] 
converges.

Furthermore, suppose that $G$ is a special group over $k$ and $\widetilde{X}$ is a $k$-scheme with $G$-action such that there exists an isomorphism $[\widetilde{X} / G] \xrightarrow{\sim} \cX$, let $\rho: \widetilde{X} \to \cX$ be the composition of the quotient map $\widetilde{X} \to [\widetilde{X} / G]$ with the isomorphism $[\widetilde{X} / G] \xrightarrow{\sim} \cX$, and let $\widetilde{C} = \sL(\rho)^{-1}(\cC)$. Then $\widetilde{C} \subset \sL(\widetilde{X})$ is a cylinder, and
\[
	\lim_{n \to \infty} \e(\theta_n(\cC))\bL^{-(n+1)\dim\cX} = \mu_{\widetilde{X}} (\widetilde{C}) \e(G)^{-1} \bL^{\dim G} \in \widehat{\sM}_k.
\]
\end{theorem}

\begin{remark}
\label{remarkSpecialCoverOfIrreducibleStackIsIrreducible}
Let $G$ and $\widetilde{X}$ be as in the statement of \autoref{theoremMainMotivicIntegrationForQuotientStacks}. Since $[\widetilde{X}/G]$ is equidimensional and $G$ is geometrically irreducible, $\widetilde{X}$ is equidimensional as well, and hence $\mu_{\widetilde{X}}$ is well defined.
\end{remark}

Before we prove \autoref{theoremMainMotivicIntegrationForQuotientStacks}, we will discuss some useful consequences. First, we can define the motivic measure $\mu_\cX$ on cylinders.

\begin{definition}
\label{definitionMotivicMeasureCylinder}
Let $\cX$ be an equidimensional finite type quotient stack over $k$, and let $\cC \subset |\sL(\cX)|$ be a cylinder. The \emph{motivic measure} of $\cC$ is
\[
	\mu_\cX(\cC) = \lim_{n \to \infty} \e(\theta_n(\cC))\bL^{-(n+1)\dim\cX} \in \widehat{\sM}_k.
\]
\end{definition}

\begin{remark}
Let $\cX$ be an equidimensional smooth Artin (not-necessarily-quotient) stack over $k$ and let $\cC \subset |\sL(\cX)|$ be a cylinder. One can verify that $\theta_n(\cC) \subset |\sL_n(\cX)|$ is constructible for each $n \in \N$ and that $\{\e(\theta_n(\cC)) \bL^{-(n+1)\dim\cX}\}_{n \in \N}$ stabilizes for sufficiently large $n$, so \autoref{definitionMotivicMeasureCylinder} also makes sense here. Although this is not used for the main results of this paper, our main conjectures are stated in the generality, so we provide the argument for completeness in \autoref{subsectionMotivicIntegrationForSmoothStacksViaTheCotangentComplex}.
\end{remark}

We now define measurable subsets analogously to the case of schemes.

\begin{definition}
Let $\cX$ be an equidimensional finite type quotient stack over $k$, let $\cC \subset |\sL(\cX)|$, let $\varepsilon \in \R_{>0}$, let $I$ be a set, let $\cC^{(0)} \subset |\sL(\cX)|$ be a cylinder, and let $\{\cC^{(i)}\}_{i \in I}$ be a collection of cylinders in $|\sL(\cX)|$.

We say $(\cC^{(0)}, (\cC^{(i)})_{i \in I})$ is a \emph{cylindrical $\varepsilon$-approximation} of $\cC$ if
\[
	(\cC \cup \cC^{(0)}) \setminus (\cC \cap \cC^{(0)}) \subset \bigcup_{i \in I} \cC^{(i)}
\]
and for all $i \in I$,
\[
	\Vert\mu_{\cX}(\cC^{(i)})\Vert < \varepsilon.
\]
\end{definition}

\begin{definition}
Let $\cX$ be an equidimensional finite type quotient stack over $k$, and let $\cC \subset |\sL(\cX)|$. We say that $\cC$ is \emph{measurable} if for any $\varepsilon \in \R_{>0}$, there exists a cylindrical $\varepsilon$-approximation of $\cC$.
\end{definition}

\begin{remark}
\label{remarkCylinderIsMeasurable}
Let $\cX$ be an equidimensional finite type quotient stack over $k$, and let $\cC \subset |\sL(\cX)|$ be a cylinder. Then for any $\varepsilon \in \R_{>0}$, we have $(\cC, \emptyset)$ is a cylindrical $\varepsilon$-approximation of $\cC$. In particular, $\cC$ is measurable.
\end{remark}

We now see that \autoref{theoremMainMotivicIntegrationForQuotientStacks} allows us to extend $\mu_\cX$ to measurable subsets.

\begin{corollary}
\label{corollaryMotivicMeasureOfMeasurableSet}
Let $\cX$ be an equidimensional finite type quotient stack over $k$, and let $\cC \subset |\sL(\cX)|$ be a measurable subset. Then there exists a unique $\mu_{\cX}(\cC) \in \widehat{\sM}_k$ such that for any $\varepsilon \in \R_{>0}$ and any cylindrical $\varepsilon$-approximation $(\cC^{(0)}, (\cC^{(i)})_{i \in I})$ of $\cC$,
\[
	\Vert\mu_\cX(\cC) - \mu_\cX(\cC^{(0)})\Vert < \varepsilon.
\]

Furthermore, suppose that $G, \widetilde{X}, \rho$ are as in the statement of \autoref{theoremMainMotivicIntegrationForQuotientStacks} and $\widetilde{C}=\sL(\rho)^{-1}(\cC)$. Then $\widetilde{C} \subset \sL(\widetilde{X})$ is measurable, and
\[
	\mu_{\cX}(\cC) = \mu_{\widetilde{X}} (\widetilde{C}) \e(G)^{-1} \bL^{\dim G} \in \widehat{\sM}_k.
\]
\end{corollary}

\begin{proof}
Let $G, \widetilde{X}, \rho, \widetilde{C}$ be as in the second part above. For any $\varepsilon \in \R_{>0}$ and any cylindrical $\varepsilon$-approximation $(\cC^{(0)}, (\cC^{(i)})_{i \in I})$ of $\cC$, \autoref{theoremMainMotivicIntegrationForQuotientStacks} implies that $(\sL(\rho)^{-1}(\cC^{(0)}), (\sL(\rho)^{-1}(\cC^{(i)}))_{i \in I})$ is a cylindrical $\varepsilon \Vert\e(G) \bL^{-\dim G}\Vert$-approximation of $\widetilde{C}$. Thus $\widetilde{C}$ is measurable, and for any cylindrical $\varepsilon$-approximation $(\cC^{(0)}, (\cC^{(i)})_{i \in I})$ of $\cC$,
\begin{align*}
	\Vert\mu_{\widetilde{X}} (\widetilde{C})& \e(G)^{-1} \bL^{\dim G} - \mu_\cX(\cC^{(0)})\Vert\\
	&= \Vert\mu_{\widetilde{X}} (\widetilde{C}) \e(G)^{-1} \bL^{\dim G} - \mu_{\widetilde{X}} (\sL(\rho)^{-1}(\cC^{(0)})) \e(G)^{-1} \bL^{\dim G}\Vert\\
	&\leq \Vert \e(G)^{-1} \bL^{\dim G} \Vert  \Vert \mu_{\widetilde{X}}(\widetilde{C}) - \mu_{\widetilde{X}}(\sL(\rho)^{-1}(\cC^{(0)})) \Vert\\
	&< \varepsilon \Vert \e(G)^{-1} \bL^{\dim G} \Vert \Vert\e(G) \bL^{-\dim G}\Vert,
\end{align*}
where the first equality follows from \autoref{theoremMainMotivicIntegrationForQuotientStacks}. Once $\mu_\cX(\cC)$ is shown to exist, this chain of inequalities proves $\mu_{\cX}(\cC) = \mu_{\widetilde{X}} (\widetilde{C}) \e(G)^{-1} \bL^{\dim G}$. To show the existence of $\mu_\cX(\cC)$, it suffices by \autoref{quotientStackIsQuotientBySpecialGroup} to assume $G$ is a general linear group, so the above chain of inequalities and \autoref{lemmaSizeOfClassOfGeneralLinearGroup} finish the proof.
\end{proof}

\begin{lemma}
\label{lemmaSizeOfClassOfGeneralLinearGroup}
Let $G$ be a general linear group over $k$. Then
\[
	\Vert \e(G)^{-1} \bL^{\dim G} \Vert \Vert\e(G) \bL^{-\dim G}\Vert = 1.
\]
\end{lemma}

\begin{proof}
Using Euler-Poincar\'{e} polynomials, it is straightforward to check (see for example the proof of \cite[Chapter 2 Lemma 4.1.3]{ChambertLoirNicaiseSebag}) that if $n_0 \in \Z$ and $\{c_n\}_{n \geq n_0}$ is a sequence of integers with $c_{n_0} \neq 0$, then
\[
	\Vert \sum_{n \geq n_0} c_n \bL^{-n} \Vert = \exp(-n_0).
\]
The lemma then follows from the fact that $\e(G)$ is a polynomial in $\bL$ (see for example the proof of \cite[Lemma 4.6]{Joyce}).
\end{proof}

\begin{definition}
\label{definitionMotivicMeasureOfMeasurableSubset}
Let $\cX$ be an equidimensional finite type quotient stack over $k$, and let $\cC \subset |\sL(\cX)|$ be a measurable subset. The \emph{motivic measure} of $\cC$ is defined to be $\mu_\cX(\cC) \in \widehat{\sM}_k$ as in the statement of \autoref{corollaryMotivicMeasureOfMeasurableSet}.
\end{definition}

\begin{remark}
\autoref{remarkCylinderIsMeasurable} implies that \autoref{definitionMotivicMeasureOfMeasurableSubset} generalizes \autoref{definitionMotivicMeasureCylinder}.
\end{remark}

In the next two subsections, we will prove \autoref{theoremMainMotivicIntegrationForQuotientStacks}.

\subsection{Jet schemes of quotient stacks}

In this subsection, we describe the jet schemes of a stack quotient as stack quotients themselves. This is the first step in providing the relationship between motivic integration for quotient stacks and motivic integration for schemes. This description, \autoref{JetSchemeOfQuotientStackIsQuotientOfJetSchemes}, is a special case of the next proposition, which describes the Weil restriction of a stack quotient.

If $S'$ and $S$ are schemes and $S' \to S$ is a finite flat morphism of finite presentation, we will let $\sR_{S'/S}$ denote the functor taking each stack over $S'$ to its Weil restriction with respect to $S' \to S$, and we note that if $\cX$ is an Artin stack over $S'$, then $\sR_{S'/S}(\cX)$ is an Artin stack over $S$ \cite[Theorem 3.7(iii)]{Rydh2011}.

\begin{proposition}
\label{WeilRestrictionOfQuotientStackIsQuotientOfWeilRestrictions}
Let $S'$ and $S$ be schemes and $S' \to S$ be a finite flat morphism of finite presentation. If $\widetilde{X}'$ is an $S'$-scheme with an action by a linear algebraic group $G'$ over $S'$, then there exists an isomorphism
\[
\sR_{S'/S}([\widetilde{X}'/G'])\xrightarrow{\sim} [\sR_{S'/S}(\widetilde{X}')/\sR_{S'/S}(G')]
\]
such that
\begin{center}
\begin{tikzcd}
\sR_{S'/S}(\widetilde{X}') \arrow{rrr}{\sR_{S'/S}(\widetilde{X}' \to [\widetilde{X}'/G'])} \arrow[rrrd] &&& \sR_{S'/S}([\widetilde{X}'/G']) \arrow[Isom,d]\\
&&& \phantom{} [\sR_{S'/S}(\widetilde{X}') / \sR_{S'/S}(G')]
\end{tikzcd}
\end{center}
\noindent commutes.
\end{proposition}

\begin{remark}
In the statement of \autoref{WeilRestrictionOfQuotientStackIsQuotientOfWeilRestrictions}, the action of $\sR_{S'/S}(G')$ on $\sR_{S'/S}(\widetilde{X}')$ is obtained by applying $\sR_{S'/S}$ to the map $G' \times_{S'} \widetilde{X}' \to \widetilde{X}'$ defining the action of $G'$ on $\widetilde{X}'$.
\end{remark}

\begin{proof}
We let $\cX'=[\widetilde{X}'/G']$, $\rho': \widetilde{X}' \to \cX'$ be the quotient map, $\cX=\sR_{S'/S}(\cX')$, $\widetilde{X}=\sR_{S'/S}(\widetilde{X}')$, and $G=\sR_{S'/S}(G')$. Since $\rho'\colon \widetilde{X}'\to\cX'$ is a smooth cover, $\sR_{S'/S}(\rho')\colon \widetilde{X}\to\cX$ is as well by \cite[Proposition 3.5(v)]{Rydh2011}. Since $\rho'$ is a $G'$-torsor, the natural map $G'\times_{S'} \widetilde{X}'\to \widetilde{X}'\times_{\cX'} \widetilde{X}'$ induced by the $G'$-action $G' \times_{S'} \widetilde{X}' \to \widetilde{X}'$ is an isomorphism, and applying Weil restriction, we see the map $G\times_S \widetilde{X}\to \widetilde{X}\times_{\cX} \widetilde{X}$ induced by the $G$-action $G \times_S \widetilde{X} \to \widetilde{X}$ is an isomorphism as well. Thus, $\sR_{S'/S}(\rho')\colon\widetilde{X}\to\cX$ is a $G$-torsor, thereby inducing an isomorphism $\cX\xrightarrow{\sim}[\widetilde{X}/G]$ which makes the diagram in the statement of the proposition commute.
\end{proof}

By the definition of jet stacks, the following is a special case of \autoref{WeilRestrictionOfQuotientStackIsQuotientOfWeilRestrictions}.

\begin{corollary}
\label{JetSchemeOfQuotientStackIsQuotientOfJetSchemes}
Let $G$ be a linear algebraic group over $k$ acting on a $k$-scheme $\widetilde{X}$, and let $n \in \N$. There exists an isomorphism 
\[
	\sL_n([\widetilde{X}/G]) \xrightarrow{\sim} [\sL_n(\widetilde{X})/\sL_n(G)],
\]
such that
\begin{center}
\begin{tikzcd}
\sL_n(\widetilde{X}) \arrow{rrr}{\sL_n(\widetilde{X} \to [\widetilde{X}/G])} \arrow[rrrd] &&& \sL_n([\widetilde{X}/G]) \arrow[Isom,d]\\
&&& \phantom{} [\sL_n(\widetilde{X}) / \sL_n(G)]
\end{tikzcd}
\end{center}
\end{corollary}
\noindent commutes.

\begin{remark}
In the statement of \autoref{JetSchemeOfQuotientStackIsQuotientOfJetSchemes}, the action of $\sL_n(G)$ on $\sL_n(\widetilde{X})$ is obtained by applying $\sL_n$ to the map $G \times_{k} \widetilde{X} \to \widetilde{X}$ defining the $G$-action on $\widetilde{X}$.
\end{remark}

\subsection{Truncation morphisms and quotient stacks}

\begin{lemma}
\label{lemmaTruncationMorphismsAndSmoothMap}
Let $\cX$ be an Artin stack over $k$, let $\widetilde{X}$ be a scheme over $k$, and let $\rho: \widetilde{X} \to \cX$ be a smooth covering. Let $\cC \subset |\sL(\cX)|$, and set $\widetilde{C} = \sL(\rho)^{-1}(\cC) \subset \sL(\widetilde{X})$. Then for all $n \in \N$,
\[
	\sL_n(\rho)^{-1}(\theta_n(\cC)) = \theta_n(\widetilde{C}).
\]
\end{lemma}

\begin{proof}
Let $n \in \N$. Clearly $\theta_n(\widetilde{C}) \subset \sL_n(\rho)^{-1}(\theta_n(\cC))$.

To prove the opposite inclusion, let $k'$ be a field extension of $k$, and let $\widetilde{\psi}_n \in \sL_n(\widetilde{X})(k')$ and $\psi \in \sL(\cX)(k')$ be such that the class of $\psi$ in $|\sL(\cX)|$ is contained in $\cC$ and $\sL_n(\rho)(\widetilde{\psi}_n) \cong \theta_n(\psi)$. We must show $\widetilde{\psi}_n\in\theta_n(\widetilde{C})$. Since $\rho$ is smooth, by the infinitesimal lifting criterion, we have a dotted arrow filling in the following diagram
\[
\xymatrix{
\Spec k'[t]/(t^{n+1})\ar[r]^-{\widetilde{\psi}_n}\ar[d] & \widetilde{X}\ar[d]^-{\rho}\\
\Spec k'[[t]]\ar[r]^-{\psi}\ar@{-->}[ur]^-{\widetilde{\psi}} & \cX
}
\]
Then 
$\widetilde{\psi}\in\widetilde{C}$, so 
$\widetilde{\psi}_n\in\theta_n(\widetilde{C})$.
\end{proof}

We may now prove the next proposition, which by \autoref{quotientStackIsQuotientBySpecialGroup} and \autoref{remarkSpecialCoverOfIrreducibleStackIsIrreducible}, implies \autoref{theoremMainMotivicIntegrationForQuotientStacks}.

\begin{proposition}\label{lastPropositionForShowingMotivicMeasureWellDefinedForCylinders}
Let $G$ be a special group over $k$, let $\widetilde{X}$ be an equidimensional finite type scheme over $k$ with $G$-action, let $\cX = [\widetilde{X}/G]$, let $\rho: \widetilde{X} \to \cX$ be the quotient map, let $\cC \subset |\sL(\cX)|$ be a cylinder, and let $\widetilde{C} = \sL(\rho)^{-1}(\cC)$. Then $\widetilde{C} \subset \sL(\widetilde{X})$ is a cylinder, the set $\theta_n(\cC) \subset |\sL_n(\cX)|$ is constructible for each $n \in \N$, and the sequence
\[
	\{\e(\theta_n(\cC)) \bL^{-(n+1)\dim\cX}\}_{n \in \N} \subset \widehat{\sM}_k
\]
converges to
\[
	\mu_{\widetilde{X}} (\widetilde{C}) \e(G)^{-1} \bL^{\dim G} \in \widehat{\sM}_k.
\]
\end{proposition}

\begin{remark}
In the statement of \autoref{lastPropositionForShowingMotivicMeasureWellDefinedForCylinders}, because $G$ is irreducible, the irreducible components of $\widetilde{X}$ are $G$-invariant, so $\cX$ is equidimensional.
\end{remark}

\begin{proof}
We first show that $\widetilde{C} \subset \sL(\widetilde{X})$ is a cylinder. Because $\cC$ is a cylinder, there exists some $n \in \N$ and some constructible subset $\cC_n \subset |\sL_n(\cX)|$ such that $\cC = (\theta_n)^{-1}(\cC_n)$. Then $\widetilde{C} = (\theta_n)^{-1}(\sL_n(\rho)^{-1}(\cC_n))$ is a cylinder.

Now we will show that for all $n \in \N$, the set $\theta_n(\cC)$ is a constructible subset of $\sL_n(\cX)$. Each $\theta_n(\widetilde{C})$ is a constructible subset of $\sL_n(\widetilde{X})$. Therefore each $\theta_n(\cC) \subset |\sL_n(\cX)|$ is constructible by Chevalley's Theorem for Artin stacks \cite[Theorem 5.2]{HallRydh2017}, \autoref{JetSchemeOfQuotientStackIsQuotientOfJetSchemes}, and \autoref{lemmaTruncationMorphismsAndSmoothMap}.

Then since $G$ is a special group, $\sL_n(G)$ is as well by \autoref{jetSchemeOfSpecialGroupIsSpecial}. Then \autoref{JetSchemeOfQuotientStackIsQuotientOfJetSchemes} and \autoref{lemmaTruncationMorphismsAndSmoothMap} imply that for each $n \in \N$, 
\[
	\e(\theta_n(\cC)) = \e(\theta_n(\widetilde{C}))\e(\sL_n(G))^{-1} = \e(\theta_n(\widetilde{C})) \e(G)^{-1} \bL^{-n\dim G},
\]
where the second equality holds because $G$ is smooth. Therefore
\begin{align*}
	\mu_{\widetilde{X}}(\widetilde{C})\e(G)^{-1}\bL^{\dim G} &= \lim_{n \to \infty} \e(\theta_n(\widetilde{C})) \e(G)^{-1} \bL^{\dim G -(n+1)\dim \widetilde{X}}\\
	&= \lim_{n \to \infty} \e(\theta_n(\cC)) \bL^{-(n+1)\dim\cX}.\qedhere
\end{align*}
\end{proof}

\subsection{Properties of motivic integration for quotient stacks}

We now state some basic properties of motivic integration for quotient stacks. We will use these properties later in this paper.

\begin{proposition}
\label{measureOfCountableDisjointUnion}
Let $\cX$ be an equidimensional finite type quotient stack over $k$, let $\{\cC^{(i)}\}_{i \in \N}$ be a sequence of pairwise disjoint measurable subsets of $|\sL(\cX)|$, and let $\cC = \bigcup_{i= 0}^\infty \cC^{(i)}$. If $\lim_{i \to \infty} \mu_{\cX}(\cC^{(i)}) = 0$, then $\cC$ is measurable and
\[
	\mu_{\cX}(\cC) = \sum_{i = 0}^\infty \mu_{\cX}(\cC^{(i)}).
\]
\end{proposition}

\begin{proof}
The set $\cC$ is measurable by the exact same proof used for the analogous statement for schemes in \cite[Chapter 6 Proposition 3.4.2]{ChambertLoirNicaiseSebag}. The remainder of the proposition follows from \autoref{corollaryMotivicMeasureOfMeasurableSet} and the analogous statement for schemes \cite[Chapter 6 Proposition 3.4.3]{ChambertLoirNicaiseSebag} applied to the scheme $\widetilde{X}$ in the statement of \autoref{corollaryMotivicMeasureOfMeasurableSet}.
\end{proof}

\begin{proposition}
\label{subsetOfNegligibleIsNegligible}
Let $\cX$ be an equidimensional finite type quotient stack over $k$, and let $\cC \subset \cD \subset |\sL(\cX)|$. If $\cD$ is measurable and $\mu_\cX(\cD) = 0$, then $\cC$ is measurable and $\mu_\cX(\cC) = 0$.
\end{proposition}

\begin{proof}
The proposition holds by the exact same proof used for the analogous statement for schemes in \cite[Chapter 6 Corollary 3.5.5(a)]{ChambertLoirNicaiseSebag}.
\end{proof}

\begin{proposition}
\label{measureOfSubsetHasSmallerSize}
Let $\cX$ be an equidimensional finite type quotient stack over $k$, and let $\cC, \cD$ be measurable subsets of $|\sL(\cX)|$. If $\cC \subset \cD$, then
\[
	\Vert \mu_\cX(\cC) \Vert \leq \Vert \mu_\cX(\cD) \Vert.
\]
\end{proposition}

\begin{proof}
By \autoref{quotientStackIsQuotientBySpecialGroup}, there exist $G, \widetilde{X}, \rho$ as in the statement of \autoref{theoremMainMotivicIntegrationForQuotientStacks} such that $G$ is a general linear group. Let $\widetilde{C} = \sL(\rho)^{-1}(\cC)$ and $\widetilde{D} = \sL(\rho)^{-1}(\cD)$. Then
\begin{align*}
	\Vert \mu_\cX(\cC) \Vert &= \Vert \mu_{\widetilde{X}}(\widetilde{C})\e(G)^{-1}\bL^{\dim G} \Vert\\
	&\leq \Vert \e(G)^{-1}\bL^{\dim G}\Vert \Vert \mu_{\widetilde{X}}(\widetilde{C})\Vert\\
	&\leq  \Vert \e(G)^{-1}\bL^{\dim G}\Vert \Vert \mu_{\widetilde{X}}(\widetilde{D})\Vert\\
	&= \Vert \e(G)^{-1}\bL^{\dim G}\Vert \Vert \mu_{\cX}(\cD) \e(G)\bL^{-\dim G}\Vert\\
	&\leq  \Vert \e(G)^{-1} \bL^{\dim G} \Vert \Vert\e(G) \bL^{-\dim G}\Vert \Vert \mu_\cX(\cD) \Vert\\
	&= \Vert \mu_\cX(\cD) \Vert,
\end{align*}
where the first and fourth lines follow from \autoref{corollaryMotivicMeasureOfMeasurableSet}, the third line follows from the analogous statement \cite[Chapter 6 Corollary 3.3.5]{ChambertLoirNicaiseSebag} for schemes applied to $\widetilde{X}$, and the last line follows from \autoref{lemmaSizeOfClassOfGeneralLinearGroup}.
\end{proof}

\begin{proposition}
\label{closedSubstackGivesNegligibeSet}
Let $\cX$ be an equidimensional finite type quotient stack over $k$, let $\cY$ be a closed substack of $\cX$ with $\dim\cY < \dim\cX$, and let $\cC \subset |\sL(\cX)|$ be the image of $|\sL(\cY)|$ in $|\sL(\cX)|$. Then $\cC$ is measurable and $\mu_\cX(\cC) = 0$.
\end{proposition}

\begin{proof}
For each $n \in \N$, let $\cC_n \subset |\sL_n(\cX)|$ be the image of $|\sL_n(\cY)|$ in $|\sL_n(\cX)|$, and let $\cC^{(n)} = (\theta_n)^{-1}(\cC_n)$. By \cite[Proposition 3.5(vi)]{Rydh2011}, each $\cC_n$ is a closed subset of $\sL_n(\cX)$, so each $\cC^{(n)}$ is a cylinder in $\sL(\cX)$.

By \autoref{quotientStackIsQuotientBySpecialGroup}, there exist $G, \widetilde{X}, \rho$ as in the statement of \autoref{theoremMainMotivicIntegrationForQuotientStacks}. Let $\widetilde{Y} = \widetilde{X} \times_\cX \cY$. Then $\sL_n(\rho)^{-1}(\cC_n)$ is the underlying set of $\sL_n(\widetilde{Y})$. Thus by \autoref{theoremMainMotivicIntegrationForQuotientStacks},
\begin{align*}
	\mu_{\cX}(\cC^{(n)}) &= \mu_{\widetilde{X}} (\sL(\rho)^{-1}(\cC^{(n)})) \e(G)^{-1} \bL^{\dim G}\\
	&= \mu_{\widetilde{X}}((\theta_n)^{-1}(\sL_n(\widetilde{Y})))\e(G)^{-1} \bL^{\dim G}.
\end{align*}
By \cite[Chapter 6 Proposition 2.3.1]{ChambertLoirNicaiseSebag},
\[
	\lim_{n \to \infty} \mu_{\widetilde{X}}((\theta_n)^{-1}(\sL_n(\widetilde{Y}))) = 0,
\]
so
\[
	\lim_{n \to \infty} \mu_{\cX}(\cC^{(n)}) = 0.
\]
Therefore for any $\varepsilon \in \R_{>0}$, we get that $(\emptyset, (\cC^{(n)}))$ is a cylindrical $\varepsilon$-approximation of $\cC$ for sufficiently large $n$, and we are done by definition of $\mu_\cX$.
\end{proof}

\begin{proposition}
\label{stackMeasurablesFormBooleanAlgebra}
Let $\cX$ be an equidimensional finite type quotient stack over $k$, and let $\cC$ and $\cD$ be measurable subsets of $|\sL(\cX)|$. Then the intersection $\cC \cap \cD$, the union $\cC \cup \cD$, and the complement $\cC \setminus \cD$ are all measurable subsets of $|\sL(\cX)|$.
\end{proposition}

\begin{proof}
The proposition holds by the exact same proof used for the analogous statement for schemes in \cite[Chapter 6 Proposition 3.2.8]{ChambertLoirNicaiseSebag}.
\end{proof}

\begin{proposition}
\label{stackMeasureOpenSubstack}
Let $\cX$ be an equidimensional finite type quotient stack over $k$, let $\iota: \cU \hookrightarrow \cX$ be the inclusion of an open substack, and let $\cC \subset (\theta_0)^{-1}(|\cU|) \subset |\sL(\cX)|$. Then $\cC$ is a measurable subset of $|\sL(\cX)|$ if and only if $\sL(\iota)^{-1}(\cC)$ is a measurable subset of $|\sL(\cU)|$, and in that case
\[
	\mu_\cX(\cC) = \mu_\cU(\sL(\iota)^{-1}(\cC)).
\]
\end{proposition}

\begin{proof}
As in the case of schemes, this is an easy consequence of the definitions and the fact that for all $n \in \N$, the morphism $\sL_n(\iota): \sL_n(\cU) \to \sL_n(\cX)$ is an open immersion by \cite[Proposition 3.5(vii)]{Rydh2011}.
\end{proof}

\subsection{Non-separatedness functions}\label{subsectionNonSeparatednessFunctions}

We now introduce notation for the non-separatedness functions $\sep_{\pi, \cC}$, $\sep_\pi$, and $\sep_\cX$ that were used in the statements of the main conjectures and theorems of this paper. Throughout this subsection, let $\cX$ be an Artin stack over $k$, let $\cC \subset |\sL(\cX)|$, let $X$ be a scheme over $k$, and let $\pi: \cX \to X$ be a map. For any field extension $k'$ of $k$, we will let $\overline{\cC}(k')$ denote the subset of $\overline{\sL(\cX)}(k')$ consisting of arcs whose classes in the set $|\sL(\cX)|$ are contained in $\cC$.

If $k'$ is a field extension of $k$ and $\varphi \in \sL(X)(k')$, we set
\[
	\sep_{\pi, \cC}(\varphi) = \#\left(\overline{\cC}(k') \cap \overline{(\sL(\pi)^{-1}(\varphi))}(k')\right) \in \N \cup \{\infty\},
\]
which induces a map $\sep_{\pi, \cC}: \sL(X) \to \N \cup \{\infty\}$ by considering each $\varphi \in \sL(X)$ as a point valued in its residue field. If furthermore we assume that $X$ is integral, finite type, separated, and has log-terminal singularities, that $\sep_{\pi, \cC}: \sL(X) \to \N \cup\{\infty\}$ has measurable fibers, and that $C \subset \sL(X)$ is a measurable subset, then we can consider the motivic integral
\[
	\int_C \sep_{\pi, \cC} \diff\mu^\Gor_X = \sum_{n \in \N} n \mu^\Gor_X(\sep_{\pi, \cC}^{-1}(n) \cap C) \in \widehat{\sM}_k[\bL^{1/m}],
\]
where $m \in \Z_{>0}$ is such that $mK_X$ is Cartier. Note that with the above assumptions, the series defining $\int_C \sep_{\pi, \cC} \diff\mu^\Gor_X$ converges because 
\[
	\lim_{n \to \infty} n \mu^\Gor_X(\sep_{\pi, \cC}^{-1}(n) \cap C) = 0,
\]
which follows from
\[
	\lim_{n \to \infty} \mu^\Gor_X(\sep_{\pi, \cC}^{-1}(n) \cap C) = 0,
\]
which, for example, is a consequence of \autoref{countableAdditivityOfGorensteinMeasure}.

Set
\[
	\sep_\pi = \sep_{\pi, |\sL(\cX)|},
\]
and 
\[
	\sep_\cX = 1/(\sep_\pi \circ \sL(\pi)): |\sL(\cX)| \to \Q_{\geq 0} \cup \{\infty\}.
\]
If furthermore we assume that $\cX$ is an equidimensional and finite type quotient stack over $k$ and that $\sep_\cX: |\sL(\cX)| \to \Q_{\geq 0} \cup \{\infty\}$ has measurable fibers, we can consider the motivic integral
\[
	\int_{\sL(\cX)} \sep_\cX \diff\mu_\cX = \sum_{n \in \Z_{\geq 1}} (1/n) \mu_\cX( \sep_\cX^{-1}(1/n) ) \in \widehat{\sM_k \otimes_\Z \Q},
\]
where the ring $\widehat{\sM_k \otimes_\Z \Q}$ is defined like $\widehat{\sM}_k$ in \autoref{subsectionPreliminariesMotivicIntegrationSchemes} by replacing any mention of $K_0(\Var_k)$ with $K_0(\Var_k) \otimes_\Z \Q$ and any mention of ``subgroup'' with ``$\Q$-subspace''. With the above assumptions, the series defining $\int_{\sL(\cX)} \sep_\cX \diff\mu_\cX$ converges because
\[
	\lim_{n \to \infty} (1/n) \mu_\cX( \sep_\cX^{-1}(1/n) ) = 0,
\]
which by the definition of the norm on $\widehat{\sM_k \otimes_\Z \Q}$ follows from
\[
	\lim_{n \to \infty} \mu_\cX( \sep_\cX^{-1}(1/n) ) = 0,
\]
which follows from \autoref{corollaryMotivicMeasureOfMeasurableSet} and properties of motivic measures for schemes.

\subsection{Motivic integration for smooth stacks via the cotangent complex}\label{subsectionMotivicIntegrationForSmoothStacksViaTheCotangentComplex}

In this subsection, we prove that the motivic measure $\mu_\cX$ is also well defined when $\cX$ is an equidimensional smooth Artin (not-necessarily-quotient) stack over $k$. We only explicitly verify this for cylinders, but by a standard argument (identical to the one for schemes in \cite{ChambertLoirNicaiseSebag}), this leads to well defined notions (that coincide with our above definitions in the case $\cX$ is a quotient stack) of measurable subsets of $|\sL(\cX)|$ and their motivic measures. The main result of this subsection is the following theorem, which immediately implies that \autoref{definitionMotivicMeasureCylinder} makes sense in this setting.

\begin{theorem}\label{TheoremMotivicMeasureCylinderNonQuotientWellDefined}
Let $\cX$ be an equidimensional smooth Artin stack over $k$, and let $\cC \subset |\sL(\cX)|$ be a cylinder. Then the set $\theta_n(\cC) \subset |\sL_n(\cX)|$ is constructible for each $n \in \N$, and the sequence 
\[
	\{\e(\theta_n(\cC)) \bL^{-(n+1)\dim\cX}\}_{n \in \N} \subset K_0(\Stack_k)
\] 
stabilizes for sufficiently large $n$.
\end{theorem}

We first prove two lemmas, after which we will return to proving \autoref{TheoremMotivicMeasureCylinderNonQuotientWellDefined}.

\begin{lemma}\label{LemmaFibersOfTruncationStackSpacesAndPiecewiseConstant}
Let $\cX$ be an equidimensional smooth Artin stack over $k$, and let $n \in \N$. There exists some $\ell \in \Z_{>0}$, a partition $|\sL_n(\cX)| = \bigsqcup_{i=1}^\ell \cC_i$ of $|\sL_n(\cX)|$ into constructible subsets $\cC_i$, and some $r_1, \dots, r_\ell, j_1, \dots, j_\ell \in \N$ such that
\begin{itemize}

\item for any $i \in \{1, \dots, \ell\}$, we have $r_i - j_i = \dim\cX$, and

\item for any $i \in \{1, \dots, \ell\}$, any field extension $k'$ of $k$, and any $\psi_n \in \sL_n(\cX)(k')$ whose class in $|\sL_n(\cX)|$ is contained in $\cC_i$, we have
\[
	(\theta^{n+1}_n)^{-1}(\psi_n) \cong \bA^{r_i}_{k'} \times_{k'} B\bG_{a}^{j_i}.
\]
	
\end{itemize}
\end{lemma}

\begin{proof}
Fix $\xi_n\colon\Spec k'\to\sL_n(\cX)$ and let $\cY_{\xi_n}$ denote the fiber of the truncation map $\sL_{n+1}(\cX)\to\sL_n(\cX)$ over $\xi_n$. For any $\alpha\colon\Spec A\to\Spec k'$, the $A$-valued points $\cY_{\xi_n}(A)$ are the category of lifts of $\xi_n\otimes_{k'}A$ to $\sL_{n+1}(\cX)$. For all $m\geq0$, let $\cX_m=\cX\otimes_k k[t]/(t^{m+1})$ and $\alpha_m\colon\Spec A[t]/(t^{m+1})\to\Spec k'[t]/(t^{m+1})$ be the map induced by $\alpha$; for $m\leq n$, let $\varphi_m\colon\Spec k'[t]/(t^{m+1})\to\cX_m$ denote the map induced by $\xi_n$. We then obtain a cartesian diagram
\[
\xymatrix{
\Spec A\ar[r]\ar[d]_-{\alpha_0\varphi_0} & \Spec A'[t]/(t^{n+1})\ar[r]\ar[d]_-{\alpha_n\varphi_n} & \Spec A'[t]/(t^{n+2})\ar@{-->}[d]\ar@/^1.25pc/[dd]\\
\cX\ar[r]\ar[d] & \cX_n\ar[r]\ar[d] & \cX_{n+1}\ar[d]\\
\Spec k\ar[r] & \Spec k[t]/(t^{n+1})\ar[r] & \Spec k[t]/(t^{n+2})
}
\]
where the curved arrow is the structure map. Let $\mathcal{J}_n$ denote the ideal sheaf of $\Spec k'[t]/(t^{n+1})\to\Spec k'[t]/(t^{n+2})$ considered as a $k'$-module. By \cite[Theorem 1.5]{Olsson2006}, and the fact that $\cX_{n+1}$ and $A[t]/(t^{n+2})$ are flat over $k[t]/(t^{n+2})$, the obstruction to the existence of a dotted arrow in the above diagram lives in
\[
\ext^1(L(\alpha_0\varphi_0)^*L_{\cX/k},\alpha_0^*\mathcal{J}_n)=\ext^1(L\varphi_0^*L_{\cX/k},\cO_{k'})\otimes_{k'}\mathcal{J}_n\otimes_{k'}A.
\]
We will show that this group vanishes and so by \cite[Theorem 1.5]{Olsson2006}, the objects (resp.~automorphisms) of $\cY_{\xi_n}(A)$ are parameterized by $\ext^n(L\varphi_0^*L_{\cX/k},\cO_{k'})\otimes_{k'}\mathcal{J}_n\otimes_{k'}A$ where $n=0$ (resp.~$n=-1$). In particular, if $V$ (resp.~$G$) denotes the affine space (resp.~algebraic vector group) over $k'$ associated to the vector space $\ext^n(L\varphi_0^*L_{\cX/k},\cO_{k'})\otimes_{k'}\mathcal{J}_n$ with $n=0$ (resp.~$n=-1$), then we have
\[
\cY_{\xi_n}\,\cong\, V\times_{k'} G\,\cong\, 
\bA^{r(\xi_n)}\times B\bG_a^{j(\xi_n)},
\]
where $r(\xi_n)=\dim \ext^0(L\varphi_0^*L_{\cX/k},\cO_{k'})$ and $j(\xi_n)=\dim \ext^1(L\varphi_0^*L_{\cX/k},\cO_{k'})$. Note that this implies
\[
e(\cY_{\xi_n})=\bL^{r(\xi_n)-j(\xi_n)}.
\]
Therefore, to finish the proof of the theorem, it suffices to show
\begin{equation}\label{eqn:dim-cotangent-complex-pulled-back}
\ext^1(L\varphi_0^*L_{\cX/k},\cO_{k'})=0,\quad\quad r(\xi_n)-j(\xi_n)=\dim\cX,
\end{equation}
and that the locus of $\varphi_0\in|\cX|$ where $r(\xi_n)$ is constant is given by a constructible set. Since these remaining statements depend only on the dimension of the $\ext$-groups over $k'$, it suffices to replace $k'$ with an extension field, and hence we can assume $k'$ is algebraically closed.

Let $\rho\colon\widetilde{X}\to\cX$ be a smooth cover. Since $k'$ is algebraically closed, we may fix a lift $\phi_0\colon\Spec k'\to\widetilde{X}$ of $\varphi_0$. Since $\widetilde{X}$ and $\rho$ are smooth, we have an exact triangle
\[
p^*L_{\cX/k}\to\Omega^1_{\widetilde{X}/k}\to\Omega^1_{\widetilde{X}/\cX}
\]
from which we obtain an exact triangle
\[
\Gamma(\phi_0^*T_{\widetilde{X}/\cX})\to\Gamma(\phi_0^*T_{\widetilde{X}/k})\to \mathrm{RHom}(L\varphi_0^*L_{\cX/k},\cO_{k'}).
\]
In particular, $\ext^n(L\varphi_0^*L_{\cX/k},\cO_{k'})=0$ for $n\neq0,-1$ and there is an exact sequence
\begin{equation}\label{eqn:fibers-via-cotangent-complex}
0\to\ext^{-1}(L\varphi_0^*L_{\cX/k},\cO_{k'})\to\Gamma(\phi_0^*T_{\widetilde{X}/\cX})\to\Gamma(\phi_0^*T_{\widetilde{X}/k})\to\ext^0(L\varphi_0^*L_{\cX/k},\cO_{k'})\to0.
\end{equation}
Thus, 
\[
r(\xi_n)-j(\xi_n)=\dim\Gamma(\phi_0^*T_{\widetilde{X}/k})-\dim\Gamma(\phi_0^*T_{\widetilde{X}/\cX})=
\dim\cX,
\]
thereby establishing \eqref{eqn:dim-cotangent-complex-pulled-back}. Finally, note from \eqref{eqn:fibers-via-cotangent-complex} that the cokernel of $\Gamma(\phi_0^*T_{\widetilde{X}/\cX})\to\Gamma(\phi_0^*T_{\widetilde{X}/k})$ depends only on $\varphi_0$ and not the choice of lift $\phi_0$. So, the locus of $\varphi_0\in|\cX|$ where $r(\xi_n)$ is constant is the image under $\rho$ of the locus of $\phi_0\in|\widetilde{X}|$ where the dimension is constant. By Chevalley's Theorem for Artin stacks \cite[Theorem 5.2]{HallRydh2017}, it is therefore it is enough to show that the locus of such $\phi_0$ is constructible. This follows by applying \cite[Lemma 0BDI]{stacks-project} to the $2$-term complex $T_{\widetilde{X}/\cX}\to T_{\widetilde{X}/k}$ and using that $|\widetilde{X}|$ is Noetherian so that all locally constructible sets are constructible.
\end{proof}

\begin{lemma}\label{LemmaConstantFibersClassInGrothendieckRingStackBase}
Let $\cY$, $\cZ$, and $\cF$ be finite type Artin stacks over $k$, let $\cY \to \cZ$ be a $k$-morphism, and assume that for any field extensions $k'$ of $k$ and any $k$-morphism $\Spec(k') \to \cZ$, there exists a $k'$-isomorphism
\[
	(\cY \times_{\cZ} \Spec(k'))_\red \cong (\cF \times_{\Spec(k)} \Spec(k'))_\red.
\]
Then
\[
	\e(\cY) = \e(\cF)\e(\cZ) \in K_0(\mathbf{Stack}_k).
\]
\end{lemma}

\begin{proof}
Because $\cZ$ can be stratified by quotient stacks \cite[Proposition 3.5.9]{Kresch}, we may assume that $\cZ = [Z / G]$ for some finite type scheme $Z$ over $k$ with an action by a general linear group $G$ over $k$. Let $\cY' = \cY \times_{\cZ} Z$. Because $Z \to \cZ$ and $\cY' \to \cY$ are $G$-torsors and $G$ is a special group, \cite[Proposition 1.1(ii)]{Ekedahl} gives
\begin{align*}
	\e(Z) &= \e(G)\e(\cZ),\\
	\e(\cY') &= \e(G)\e(\cY).
\end{align*}
By the hypotheses on $\cY \to \cZ$, \autoref{piecewiseTrivialFibrationCriterion} implies that $\cY' \to Z$ is a piecewise trivial fibration with fiber $\cF$, so in particular,
\[
	\e(\cY') = \e(\cF)\e(Z).
\]
Thus,
\[
	\e(G)\e(\cY) = \e(G)\e(\cF)\e(\cZ).
\]
Because $G$ is a special group, $\e(G)^{-1} \in K_0(\mathbf{Stack}_k)$, so we are done.
\end{proof}

We may now prove \autoref{TheoremMotivicMeasureCylinderNonQuotientWellDefined}.

\begin{proof}[Proof of \autoref{TheoremMotivicMeasureCylinderNonQuotientWellDefined}]
By definition there exists some $n_0 \in \N$ and some constructible subset $\cC_{n_0} \subset |\sL_{n_0}(\cX)|$ such that $\cC = (\theta_{n_0})^{-1}(\cC_{n_0})$. Because $\cX$ is smooth, infinitesimal lifting implies that the truncation maps $\theta_n: |\sL(\cX)| \to |\sL_n(\cX)|$ are all surjective, so
\[
	\theta_n(\cC) = \begin{cases} (\theta^n_{n_0})^{-1}(\cC_{n_0}), & n \geq n_0, \\ \theta^{n_0}_n(\cC_{n_0}), & n < n_0. \end{cases}
\]
Thus all $\theta_n(\cC)$ are constructible (immediately for $n \geq n_0$ and by Chevalley's theorem for Artin stacks \cite[Theorem 5.2]{HallRydh2017} for $n < n_0$).

The remainder of the theorem then follows from the fact that \autoref{LemmaFibersOfTruncationStackSpacesAndPiecewiseConstant} and \autoref{LemmaConstantFibersClassInGrothendieckRingStackBase} imply that for any $n \geq n_0$,
\[
	\e(\theta_n(\cC)) = (\theta^n_{n_0})^{-1}(\cC_{n_0}) = \e(\cC_{n_0})\bL^{(n-n_0)\dim\cX}.\qedhere
\]
\end{proof}

\section{Fibers of the map of arcs}
\label{sectionFibersOfTheMapsOfArcs}

Our goal in this section is to give a combinatorial characterization of the fibers of $\sL(\pi)\colon\sL(\cX)\to\sL(X)$, where $\cX$ is a fantastack and $\pi\colon\cX\to X$ is its good moduli space map, see \autoref{MainPropositionFibersOfTheMapOfArcs}. We accomplish this goal by first defining the tropicalization of arcs both for toric varieties and toric stacks.

\subsection{Tropicalizing arcs of toric stacks}
\label{subsec:trop-of-arcs}

Given a toric variety $X=\Spec(k[P])$, a $k$-algebra $R$, and an arc $\varphi \in \sL(\widetilde{X})(R)$, we denote by $\varphi^*(p)$ the image of $p$ under $P\to k[P] \to R\llbracket t \rrbracket$, where the latter map is the pullback corresponding to $\varphi$.

\begin{definition}
\label{def:tropicalization-of-arcs}
If $\sigma$ is a pointed rational cone on a finite rank lattice $N$ and $k'$ is a field extension of $k$, we define the \emph{tropicalization map}
\[
\trop: \sL(X_\sigma)(k') \to \Hom(\sigma^\vee\cap N^*, \N \cup \{\infty\})
\]
by $\trop(\varphi)(p):=\ord_t\varphi^*(p)$ where $\ord_t$ denotes the order of vanishing at $t$.

More generally, if $(\sigma,\nu\colon \widetilde{N}\to N)$ is a stacky fan with $\sigma$ a pointed cone and $\cX:=\cX_{\sigma,\nu}:=[X_\sigma/G_\nu]$ is the corresponding toric stack, then we define the \emph{tropicalization map} on isomorphism classes of arcs 
\[
\trop: \overline{\sL(\cX)}(k') \to \Hom(\sigma^\vee\cap \widetilde{N}^*, \N \cup \{\infty\})
\]
as follows. If $\psi\in\sL(\cX)(k')$, then fix a finite field extension $k''$ of $k'$ and a lift $\widetilde{\psi}\in\sL(X_\sigma)(k'')$ of $\psi$. We let $\trop(\psi):=\trop(\widetilde{\psi})$. We show in \autoref{l:TropDependsOnlyOnIsoClass} that this is well-defined.
\end{definition}

\begin{remark}
\label{rmk:non-full-dimensional}
Note that we have a natural inclusion
\[
\sigma\cap N=\Hom(\sigma^\vee\cap N^*, \N)\subset\Hom(\sigma^\vee\cap N^*, \N \cup \{\infty\}).
\]
\end{remark}

\begin{lemma}
\label{l:pullback-diag-torsors}
Let $\Omega$ be any field of characteristic $0$ and let $G=\bG_m^r\times\prod_{i=1}^N\mu_{n_i}$. Then every $G$-torsor over $\Spec(\Omega[[t]])$ is isomorphic to the pullback of a $G$-torsor over $\Spec(\Omega)$.
\end{lemma}
\begin{proof}
Let $q\colon\Spec(\Omega[[t]])\to\Spec(\Omega)$ denote the structure map. Since $G$ is an \'etale group scheme, $G$-torsors on any $\Omega$-scheme $Y$ are classified up to isomorphism by $H^1_{et}(Y,G)=H^1_{et}(Y,\bG_m)^{\oplus r}\oplus\bigoplus_{i=1}^N H^1_{et}(Y,\mu_{n_i})$. In particular, it suffices to show that the pullback map $q^*\colon H^1_{et}(\Spec(\Omega),G)\to H^1_{et}(\Spec(\Omega[[t]]),G)$ is an isomorphism when $G$ is either $\bG_m$ or $\mu_n$.

We first handle the case $G=\bG_m$. Since $H^1_{et}(Y,\bG_m)=\mathrm{Pic}(Y)$ and since both $\mathrm{Pic}(\Spec(\Omega))$ and $\mathrm{Pic}(\Spec(\Omega[[t]]))$ are trivial, we see that $q^*$ is an isomorphism.

We next handle the case $G=\mu_n$. From the Kummer sequence
\[
1\to\mu_n\to\bG_m\xrightarrow{\times n}\bG_m\to1,
\]
we see that if $Y$ is any $\Omega$-scheme with trivial Picard group, we have
\[
H^1_{et}(Y,\mu_n)=\cO_Y(Y)^*/(\cO_Y(Y)^*)^n,
\]
see e.g., \cite[p.~125]{MilneEtCoh}. So it remains to show $k^*/(k^*)^n\to k[[t]]^*/(k[[t]]^*)^n$ is an isomorphism. Since every element $f(t)\in k[[t]]^*$ can be written uniquely as $ag(t)$ with $a\in k^*$ and $g(t)\in k[[t]]^*$ with $g(t)-1\in tk[[t]]$, it is enough to prove that such $g(t)$ are in $(k[[t]]^*)^n$. This follows immediately from Hensel's Lemma: the polynomial $P(x)=x^n-g(t)\in k[[t]][x]$ has a root since $P(1)=0$ mod $t$ and $P'(1)\neq0$ mod $t$.
\end{proof}

\begin{lemma}
\label{l:TropDependsOnlyOnIsoClass}
With notation as in \autoref{def:tropicalization-of-arcs}, such a lift $\widetilde{\psi}$ exists and $\trop(\psi)$ is independent of both $k''$ and $\widetilde{\psi}$.
\end{lemma}
\begin{proof}
For ease of notation, let $F:=\sigma^\vee\cap \widetilde{N}^*$ and $G:=G_\nu=\Spec(k[A])$, where $A$ is a finitely generated abelian group. Note that the $G$-action on $X_\sigma$ corresponds to a monoid map $\eta\colon F\to A$. The arc $\psi$ corresponds to a $G$-torsor $Q\to\Spec(k'\llbracket t\rrbracket)$ and $G$-equivariant map $Q\to X_\sigma$. By \autoref{l:pullback-diag-torsors}, $Q$ is isomorphic to the pullback of a $G$-torsor over $\Spec(k')$, which can be itself be trivialized after a finite field extension $k''$. Thus, after base change to $k''\otimes_{k'} k'\llbracket t\rrbracket\simeq k''\llbracket t\rrbracket$, we obtain a trivialization of $Q$ and hence a lift $\widetilde{\psi}$.

Next, it is clear that if $\widetilde{\psi}\in\sL(X_\sigma)(k'')$ is a lift of $\psi$ and $k'''$ is a finite field extension of $k''$, then $\trop(\widetilde{\psi})=\trop(\widetilde{\psi}\otimes_{k''}k''')$. So, it suffices to show that if $\widetilde{\psi}_1,\widetilde{\psi}_2\in\sL(X_\sigma)(k'')$ are both lifts of $\psi$, then $\trop(\widetilde{\psi}_1)=\trop(\widetilde{\psi}_2)$. In this case, there exists $g\in G(k''\llbracket t\rrbracket)$ such that $g\cdot\widetilde{\psi}_1=\widetilde{\psi}_2$. Letting $g^*(a)$ denote the pullback of $a\in A$ under the map $g^*\colon k''[A]\to k''\llbracket t\rrbracket^*$, we therefore have
\[
g^*(\eta(f))\psi_1^*(f)=\psi_2^*(f).
\]
Since $g^*(\eta(f))$ is a unit, the power series $\psi_1^*(f)$ and $\psi_2^*(f)$ have the same $t$-order of vanishing, i.e.~$\trop(\widetilde{\psi}_1)=\trop(\widetilde{\psi}_2)$.
\end{proof}

We record some basic properties of $\trop$ that will be useful later on.

\begin{definition}
\label{def:trop-inverse}
Let $(\sigma,\nu\colon \widetilde{N}\to N)$ be a stacky fan with $\sigma$ a pointed cone. For any $w \in \sigma \cap \widetilde{N} \subset \Hom(\sigma^\vee \cap \widetilde{N}^*, \N \cup \{\infty\})$, let
\[
\trop^{-1}(w)=\{\psi\in\overline{\sL(\cX_{\sigma,\nu})}(k')\mid \textrm{$k'$\ is\ a field\ extension\ of $k$\ and\ }\trop(\psi)=w\}
\subset|\sL(\cX_{\sigma,\nu})|
\]
where the arcs are taken up to equivalence.
\end{definition}

\begin{remark}
\label{fiberOfTropIsCylinder}
Let $\sigma$ be a pointed rational cone on a finite rank lattice $N$ and let $P=\sigma^\vee\cap N^*$ so that $X_\sigma=\Spec(k[P])$. For any $p\in P$, let $\chi^p\in k[P]$ be the corresponding monomial. If $p_1, \dots, p_s$ are generators for $P$, then for every $w \in \sigma \cap N$, we see
\[
	\trop^{-1}(w) = \bigcap_{i = 1}^s \ord_{\chi^{p_i}}^{-1}(\langle w, p_i \rangle).
\]
and hence, $\trop^{-1}(w) \subset \sL(X_\sigma)$ is a cylinder.
\end{remark}

\begin{lemma}
\label{l:trop-basic-properties}
Let $\sigma$ (resp.~$\widetilde{\sigma}$) be a pointed rational cone on a finite rank lattice $N$ (resp.~$\widetilde{N}$). If $\rho\colon X_{\widetilde{\sigma}}\to X_{\sigma}$ is a toric morphism and $\beta\colon\widetilde{\sigma}\cap \widetilde{N}\to \sigma\cap N$ is the induced map, then
\begin{enumerate}
\item\label{l:trop-basic-properties::image} for every field extension $k'$ of $k$ and every arc $\psi\in\sL(X_{\widetilde{\sigma}})(k')$, if $\trop(\psi)\in\widetilde\sigma^\vee\cap\widetilde{N}^*$, then
\[
\trop(\sL(\rho)(\psi))=\beta(\trop(\psi)).
\]
\item\label{l:trop-basic-properties::inverse-image} if for all $f\in\widetilde{\sigma}^\vee\cap\widetilde{N}^*$ there exists $f'\in\widetilde{\sigma}^\vee\cap\widetilde{N}^*$ such that $f+f'$ lies in the image of $\sigma^\vee\cap N^*$, then
\[
\sL(\rho)^{-1}(\trop^{-1}(w))=\bigcup_{\widetilde{w} \in \beta^{-1}(w)} \trop^{-1}(\widetilde{w}).
\]
\end{enumerate}
\end{lemma}
\begin{proof}
Let $\rho^*\colon\sigma^\vee\cap N^*\to\widetilde\sigma^\vee\cap\widetilde{N}^*$ denote the pullback map on monoids. First note that if $k'$ is a field extension of $k$, $\psi\in\sL(X_{\widetilde{\sigma}})(k')$, and $p\in\sigma^\vee\cap N^*$, then
\[
\trop(\sL(\rho)(\psi))(p)=\ord_t(\psi^*\rho^*(p))=(\trop(\psi))(\rho^*(p)).
\]
To prove (\ref{l:trop-basic-properties::image}), let $\trop(\psi)=\widetilde{w}\in\widetilde\sigma^\vee\cap\widetilde{N}^*$. Then by the above equalities, we see
\[
\trop(\sL(\rho)(\psi))(p)=\langle\widetilde{w},\rho^*(p)\rangle=\langle\beta(\widetilde{w}),p\rangle
\]
so $\trop(\sL(\rho)(\psi))=\beta(\widetilde{w})$.

Part (\ref{l:trop-basic-properties::inverse-image}) follows immediately from part (\ref{l:trop-basic-properties::image}) provided we can show that $\trop(\sL(\rho)(\psi))\in\sigma\cap N$ implies $\trop(\psi)\in\widetilde\sigma\cap\widetilde{N}$. Let $f\in\widetilde{\sigma}^\vee\cap\widetilde{N}^*$. By hypothesis, there exists $f'\in\widetilde{\sigma}^\vee\cap\widetilde{N}^*$ such that $f+f'=\rho^*(p)$ for some $p\sigma^\vee\cap N^*$. Then
\[
(\trop(\psi))(f)+(\trop(\psi))(f')=(\trop(\psi))(\rho^*(p))=\trop(\sL(\rho)(\psi))(p)\neq\infty.
\]
So $(\trop(\psi))(f)\neq\infty$ for all $f$, and hence $\trop(\psi)\in\widetilde\sigma\cap\widetilde{N}$.
\end{proof}

\begin{corollary}
\label{lemmaTropOfStackArcMapsToTropOfGoodModuliSpaceArc}
Let $\cX=\cF_{\sigma,\nu}$ be a fantastack and suppose the good moduli space map $\pi\colon\cX\to X=X_\sigma$ is an isomorphism over the torus $T\subset X$. Let $k'$ be a field extension of $k$ and $\varphi\in \sL(X)(k')$ with $\trop(\varphi)=w\in\sigma\cap N$. If $\psi \in \sL(\cX)(k')$ and $\sL(\pi)(\psi)=\varphi$, then $\trop(\psi) \in \beta^{-1}(w)$.
\end{corollary}
\begin{proof}
We keep the notation listed in \autoref{notationforcanonicalfantastackforfulldimensionalcone} and let $\cX=[\widetilde{X}/G_\nu]$. We know there exists a finite field extension $k''$ of $k'$ and a lift $\widetilde{\psi}\in\sL(\widetilde{X})(k'')$ of $\psi$. By construction, $\trop(\psi)=\trop(\widetilde{\psi})$. From \autoref{elementofFaddstogetelementofP}, we know the hypotheses of \autoref{l:trop-basic-properties}(\ref{l:trop-basic-properties::inverse-image}) are satisfied, so $\trop(\widetilde{\psi})\in\beta^{-1}(w)$.
\end{proof}

\subsection{Lifting arcs to a fantastack}

We can now state the main result of this section, which shows that for the good moduli space map $\pi\colon\cX\to X$ of a fantastack, isomorphism classes of arcs in the fibers of $\sL(\pi)$ are completely determined by their tropicalizations.


\begin{theorem}
\label{MainPropositionFibersOfTheMapOfArcs}
Let $\cX=\cF_{\sigma,\nu}$ be a fantastack and assume the good moduli space map $\pi\colon\cX\to X:=X_\sigma$ is an isomorphism over the torus $T\subset X$. With the notation listed in \autoref{notationforcanonicalfantastackforfulldimensionalcone}, let $k'$ be a field extension of $k$ and $\varphi \in \sL(X)(k')$ with $\trop(\varphi) = w\in\sigma\cap N$. Then
\[
	\trop: \overline{\sL(\cX)}(k') \to \Hom(F, \N \cup \{\infty\})
\]
induces, by restriction, a bijection
\[
	\overline{(\sL(\pi)^{-1}(\varphi))}(k') \to \beta^{-1}(w).
\]
In particular, $\overline{(\sL(\pi)^{-1}(\varphi))}(k')$ is a finite set.
\end{theorem}

For the rest of this section, we fix the notation as in the statement of \autoref{MainPropositionFibersOfTheMapOfArcs}. By \autoref{lemmaTropOfStackArcMapsToTropOfGoodModuliSpaceArc}, we know any $\psi \in \sL(\cX)(k')$ with $(\sL(\pi))(\psi) = \varphi$ must satisfy $\trop(\psi) \in \beta^{-1}(w)$. We therefore have an induced map $\overline{(\sL(\pi)^{-1}(\varphi))}(k') \to \beta^{-1}(w)$. We show injectivity and surjectivity in \autoref{prop:inj-lifts-trop} and \autoref{prop:surj-lifts-trop}, respectively. Note that finiteness of $\overline{(\sL(\pi)^{-1}(\varphi))}(k')$ then follows from \autoref{fiberabovewisfiniteset}.

\begin{proposition}
\label{prop:inj-lifts-trop}
The restriction of the map
\[
	\trop: \overline{\sL(\cX)}(k') \to \Hom(F, \N \cup \{\infty\})
\]
to $\overline{(\sL(\pi)^{-1}(\varphi))}(k')$ is injective.
\end{proposition}

\begin{proof}
Let $\psi_1, \psi_2 \in \sL(\cX)(k')$ correspond to $G$-torsors $Q_1, Q_2 \to \Spec(k'\llbracket t \rrbracket)$ and $G$-equivariant maps $\gamma_1: Q_1 \to \widetilde{X}$ and $\gamma_2: Q_2 \to \widetilde{X}$, respectively, and assume that $\sL(\pi)(\psi_1) = \sL(\pi)(\psi_2) = \varphi$ and $\trop(\psi_1) = \trop(\psi_2)$. We need to prove that $\psi_1$ and $\psi_2$ are isomorphic, i.e., that there exists an isomorphism of $G$-torsors $\alpha: Q_1 \to Q_2$ such that $\gamma_1 = \gamma_2 \circ \alpha$. In fact, we prove the stronger statement that there exists a \emph{unique} such $\alpha$. 

To prove this stronger statement, by descent, it is enough to show the existence of a unique such $\alpha$ \'etale locally on $k'\llbracket t \rrbracket$. By \autoref{l:pullback-diag-torsors}, the $Q_i$ are isomorphic to pullbacks of torsors over $k'$, which can themselves be trivialized after base change to a finite field extension $k''$ of $k'$. Since $k'' \otimes_{k'} k'\llbracket t \rrbracket \cong k'' \llbracket t \rrbracket$, after replacing $k'$ by $k''$, we may therefore assume that the $Q_i$ are trivial $G$-torsors. Since $\trop(\psi_i)$ depends only on the isomorphism class of $\psi_i$, we may further assume $Q_1=Q_2=G\otimes_{k}k'\llbracket t \rrbracket$.

Next, the identity section of the $G$-torsor $Q_i$ then yields a lift $\widetilde{\psi}_i\in\sL(\widetilde{X})(k')$ of $\psi_i$. Then the map $\gamma_i\colon \Spec(k' \llbracket t \rrbracket [A])=Q_i\to\widetilde{X}=\Spec(k[F])$ satisfies
\[
\gamma_i^*(f)=\widetilde{\psi}_i^*(f) u^{\overline{f}} \in k' \llbracket t \rrbracket [A],
\]
where $u^{\overline{f}} \in k[A]$ is the monomial indexed by the image $\overline{f}$ of $f$ in $A$. Since $\trop(\widetilde{\psi}_i)=\trop(\psi_i)$, we see $\trop(\widetilde{\psi}_1)=\trop(\widetilde{\psi}_2)\in \widetilde{\sigma} \cap \widetilde{N}$. Thus, for all $f \in F$, the power series $\widetilde{\psi}_1^*(f)$ and $\widetilde{\psi}_2^*(f)$ are nonzero and have the same $t$-order of vanishing. It follows that there is a unique unit $g^{(f)} \in k'\llbracket t \rrbracket$ such that
\[
	\psi_1^*(f) = g^{(f)} \psi_2^*(f).
\]
Since $\sL(\pi)(\widetilde\psi_1) = \sL(\pi)(\widetilde\psi_2)$, we see $\widetilde\psi_1^*(p) = \widetilde\psi_2^*(p)$, and so $g^{(p)} = 1$ for all $p \in P$.

Thus the semigroup homomorphism $F \to k'\llbracket t \rrbracket^\times: f \mapsto g^{(f)}$ induces a group homomorphism $A \to k' \llbracket t \rrbracket^\times$, which corresponds to an element $g \in G(k'\llbracket t \rrbracket)$ and hence an automorphism $\alpha$ of the $G$-torsor $G\otimes_k k' \llbracket t \rrbracket$. By construction, $\widetilde{\psi_1} = \widetilde{\psi_2} \circ \alpha$, and so $\gamma_1 = \gamma_2 \circ \alpha$. Moreover, the uniqueness of $\alpha$ follows from the uniqueness of each $g^{(f)}$.
\end{proof}

We now complete the proof of \autoref{MainPropositionFibersOfTheMapOfArcs} by showing surjectivity of the map $\overline{(\sL(\pi)^{-1}(\varphi))}(k') \to \beta^{-1}(w)$.

\begin{proposition}
\label{prop:surj-lifts-trop}
The image of $\overline{(\sL(\pi)^{-1}(\varphi))}(k')$ under the map
\[
	\trop: \overline{\sL(\cX)}(k') \to \Hom(F, \N \cup \{\infty\})
\]
is equal to $\beta^{-1}(w)$.
\end{proposition}

\begin{proof}
Recall that in \autoref{lemmaTropOfStackArcMapsToTropOfGoodModuliSpaceArc}, we proved that the image of $\overline{(\sL(\pi)^{-1}(\varphi))}(k')$ under $\trop$ is contained in $\beta^{-1}(w)$. 

Let $\widetilde{w} \in \beta^{-1}(w)$. We will construct an arc $\psi \in \sL(\cX)(k')$ satisfying $\trop(\psi) = \widetilde{w}$ and $\sL(\pi)(\psi) = \varphi$. Let $\eta: \Spec(k'\llparenthesis t \rrparenthesis ) \to X$ be the generic point of $\varphi$, i.e., $\eta$ is the composition $\Spec(k'\llparenthesis t \rrparenthesis) \to \Spec(k'\llbracket t \rrbracket) \xrightarrow{\varphi} X$. Since $w \in \sigma \cap N$, we see that $\eta$ factors through $T \hookrightarrow X$. Thus $\eta$ is given by a group homomorphism $M = P^\gp \to k' \llparenthesis t \rrparenthesis^\times$. 

Given our inclusion $M\hookrightarrow\widetilde{M}$, we can choose a $\Z$-basis $f_1, \dots, f_r$ for $\widetilde{M} = F^\gp$ and $m_1, \dots, m_d \in \Z_{>0}$ such that $m_1 f_1, \dots, m_d f_d$ is a $\Z$-basis for $M$. Since $\trop(\varphi) = w$, the image of $m_i f_i$ under $M \to k' \llparenthesis t \rrparenthesis^\times$ has $t$-order of vanishing $\langle w, m_i f_i \rangle$, so can be written as $t^{\langle w, m_i f_i \rangle} g_i$ with $g_i \in k'\llbracket t \rrbracket^\times$.

Now set $R$ to be the $k'\llbracket t \rrbracket$-algebra
\[
	R = k'\llbracket t \rrbracket [x_1, \dots, x_d, x_{d+1}^{\pm 1}, \dots, x_r^{\pm 1}] / (x_1^{m_1} - g_1, \dots, x_d^{m_d} - g_d),
\]
and give $\Spec(R)$ the $G$-action obtained by letting $x_i$ have grading $\overline{f}_i$, where $\overline{f}_i$ is the image of $f_i$ in $A$. Since each $g_i$ is a unit in $k'\llbracket t \rrbracket$, we see $\Spec(R) \to \Spec(k'\llbracket t \rrbracket)$ is a $G$-torsor.

Let $\psi \in \sL(\cX)(k')$ be the arc corresponding to the $G$-torsor $\Spec(R)$ and the $G$-equivariant map $\Spec(R) \to \widetilde{X} = \Spec(k[F])$ defined by
\[
	F \to R: f \mapsto t^{\langle \widetilde{w}, f \rangle} \prod_{i=1}^r x_i^{c_i},
\]
for $f = \sum_{i=1}^r c_i f_i$ and $c_1, \dots, c_r \in \Z$. Note that this map is well-defined since $\langle \widetilde{w}, f \rangle \geq 0$ when $f \in F$, and each $x_i$ is a unit in $R$ as a consequence of the fact that each $g_i$ is a unit in $k' \llbracket t \rrbracket$. Note further that, since the $x_i$ are units, the map sends $f$ to an element whose $t$-order of vanishing is $\langle\widetilde{w},f\rangle$. As a result, any lift $\widetilde{\psi}\in\sL(\widetilde{X})(k'')$ obtained from a trivialization of the $G$-torsor after base change to $k''$, satisfies $\trop(\psi)=\trop(\widetilde{\psi})=\widetilde{w}$.


To finish the proof, we must show $\sL(\pi)(\psi) = \varphi$, i.e.~$\pi\circ\psi=\varphi$. Since $X$ is separated, it is enough to prove equality after precomposing by the generic point $\xi\colon\Spec(k'\llparenthesis t \rrparenthesis) \to \Spec(k'\llbracket t \rrbracket)$. But this follows from the fact that $\eta=\varphi\circ\xi$ and 
that for any $i \in \{1, \dots, d\}$, the image of $m_i f_i$ under the pullback of $\Spec(R) \to \widetilde{X} \xrightarrow{\widetilde{\pi}} X$ is equal to
\[
	t^{\langle \widetilde{w}, m_i f_i \rangle} x_i^{m_i} = t^{\langle w, m_i f_i \rangle} g_i.\qedhere
\]
\end{proof}

\section{Fibers of the maps of jets}\label{sectionFibersOfTheMapsOfJets}

Throughout this section let $d \in \N$, let $N \cong \Z^d$ be a lattice, let $T = \Spec(k[N^*])$ be the algebraic torus with co-character lattice $N$, let $\sigma$ be a pointed rational cone in $N_\R$, let $X$ be the affine $T$-toric variety associated to $\sigma$, let $\cX$ be the canonical stack over $X$, and let $\pi: \cX \to X$ be the canonical map.

In this section we will control the fibers of the maps
\[
	\sL_n(\pi): \sL_n(\cX) \to \sL_n(X)
\]
in the case where $\sigma$ is $d$-dimensional. In particular, we will prove the following.

\begin{theorem}
\label{fibersofhtemapsofjetsmainproposition}
Suppose that $\sigma$ is $d$-dimensional, and let $w \in \sigma \cap N$. Then there exist $n_w \in \N$ and $\Theta_w \in K_0(\Stack_k)$ and a sequence of finite type Artin stacks $\{\cF_n\}_{n \geq n_w}$ over $k$ such that
\begin{enumerate}

\item for each $n \geq n_w$,
\[
	\e(\cF_n) = \Theta_w \in K_0(\Stack_k),
\]
and

\item for each field extension $k'$ of $k$, each arc $\varphi \in \sL(X)(k')$ with $\trop(\varphi) = w$, and each $n \geq n_w$, we have
\[
	\sL_n(\pi)^{-1}(\theta_n(\varphi))_\red \cong \cF_n \otimes_k k'.
\]

\end{enumerate}
\end{theorem}

For the remainder of this section, we will assume that $\sigma$ is $d$-dimensional, and we will use the notation listed in \autoref{notationforcanonicalfantastackforfulldimensionalcone}.

\subsection{Algebraic groups and jets}

We begin by introducing some algebraic groups, which in \autoref{fibersAsStackQuotientOfGroupQuotient} below, will be used to express the fibers of each $\sL_n(\pi)$.

\begin{remark}
In what follows, for each $k$-algebra $R$, each $g_n \in \sL_n(\widetilde{T})(R)$, and each $f \in \widetilde{M}$, let $g_n^*(f) \in (R[t]/(t^{n+1}))^\times$ denote the image of $f$ under the pullback map $ k[\widetilde{M}] \to R[t]/(t^{n+1})$ corresponding to the jet $g_n: \Spec(R[t]/(t^{n+1})) \to \widetilde{T}$. We also use the analogous notation when $T$ and $M$ are in place of $\widetilde{T}$ and $\widetilde{M}$.
\end{remark}

\noindent For each $w \in \sigma \cap N$ and $n \in \N$, let $G_n^w$ be the sub-group-object of $\sL_n(\widetilde{T})$ given by,
\[
	G_n^w(R) = \{ g_n \in \sL_n(\widetilde{T})(R) \mid  \text{$ g_n^*(p)t^{\langle w, p \rangle} = t^{\langle w, p \rangle} \in R[t]/(t^{n+1})$ for all $p \in P$}\}.
\]
For each $\widetilde{w} \in \widetilde{\sigma} \cap \widetilde{N}$ and $n \in \N$, let $H_n^{\widetilde{w}}$ be the sub-group-object of $\sL_n(\widetilde{T})$ given by,
\[
	H_n^{\widetilde{w}}(R) = \{ g_n \in \sL_n(\widetilde{T})(R) \mid  \text{$ g_n^*(f) t^{\langle \widetilde{w}, f \rangle} = t^{\langle \widetilde{w}, f \rangle} \in R[t]/(t^{n+1})$ for all $f \in F$}\}.
\]
By definition, $H_n^{\widetilde{w}}$ is a sub-group-object of $G_n^{\beta(\widetilde{w})}$. We now show that these group objects are in fact algebraic groups.

\begin{proposition}
\label{groupobjectsrepresentedbyclosedsubgroups}
The sub-group-objects $G_n^w$ and $H_n^{\widetilde{w}}$ are represented by closed subgroups of $\sL_n(\widetilde{T})$.
\end{proposition}

\begin{proof}
For each $f \in F$, let $\sL_n(f): \sL_n(\widetilde{T}) \to \sL_n(\bG_m)$ be the map induced by the character $\widetilde{T} \to \bG_m$ corresponding to $f$. Then
\[
	G_n^w = \bigcap_{p \in P} (\theta^n_{n-\langle w, p \rangle} \circ \sL_n(p))^{-1}(1),
\]
and
\[
	H_n^{\widetilde{w}} = \bigcap_{f \in F} (\theta^n_{n-\langle \widetilde{w}, f \rangle} \circ \sL_n(f))^{-1}(1),
\]
where each $\theta_{n-n'}^n: \sL_{n}(\bG_m) \to \sL_{n-n'}(\bG_m)$ is the truncation morphism, each $1 \in \sL_{n-n'}(\bG_m)$ is the identity element, and by convention, if $n-n' < 0$, we set $\sL_{n-n'}(\bG_m)$ to be the single point group $\{1\}$.
\end{proof}

\begin{remark}
\label{JetSchemeOfGIsClosedSubgroupOfGnw}
Note that
\[
	\sL_n(G)(R) = \{g_n \in \sL_n(\widetilde{T})(R) \mid \text{$g_n^*(p) = 1 \in R[t]/(t^{n+1})$ for all $p \in P$}\},
\]
so for any $n \in \N$ and $w \in \sigma \cap N$, we have that $\sL_n(G)$ is a closed subgroup of $G_n^w$.
\end{remark}

We also prove the following characterization of the $H_n^{\widetilde{w}}$.

\begin{proposition}
\label{powerofadditivegroup}
Let $\widetilde{w} \in \widetilde{\sigma} \cap \widetilde{N}$, let $n \in \N$, and let $f_1, \dots, f_r$ be the minimal generators of $F$. If $n \geq \max_{i=1}^r( 2\langle \widetilde{w}, f_i\rangle - 1)$, then
\[
	H_n^{\widetilde{w}} \cong \bG_a^{\sum_{i = 1}^r \langle \widetilde{w}, f_i \rangle}
\]
as algebraic groups.
\end{proposition}

\begin{proof}
Let $R$ be a $k$-algebra. Then
\begin{align*}
	H_n^{\widetilde{w}}(R) &= \{ g_n \in \sL_n(\widetilde{T})(R) \mid  \text{$ g_n^*(f)t^{\langle \widetilde{w}, f \rangle} = t^{\langle \widetilde{w}, f \rangle} \in R[t]/(t^{n+1})$ for all $f \in F$}\} \\
	&=\{ g_n \in \sL_n(\widetilde{T})(R) \mid  \text{$ g_n^*(f_i)t^{\langle \widetilde{w}, f_i \rangle} = t^{\langle \widetilde{w}, f_i \rangle} \in R[t]/(t^{n+1})$ for all $i = 1, \dots, r$}\}\\
	&\cong \{ (g^{(i)})_i \in ((R[t](t^{n+1}))^\times)^r \mid \text{$g^{(i)} t^{\langle \widetilde{w}, f_i \rangle} = t^{\langle \widetilde{w}, f_i \rangle} \in R[t]/(t^{n+1})$ for all $i$}\}. 
\end{align*}
Since $n \geq 0$, the hypotheses guarantee that $n \geq \langle \widetilde{w}, f_i \rangle$ for all $i \in \{1, \dots, r\}$. Thus if $i \in \{1, \dots r\}$ and $g^{(i)} \in (R[t]/(t^{n+1}))^\times$, then
\[
	g^{(i)}t^{\langle \widetilde{w}, f_i \rangle} = t^{\langle \widetilde{w}, f_i \rangle} \in R[t]/(t^{n+1})
\]
if and only if
\[
	g^{(i)} = 1 + a_{n-\langle \widetilde{w}, f_i \rangle + 1} t^{n-\langle \widetilde{w}, f_i \rangle+1} + \dots + a_{n} t^n \in R[t]/(t^{n+1})
\]
for some $a_{n-\langle \widetilde{w}, f_i \rangle + 1}, \dots, a_n \in R$. The proposition thus follows from the fact that if $a_{n-\langle \widetilde{w}, f_i \rangle+1}, \dots, a_n$, $b_{n-\langle \widetilde{w}, f_i \rangle+1}, \dots, b_n \in R$, then in $R[t]/(t^{n+1})$,
\begin{align*}
	(1 + a_{n-\langle \widetilde{w}, f_i \rangle+1} t^{n-\langle \widetilde{w}, f_i \rangle+1} + \dots + a_n t^n) (1 + b_{n-\langle \widetilde{w}, f_i \rangle+1} t^{n-\langle \widetilde{w}, f_i \rangle+1} + \dots + b_n t^n) \\ = 1 + (a_{n-\langle \widetilde{w}, f_i \rangle+1}+b_{n-\langle \widetilde{w}, f_i \rangle+1}) t^{n-\langle \widetilde{w}, f_i \rangle+1} + \dots + (a_n+b_n) t^n
\end{align*}
because for any $m_1, m_2 \geq n-\langle \widetilde{w}, f_i \rangle+1$, we have 
\[
	m_1+m_2 \geq 2n-2\langle \widetilde{w}, f_i \rangle+2 = 2n+1 - (2\langle \widetilde{w}, f_i \rangle-1) \geq n+1.\qedhere
\]
\end{proof}

We finish this subsection with the next two propositions, which characterize the group quotients $G_n^w / \sL_n(G)$.

\begin{proposition}
\label{groupQuotientAsStabilizerOfJetInGoodModuliSpace}
Let $w \in \sigma \cap N$ and $n \in \N$. Then the group quotient $G_n^w / \sL_n(G)$ has functor of points given by
\[
	R \mapsto \{h_n \in \sL_n(T)(R) \mid \text{$h_n^*(p) t^{\langle w, p \rangle} =  t^{\langle w, p \rangle} \in R[t]/(t^{n+1})$ for all $p \in P$}\}.
\]
\end{proposition}
\begin{proof}
First, the sub-group-object is represented by a closed subgroup of $\sL_n(T)$ by an identical argument as in Proposition \ref{groupobjectsrepresentedbyclosedsubgroups}. Call this closed subgroup $H \subset \sL_n(T)$. We will show that $G_n^w / \sL_n(G) \cong H$ as schemes.

By definition of $H$ and $G_n^w$, the closed subgroup $G_n^w \subset \sL_n(\widetilde{T})$ is equal to the preimage of $H$ under the group homomorphism $\sL_n(\widetilde{T}) \to \sL_n(T)$ obtained by applying the functor $\sL_n$ to the group homomorphism $\widetilde{\pi}|_{\widetilde{T}}: \widetilde{T} \to T$. Thus we obtain a group homomorphism $G_n^w \to H$. Endow $H$ with the $G_n^w$-action obtained by left multiplication after $G_n^w \to H$.

For any $k$-algebra $R$ and $g_n \in G_n^w(R)$, we have that $g_n \in \sL_n(G)(R)$ if and only if $g_n^*(p) = 1$ for all $p \in P$, which is equivalent to $g_n$ being in the kernel of $G_n^w \to H$. Therefore the scheme-theoretic stabilizer of the identity $1 \in H(k)$ is equal to $\sL_n(G)$. Thus by \cite[III, 3, Proposition 5.2]{DemazureGabriel}, we have a locally closed embedding
\[
	G_n^w/\sL_n(G) \hookrightarrow H
\]
whose image, as a set, is equal to the image of the map of underlying sets $G_n^w \to H$. Since $k$ has characteristic 0 so $H$ is reduced, we only need to show that $G_n^w \to H$ is surjective on underlying sets, which would follow if we can show that $\sL_n(\widetilde{T}) \to \sL_n(T)$ is surjective on underlying sets. The latter follows immediately from \autoref{JetSchemeOfQuotientStackIsQuotientOfJetSchemes}, which implies that $\sL_n(\widetilde{T}) \to \sL_n(T)$ is a $\sL_n(G)$-torsor.
\end{proof}

\begin{proposition}
\label{GroupQuotientIsAffineSpace}
Let $w \in \sigma \cap N$. Then there exist $n'_w, j'_w \in \N$ such that for all $n \geq n'_w$, we have an isomorphism of schemes
\[
	G_n^w / \sL_n(G) \cong \bA^{j'_w}_k.
\]
\end{proposition}

\begin{proof}
Let $p_1, \dots, p_s$ be a set of generators of the semigroup $P$, set
\[
	n'_w = \max_{1\leq i\leq s}( 2 \langle w, p_i \rangle - 1),
\]
and $m = \max_{1\leq i\leq s} \langle w, p_i \rangle$. Consider the affine space $\bA_k^{ms}$ with coordinates $(x^{(i)}_\ell )_{i \in \{1, \dots, s\}, \ell \in \{1, \dots, m\}}$. Let $V$ be the linear subspace of $\bA_k^{ms}$ defined by the vanishing of all $x^{(i)}_\ell$ for $\ell > \langle w, p_i \rangle$ and the vanishing of all $\sum_{i =1}^s m_i x^{(i)}_\ell$ for all $\ell \in \{1, \dots, m\}$ and all $m_1, \dots, m_s \in \Z$ such that
\[
	\sum_{i =1}^s m_i  p_i = 0 \in P^{\gp} = M.
\]
Set $j'_w = \dim V$. It suffices to show that for all $n \geq n'_w$, we have that $G_n^w / \sL_n(G) \cong V$ as schemes.

Let $n \geq n'_w$, and let $H$ be the closed subgroup of $\sL_n(T)$ representing the sub-group-object in the statement of \autoref{groupQuotientAsStabilizerOfJetInGoodModuliSpace}. By \autoref{groupQuotientAsStabilizerOfJetInGoodModuliSpace}, it is sufficient to prove that $H \cong V$ as schemes. 

Let $H \to \bA^{ms}_k$ be the morphism that, for each $k$-algebra $R$, takes $h_n \in H(R)$ to $(a^{(i)}_{n-\ell+1})_{i,\ell} \in R^{ms} = \bA^{ms}_k(R)$, where for all $i,\ell$ we have that $a^{(i)}_{n-\ell+1}$ is the coefficient of $t^{n-\ell+1}$ in $p_i(h_n) \in R[t]/(t^{n+1})$. This morphism $H \to \bA_k^{ms}$ factors through an isomorphism $H \xrightarrow{\sim} V$ by the definition of $H$, the construction of $V$, and the fact that $n \geq n'_w$ implies that for any $k$-algebra $R$, any $m_1, \dots, m_s \in \Z$, and any $(a^{(i)}_{n-\ell+1})_{i,\ell} \in R^{ms}$, we have that
\begin{align*}
	\prod_{i = 1}^s &\left( 1 + a^{(i)}_{n-m+1} t^{n-m+1} + \dots + a^{(i)}_n t^n \right)^{m_i} \\
	&= 1 + \left( \sum_{i = 1}^s m_i  a^{(i)}_{n-m+1}  \right) t^{n-m+1} + \dots + \left( \sum_{i = 1}^s m_i  a^{(i)}_{n}  \right) t^{n} \in R[t]/(t^{n+1}).\qedhere
\end{align*}
\end{proof}

\subsection{Components of the fibers}

In this subsection, we will control fibers of each $\sL_n(\pi): \cX \to X$ by controlling the connected components of the fibers of each map $\sL_n(\widetilde{\pi}): \sL_n(\widetilde{X}) \to \sL_n(X)$. In particular we will show that for $n$ sufficiently large and $\varphi \in \sL(X)$ with $\trop(\varphi)=w \in \sigma \cap N$, the connected components of $\sL_n(\widetilde{\pi})^{-1}(\theta_n(\varphi))$ are indexed by $\beta^{-1}(w)$. To do this, we will define analogs of the map $\trop$ for the jet schemes $\sL_n(\widetilde{X})$.

For any $n \in \N$, let $\N_n = \{0, 1, \dots, n, \infty\}$ with the monoid structure making
\[
	\N \cup \{\infty\} \to \N_n: \ell \mapsto \begin{cases} \ell, & \ell \leq n \\ \infty, &\ell > n  \end{cases}
\]
a map of monoids. For any $\widetilde{w} \in \Hom(F, \N \cup \{\infty\})$, let $\widetilde{w}_n \in \Hom(F, \N_n)$ be the composition of $\widetilde{w}: F \to \N \cup \{\infty\}$ with $\N \cup \{\infty\} \to \N_n$.

\begin{remark}
In what follows, for each $k$-algebra $R$, each $\widetilde{\psi}_n \in \sL_n(\widetilde{X})(R)$, and each $f \in F$, let $\widetilde{\psi}_n^*(f) \in R[t]/(t^{n+1})$ denote the image of $f$ under the pullback map $ k[F] \to R[t]/(t^{n+1})$ corresponding to the jet $\widetilde{\psi}_n: \Spec(R[t]/(t^{n+1})) \to \widetilde{X}$. We use the analogous notation when $X$ and $P$ are in place of $\widetilde{X}$ and $F$.
\end{remark}

For any field extension $k'$ of $k$ and $\widetilde{\psi}_n \in \sL_n(\widetilde{X})(k')$, define
\[
	\trop_n(\widetilde{\psi}_n) \in \Hom(F, \N_n)
\]
to be the map taking each $f \in F$ to the $t$-order of vanishing of $\widetilde{\psi}_n^*(f) \in k'[t]/(t^{n+1})$. Also define the map
\[
	\trop_n: \sL_n(\widetilde{X}) \to \Hom(F, \N_n)
\]
by considering each $\widetilde{\psi}_n \in \sL(\widetilde{X})$ as a point valued in its residue field.

\begin{remark}
As a direct consequence of the definition of $\trop_n$ is compatible with field extensions. In other words, for any field extension $k'$ of $k$ and $\widetilde{\psi}_n \in \sL_n(\widetilde{X})(k')$, we have that $\trop_n(\widetilde{\psi}_n)$ is equal to $\trop_n$ applied to the image of $\widetilde{\psi}_n$ in the underlying set of $\sL_n(\widetilde{X})$.
\end{remark}

\begin{remark}
For any $\widetilde{\psi} \in \sL(\widetilde{X})(k')$,
\[
	(\trop(\widetilde{\psi}))_n = \trop_n(\theta_n(\widetilde{\psi})).
\]
Note that because $\widetilde{X}$ is an affine space and thus is smooth, all jets of $\widetilde{X}$ are truncations of arcs of $\widetilde{X}$, so the above equality in fact determines $\trop_n$.
\end{remark}

We next stratify the fiber of $\sL_n(\widetilde{\pi})$ according to the value of $\trop_n$. We show that for $n$ sufficiently large, each stratum is a union of connected components.

\begin{lemma}
\label{uppersemicontinuouslemmafortrop}
For each $f \in F$, the map 
\[
	\sL_n(\widetilde{X}) \to \N_n: \widetilde{\psi}_n \mapsto (\trop_n(\widetilde{\psi}_n))(f)
\]
is upper semi-continuous.
\end{lemma}

\begin{proof}
Since $\widetilde{X}$ is affine, $\sL_n(\widetilde{X}) = \Spec(R)$ for some $k$-algebra $R$. Let $\Psi$ be the universal family of $\sL_n(\widetilde{X})$, i.e., let
\[
	\Psi: \Spec(R[t]/(t^{n+1})) \to \widetilde{X}
\] 
be the $R$-valued jet corresponding to the identity $\Spec(R) = \sL_n(\widetilde{X})$, and let $\Psi^*(f) \in R[t]/(t^{n+1})$ be the result of pulling back $f$ along $\Psi$. Write
\[
	\Psi^*(f) = a_0 + \dots + a_n t^n,
\]
where $a_0, \dots, a_n \in R$. Then the jets $\widetilde{\psi}_n \in \sL_n(\widetilde{X})$ with $(\trop_n(\widetilde{\psi}_n))(f) \geq \ell$ are exactly the points of $\Spec(R)$ where $a_0, \dots, a_{\ell-1}$ vanish.
\end{proof}

\begin{proposition}
\label{tropislocallyconstantonfiber}
Let $w \in \sigma \cap N$. Then there exists some $n_0 \in \N$ such that for any $n \geq n_0$, any field extension $k'$ of $k$, and any $\varphi \in \sL(X)(k')$ with $\trop(\varphi) =w$, we have that the restriction of $\trop_n: \sL_n(\widetilde{X}) \to \Hom(F, \N_n)$ to the fiber $\sL_n(\widetilde{\pi})^{-1}(\theta_n(\varphi))$ is locally constant.
\end{proposition}

\begin{proof}
Let $f \in F$. We will show that there exists some $n_f \in \N$ such that for any $n \geq n_f$, any field extension $k'$ of $k$, and any $\varphi \in \sL(X)(k')$ with $\trop(\varphi) =w$, we have that the restriction of 
\[
	\sL_n(\widetilde{X}) \to \N_n : \widetilde{\psi}_n \mapsto (\trop_n(\widetilde{\psi}_n))(f)
\]
to the fiber $\sL_n(\widetilde{\pi})^{-1}(\theta_n(\varphi))$ is locally constant. 

By \autoref{elementofFaddstogetelementofP}, there exists some $f' \in F$ such that $f + f' \in P$. Set
\[
	n_f = \langle w, f + f' \rangle.
\]
Let $n \geq n_f$, let $k'$ be a field extension of $k$, and let $\varphi \in \sL(X)(k')$ with $\trop(\varphi)=w$. Since $\sL_n(\widetilde{\pi}): \sL_n(\widetilde{X}) \to \sL_n(X)$ is finite type, it is sufficient to show that on any irreducible component $C$ of the fiber $\sL_n(\widetilde{\pi})^{-1}(\theta_n(\varphi))$, the map
\[
	\alpha: C \to \N_n : \widetilde{\psi}_n \mapsto (\trop_n(\widetilde{\psi}_n))(f)
\]
is constant. By \autoref{uppersemicontinuouslemmafortrop}, $\alpha$ and the map	
\[
	\gamma: C \to \N_n: \widetilde{\psi}_n \mapsto (\trop_n(\widetilde{\psi}_n))(f')
\]
are upper semi-continuous. Also for all $\widetilde{\psi}_n \in C$
\[
	\alpha(\widetilde{\psi}_n) + \gamma(\widetilde{\psi}_n) = (\trop_n(\widetilde{\psi}_n))(f + f') = \langle w, f + f'\rangle_n,
\]
where $\langle w, f + f'\rangle_n$ is the image of $\langle w, f+f'\rangle$ in $\N_n$. Thus the sum of $\alpha$ and $\gamma$ is a constant function. Furthermore, because $n \geq \langle w, f + f' \rangle$, the sum of $\alpha$ and $\gamma$ is not equal to $\infty$. Therefore $\alpha$ is constant by upper semi-continuity of $\alpha$ and $\gamma$ and the fact that $C$ is irreducible.

Now the proposition is obtained by taking $n_0$ to be larger than all $n_f$ as $f$ vary over the minimal generators of $F$.
\end{proof}

For any $n \in \N$, any $\widetilde{w} \in \widetilde{\sigma} \cap \widetilde{N}$, any field extension $k'$ of $k$, and any $\varphi \in \sL(X)(k')$, let $C_n^{\widetilde{w}}(\varphi)$ denote the locus in $\sL_n(\widetilde{\pi})^{-1}(\theta_n(\varphi))$ where $\trop_n$ is equal to $\widetilde{w}_n$. We will be interested in the case where $C_n^{\widetilde{w}}(\varphi)$ is a union of connected components of $\sL_n(\widetilde{\pi})^{-1}(\theta_n(\varphi))$. In that case, we will give $C_n^{\widetilde{w}}(\varphi)$ its reduced scheme structure.

\begin{proposition}
\label{ComponentsOfFiberPiTildeGroupAction}
Let $w \in \sigma \cap N$. Then there exists some $n_1 \in \N$ such that for any $n \geq n_1$, any field extension $k'$ of $k$, any $\varphi \in \sL(X)(k')$ with $\trop(\varphi) = w$, and for any $\widetilde{w} \in \beta^{-1}(w)$, 
\begin{enumerate}
\item\label{DecomposeFiberIntoCnw} $C_n^{\widetilde{w}}(\varphi)$ is a union of connected components of $\sL_n(\widetilde{\pi})^{-1}(\theta_n(\varphi))$, and we have an isomorphism of schemes
\[
	\sL_n(\widetilde{\pi})^{-1}(\theta_n(\varphi))_{\red} \cong \bigsqcup_{\widetilde{w}' \in \beta^{-1}(w)} C_n^{\widetilde{w}'}(\varphi),
\]


\item\label{ComponentsOfFiberPiTildeInvariantUnderAction} $C_n^{\widetilde{w}}(\varphi)$ is invariant under the action of $G_n^w \otimes_k k'$ on $\sL_n(\widetilde{X}) \otimes_k k'$,

\item\label{TransitiveActionOnComponents} for each field extension $k''$ of $k'$, the action of $(G_n^w \otimes_k k')(k'')$ on $C_n^{\widetilde{w}}(\varphi)(k'')$ is transitive, and

\item\label{StabilizerOfActionOnComponents} for each field extension $k''$ of $k'$, the scheme-theoretic stabilizer of any $k''$-point of $C_n^{\widetilde{w}}(\varphi)$ under the $G_n^w \otimes_k k'$-action is equal to $H_n^{\widetilde{w}} \otimes_k k''$.

\end{enumerate}
\end{proposition}

\begin{remark}
\label{RemarkGnwActionOJetSchemeOfCover}
In the statement of \autoref{ComponentsOfFiberPiTildeGroupAction} above, the action of $G_n^w$ on $\sL_n(\widetilde{X})$ is the one induced by the inclusion $G_n^w \hookrightarrow \widetilde{T}$ and the functor $\sL_n$ applied to the toric action $\widetilde{T} \times_k \widetilde{X} \to \widetilde{X}$.
\end{remark}

\begin{proof}
Let $f_1, \dots, f_r$ be the minimal generators of $F$, and let $p_1, \dots, p_s$ be a set of generators for the semigroup $P$. Let $n_0 \in \N$ be as in the statement of \autoref{tropislocallyconstantonfiber}, and let $n_1 \geq n_0$ be such that $n_1 \geq \langle w, p_i \rangle$ for all $i \in \{1, \dots, s\}$ and $n_1 \geq \langle \widetilde{w}, f_i \rangle$ for all $i \in \{1, \dots, r\}$ and $\widetilde{w} \in \beta^{-1}(w)$. Note that we can choose such an $n_1$ by \autoref{fiberabovewisfiniteset}.

Let $n \geq n_1$, let $k'$ be a field extension of $k$, and let $\varphi \in \sL(X)(k')$ with $\trop(\varphi)= w$. We begin by proving the first part of the proposition.

\begin{enumerate}

\item By our choice of $n_0$ and \autoref{tropislocallyconstantonfiber}, $C_n^{\widetilde{w}}(\varphi)$ is a union of connected components of $\sL_n(\widetilde{\pi})^{-1}(\theta_n(\varphi))$, so it suffices to prove $\sL_n(\widetilde{\pi})^{-1}(\theta_n(\varphi))$ and $\bigsqcup_{\widetilde{w}' \in \beta^{-1}(w)} C_n^{\widetilde{w}'}(\varphi)$ are equal as sets.

Let $\widetilde{w} \in \widetilde{\sigma} \cap \widetilde{N}$. We first show that if $C_n^{\widetilde{w}}\neq\emptyset$, then $\widetilde{w} \in \beta^{-1}(w)$. If $\widetilde{\psi}_n \in \sL_n(\widetilde{\pi})^{-1}(\theta_n(\varphi))$, then for all $i \in \{1, \dots, s\}$,
\[
	(\trop_n(\widetilde{\psi}_n))(p_i) = \langle w, p_i \rangle_n,
\]
where $\langle w, p_i \rangle_n$ is the image of $\langle w, p_i \rangle$ in $\N_n$. Since $n \geq n_1 \geq \langle w, p_i \rangle$, this implies that if $\widetilde{\psi}_n \in C_n^{\widetilde{w}}$, then $\langle \widetilde{w}, p_i \rangle = \langle w, p_i \rangle$ for all $i \in \{1, \dots, s\}$, and so $\beta(\widetilde{w}) = w$. 

Having shown $C_n^{\widetilde{w}}(\varphi)=\emptyset$ whenever $\widetilde{w} \notin \beta^{-1}(w)$, we need only show that if $\widetilde{w}_1, \widetilde{w}_2$ are distinct elements of $\beta^{-1}(w)$, then $(\widetilde{w}_1)_n \neq (\widetilde{w}_2)_n$. This follows from the fact that for each $i \in \{1, \dots, r\}$ and $\widetilde{w} \in \beta^{-1}(w)$,
\[
	\langle \widetilde{w}, f_i \rangle \leq n_1 \leq n.
\]

\setcounter{ResumeEnumerate}{\value{enumi}}
\end{enumerate}

\noindent For the rest of this proof, let $\widetilde{w} \in \beta^{-1}(w)$ and set the following notation: 
for each field extension $k''$ of $k'$, let $\varphi_{n, k''} \in \sL_n(X)(k'')$ be the composition 
\[
	\Spec(k''[t]/(t^{n+1})) \to \Spec(k'[t]/(t^{n+1})) \xrightarrow{\theta_n(\varphi)} X,
\] 
where the map $\Spec(k''[t]/t^{n+1}) \to \Spec(k'[t]/t^{n+1})$ is given by the $k'$-algebra map $k'[t]/(t^{n+1}) \to k''[t]/(t^{n+1}): t \mapsto t$. Note that the $k''$-points of $C_n^{\widetilde{w}}(\varphi)$ are precisely those $\widetilde{\psi}_n: \Spec(k''[t]/(t^{n+1})) \to \widetilde{X}$ such that $\trop_n(\widetilde{\psi}_n) = \widetilde{w}_n$ and such that the composition $ \Spec(k''[t]/(t^{n+1})) \xrightarrow{\widetilde{\psi}_n} \widetilde{X} \to X$ is equal to $\varphi_{n,k''}$. We now prove the remaining parts of the proposition.

\begin{enumerate}
\setcounter{enumi}{\value{ResumeEnumerate}}

\item Since $C_n^{\widetilde{w}}(\varphi)$ is reduced by definition, it suffices to show that for each field extension $k''$ of $k'$, we have that $C_n^{\widetilde{w}}(\varphi)(k'')$ is invariant under the action of $G_n^w(k'')$ on $\sL_n(\widetilde{X})(k'')$.

Let $k''$ be a field extension of $k'$, let $\widetilde{\psi}_n \in C_n^{\widetilde{w}}(\varphi)(k'')$, and let $g_n \in G_n^w(k'')$. Then for all $f \in F$,
\[
	(g_n \cdot \widetilde{\psi}_n)^*(f) = g_n^*(f) \widetilde{\psi}_n^*(f)
\]
has the same $t$-order of vanishing as $\widetilde{\psi}_n^*(f)$ because $g_n^*(f) \in k''[t]/(t^{n+1})$ is a unit. Thus 
\[
	\trop_n(g_n \cdot \widetilde{\psi}_n) = \trop_n(\widetilde{\psi}_n) = \widetilde{w}_n.
\]
We also have that for all $p \in P$,
\[
	(g_n \cdot \widetilde{\psi}_n)^*(p) = g_n^*(p) \widetilde{\psi}_n^*(p) = g_n^*(p) \varphi_{n,k''}^*(p) = \varphi_{n,k''}^*(p),
\]
where the last equality follows from the definition of $G_n^w$ and the fact that $\trop(\varphi) = w$ implies that $\varphi_{n,k''}^*(p)$ is divisible by $t^{\langle w, p \rangle}$. Therefore the composition $\Spec(k''[t]/(t^{n+1})) \xrightarrow{g_n \cdot \widetilde{\psi}_n} \widetilde{X} \to X$ is equal to $\varphi_{n,k''}$. Thus
\[
	g_n \cdot \widetilde{\psi}_n \in C_n^{\widetilde{w}}(\varphi)(k'').
\]

\item Let $k''$ be a field extension of $k'$, and let $\widetilde{\psi}_n, \widetilde{\psi}'_n \in C_n^{\widetilde{w}}(\varphi)(k'')$. We will first show that there exists some $g_n \in \sL_n(\widetilde{T})(k'')$ such that $g_n \cdot \widetilde{\psi}_n = \widetilde{\psi}'_n$.

Let $f_1, \dots, f_r$ be the minimal generators of $F$. For each $i \in \{1, \dots, r\}$, we have that $\widetilde{\psi}_n^*(f_i)$ and $(\widetilde{\psi}'_n)^*(f_i)$ have the same $t$-order of vanishing because $\trop_n(\widetilde{\psi}_n) = \widetilde{w} = \trop_n(\widetilde{\psi}'_n)$. Thus there exists a unit $g^{(i)} \in k[t]/(t^{n+1})$ such that
\[
	g^{(i)} \widetilde{\psi}_n^*(f_i) = (\widetilde{\psi}'_n)^*(f_i).
\]
Letting $g_n \in \sL_n(\widetilde{T})(k'')$ be such that $g_n^*(f_i) = g^{(i)}$
for all $i \in \{1, \dots, r\}$,
\[
	g_n \cdot \widetilde{\psi}_n = \widetilde{\psi}'_n.
\]

Now it suffices to show that if $g_n \in \sL_n(\widetilde{T})(k'')$ and $\widetilde{\psi}_n$, $g_n \cdot \widetilde{\psi}_n \in C_n^{\widetilde{w}}(\varphi)(k'')$, then $g_n \in G_n^w(k'')$. For all $p \in P$,
\[
	g_n^*(p) \varphi_{n,k''}^*(p) = g_n^*(p) \widetilde{\psi}_n^*(p) = (g_n \cdot \widetilde{\psi}_n)^*(p) = \varphi_{n,k''}^*(p),
\]
so because $\varphi_{n,k''}^*(p)$ is a unit multiple of $t^{\langle w, p \rangle}$, we have that $g_n \in G_n^w(k'')$ by definition.

\item Let $k''$ be a field extension of $k'$, let $\widetilde{\psi}_n \in C_n^{\widetilde{w}}(\varphi)(k'')$, and for any $k''$-algebra $R$, let $\widetilde{\psi}_{n,R} \in C_n^{\widetilde{w}}(\varphi)(R)$ be the composition
\[
	\Spec(R[t]/(t^{n+1})) \to \Spec(k''[t]/(t^{n+1})) \xrightarrow{\widetilde{\psi}_n} \widetilde{X}.
\]
Let $g_n \in G_n^w(R)$. Then for each $f \in F$, we have that $\widetilde{\psi}_{n,R}^*(f) \in R[t]/(t^{n+1})$ is the product of a unit in $k''[t]/(t^{n+1})$ and $t^{\langle \widetilde{w}, f\rangle}$, so
\[
	g_n^*(f) \widetilde{\psi}_{n,R}^*(f) = \widetilde{\psi}_{n,R}^*(f) \iff g_n^*(f)t^{\langle \widetilde{w}, f\rangle} = t^{\langle \widetilde{w}, f\rangle}.
\]
Therefore $g_n$ is in the stabilizer of $\widetilde{\psi}_n$ if and only if $g_n \in H_n^{\widetilde{w}}(R)$.\qedhere
\end{enumerate}
\end{proof}

In the next proposition, we use \autoref{MainPropositionFibersOfTheMapOfArcs} and \autoref{ComponentsOfFiberPiTildeGroupAction} to control the reduced fibers of each $\sL_n(\pi): \sL_n(\cX) \to \sL_n(X)$.

\begin{proposition}
\label{fibersAsStackQuotientOfGroupQuotient}
Let $w \in \sigma \cap N$. Then there exists some $n_2 \in \N$ such that for any $n \geq n_2$, any field extension $k'$ of $k$, and any $\varphi \in \sL(X)(k')$ with $\trop(\varphi) = w$, we have
\[
	\sL_n(\pi)^{-1}(\theta_n(\varphi))_\red \cong \left(\bigsqcup_{\widetilde{w} \in \beta^{-1}(w)} [(G_n^w/\sL_n(G)) / H_n^{\widetilde{w}}]\right) \otimes_k k',
\]
where 
$H_n^{\widetilde{w}}$ acts on $G_n^w/\sL_n(G)$ via the group homomorphism $H_n^{\widetilde{w}} \hookrightarrow G_n^w \to G_n^w/\sL_n(G)$ and left multiplication.
\end{proposition}

\begin{proof}
Let $f_1, \dots, f_r$ be the minimal generators of $F$. Let $n_1$ be as in the statement of \autoref{ComponentsOfFiberPiTildeGroupAction}, and let $n_2 \geq n_1$ be such that $n_2 \geq \max_{i=1}^r( 2\langle \widetilde{w}, f_i\rangle - 1)$ for all $\widetilde{w} \in \beta^{-1}(w)$. Note that we can choose such an $n_2$ by \autoref{fiberabovewisfiniteset}.

Let $n \geq n_2$, let $k'$ be a field extension of $k$, and let $\varphi \in \sL(X)(k')$ be such that $\trop(\varphi) = w$. By \autoref{JetSchemeOfQuotientStackIsQuotientOfJetSchemes}, there exists an isomorphism $\sL_n(\cX) \xrightarrow{\sim} [\sL_n(\widetilde{X}) / \sL_n(G)]$ such that the following diagram commutes:
\begin{center}
\begin{tikzcd}
\sL_n(\widetilde{X}) \arrow[r] \arrow[rd] \arrow[rr, bend left, "\sL_n(\widetilde{\pi})"] & \sL_n(\cX) \arrow[Isom,d] \arrow[r, "\sL_n(\pi)"] & \sL_n(X)\\
& \phantom{} [\sL_n(\widetilde{X}) / \sL_n(G)]
\end{tikzcd}
\end{center}
Therefore
\[
	\sL_n(\pi)^{-1}(\theta_n(\varphi)) \cong [\sL_n(\widetilde{\pi})^{-1}(\theta_n(\varphi)) / (\sL_n(G) \otimes_k k')],
\]
where $\sL_n(G) \otimes_k k'$ acts on $\sL_n(\widetilde{\pi})^{-1}(\theta_n(\varphi))$ by restriction of its action on $\sL_n(\widetilde{X}) \otimes_k k'$, which itself is the action induced by the inclusion $\sL_n(G) \hookrightarrow \sL_n(\widetilde{T})$ and the functor $\sL_n$ applied to the toric action $\widetilde{T} \times_k \widetilde{X} \to \widetilde{X}$. Thus by \autoref{JetSchemeOfGIsClosedSubgroupOfGnw} and \autoref{RemarkGnwActionOJetSchemeOfCover}, the action of $\sL_n(G) \otimes_k k'$ on $\sL_n(\widetilde{\pi})^{-1}(\theta_n(\varphi))$ is the restriction of the action on $\sL_n(\widetilde{X}) \otimes_k k'$ induced by the inclusion $\sL_n(G) \otimes_k k' \hookrightarrow G_n^w \otimes_k k'$ and the action of $G_n^w \otimes_k k'$ on $\sL_n(\widetilde{X}) \otimes_k k'$. Thus \autoref{ComponentsOfFiberPiTildeGroupAction}(\ref{ComponentsOfFiberPiTildeInvariantUnderAction}) implies that for all $\widetilde{w} \in \beta^{-1}(w)$, we have $C_n^{\widetilde{w}}(\varphi)$ is invariant under the $\sL_n(G) \otimes_k k'$ action, so by \autoref{ComponentsOfFiberPiTildeGroupAction}(\ref{DecomposeFiberIntoCnw}),
\[
	\sL_n(\pi)^{-1}(\theta_n(\varphi))_\red \cong \bigsqcup_{\widetilde{w} \in \beta^{-1}(w)} [C_n^{\widetilde{w}}(\varphi) / (\sL_n(G) \otimes_k k') ].
\]

Let $\widetilde{w} \in \beta^{-1}(w)$. It will be sufficient to prove that
\[
	[C_n^{\widetilde{w}}(\varphi) / (\sL_n(G) \otimes_k k') ] \cong [(G_n^w/\sL_n(G)) / H_n^{\widetilde{w}}] \otimes_k k'.
\]
We begin by establishing that $C_n^{\widetilde{w}}(\varphi) / (\sL_n(G') \otimes_k k')$ is an affine scheme with a $k'$-point. Since $G$ is a diagonalizable group scheme, we have $G\cong T' \times_k G'$ where $T'$ is a torus and $G'$ is a finite group. This yields an identification
\[
	\sL_n(G) \otimes_k k' \cong (\sL_n(T') \times_k \sL_n(G')) \otimes_k k'.
\]
Note that $\sL_n(G')\cong G'$ since $G'$ is a finite group. By \autoref{powerofadditivegroup}, our choice of $n_2$, and the fact that $k$ has characteristic 0, we then see
\[
	\{1\}=\sL_n(G') \cap H_n^{\widetilde{w}} \subset G_n^w.
\]
Thus \autoref{ComponentsOfFiberPiTildeGroupAction}(\ref{StabilizerOfActionOnComponents}) implies that $\sL_n(G') \otimes_k k'$ acts freely on $C_n^{\widetilde{w}}(\varphi)$. Note that $C_n^{\widetilde{w}}(\varphi)$ is affine because $\sL_n(\widetilde{X})$ and $\sL_n(X)$ are affine, so
\[
C_n^{\widetilde{w}}(\varphi) / (\sL_n(G') \otimes_k k')\longrightarrow [C_n^{\widetilde{w}}(\varphi) / (\sL_n(G) \otimes_k k') ]
\]
is a $(\sL_n(T') \otimes_k k')$-torsor and the source is an affine scheme. By \autoref{MainPropositionFibersOfTheMapOfArcs}, there exists some $\psi \in (\sL(\pi)^{-1}(\varphi))(k')$ with $\trop(\psi) = \widetilde{w}$, so
\[
	\theta_n(\psi) \in [C_n^{\widetilde{w}}(\varphi) / (\sL_n(G) \otimes_k k') ](k').
\]
Since $\sL_n(T')$ is a special group by \autoref{jetSchemeOfSpecialGroupIsSpecial}, $\theta_n(\psi)$ lifts to a $k'$-point of $C_n^{\widetilde{w}}(\varphi) / (\sL_n(G') \otimes_k k')$.  

Next, by \autoref{ComponentsOfFiberPiTildeGroupAction}(\ref{ComponentsOfFiberPiTildeInvariantUnderAction}, \ref{TransitiveActionOnComponents}, \ref{StabilizerOfActionOnComponents}), the group $(G_n^w / \sL_n(G')) \otimes_k k'$ acts transitively on $C_n^{\widetilde{w}}(\varphi) / (\sL_n(G') \otimes_k k')$ and each $k'$-point has stabilizer $(H_n^{\widetilde{w}} / (H_n^{\widetilde{w}} \cap \sL_n(G')) \otimes_k k'$, so \cite[III, 3, Proposition 5.2]{DemazureGabriel} gives an $(\sL_n(T') \otimes_k k')$-equivariant isomorphism
\begin{align*}
	C_n^{\widetilde{w}}(\varphi) / (\sL_n(G') \otimes_k k') &\cong \left( C_n^{\widetilde{w}}(\varphi) / (\sL_n(G') \otimes_k k') \right)_\red \\
	&\cong \left( \left(G_n^w / \sL_n(G')\right)/\left( H_n^{\widetilde{w}} / (H_n^{\widetilde{w}} \cap \sL_n(G'))\right) \right)\otimes_k k'\\
	&\cong [(G_n^w / H_n^{\widetilde{w}}) / \sL_n(G')] \otimes_k k',
\end{align*}
where the last isomorphism holds since $\sL_n(G') \cap H_n^{\widetilde{w}} = \{1\}$. Taking the quotient by $\sL_n(T') \otimes_k k'$, we obtain
\[
	[C_n^{\widetilde{w}}(\varphi) / (\sL_n(G) \otimes_k k') ] \cong [(G_n^w / H_n^{\widetilde{w}}) / \sL_n(G)] \otimes_k k' \cong [(G_n^w/\sL_n(G)) / H_n^{\widetilde{w}}] \otimes_k k'.\qedhere
\]
\end{proof}

We may now complete the proof of \autoref{fibersofhtemapsofjetsmainproposition}.

\begin{proof}[Proof of \autoref{fibersofhtemapsofjetsmainproposition}]
Let $w \in \sigma \cap N$, let $f_1, \dots, f_r$ be the minimal generators of $F$, let $n'_w$ and $j'_w$ be as in the statement of \autoref{GroupQuotientIsAffineSpace}, and let $n_2$ be as in the statement of \autoref{fibersAsStackQuotientOfGroupQuotient}. Recalling that $\beta^{-1}(w)$ is a finite set by \autoref{fiberabovewisfiniteset}, set
\[
	n_w = \max \{ n'_w, n_2, 2 \langle \widetilde{w}, f_i \rangle - 1 \mid i \in \{1, \dots, r\}, \widetilde{w} \in \beta^{-1}(w)  \},
\]
and
\[
	\Theta_w = \sum_{\widetilde{w} \in \beta^{-1}(w)} \bL^{j'_w - \sum_{i=1}^r \langle \widetilde{w}, f_i \rangle} \in K_0(\Stack_k),
\]
and for each $n \geq n_w$, set
\[
	\cF_n = \bigsqcup_{\widetilde{w} \in \beta^{-1}(w)} [(G_n^w/\sL_n(G)) / H_n^{\widetilde{w}}].
\]
We now finish proving each part of \autoref{fibersofhtemapsofjetsmainproposition} separately.
\begin{enumerate}

\item For all $n \geq n_w$,
\begin{align*}
	\e(\cF_n) &= \sum_{\widetilde{w} \in \beta^{-1}(w)} \e([(G_n^w/\sL_n(G)) / H_n^{\widetilde{w}}])\\
	&= \sum_{\widetilde{w} \in \beta^{-1}(w)} \e(G_n^w / \sL_n(G)) \bL^{-\sum_{i=1}^r \langle \widetilde{w}, f_i \rangle}
	= \sum_{\widetilde{w} \in \beta^{-1}(w)} \bL^{j'_w - \sum_{i=1}^r \langle \widetilde{w}, f_i \rangle}
	= \Theta_w,
\end{align*}
where the second equality follows from \autoref{powerofadditivegroup} and the fact that $\bG_a$ is a special group, and the third equality follows from \autoref{GroupQuotientIsAffineSpace}.

\item This is \autoref{fibersAsStackQuotientOfGroupQuotient}, i.e., it follows from our choice of $n_2$ and each $\cF_n$.\qedhere
\end{enumerate}
\end{proof}

\section{Gorenstein measure and toric varieties}\label{sectionGorensteinMeasureToricVariety}

Let $d \in \N$, let $N \cong \Z^d$ be a lattice, let $T = \Spec(k[N^*])$ be the algebraic torus with co-character lattice $N$, let $\sigma$ be a pointed rational cone in $N_\R$, and let $X$ be the affine $T$-toric variety associated to $\sigma$. We assume that $X$ is $\Q$-Gorenstein and let $m \in \Z_{>0}$ and $q \in N^*$ be such that if $v$ is the first lattice point of any ray of $\sigma$,
\[
	\langle v, q \rangle = m.
\]
Then $mK_X$ is Cartier, so we have the ideal sheaf $\sJ_{X,m}$ on $X$. Also note that any $\Q$-Gorenstein toric variety has log-terminal singularities \cite[Corollary 4.2]{Batyrev98}, so the Gorenstein measure $\mu^\Gor_X$ is well defined.

In this section, we prove \autoref{mainPropositionGorensteinMeasureToricVariety} below about the Gorenstein measure $\mu^\Gor_X$. In \autoref{sectionStringInvariantsAndToricArtinStacks}, we will use this theorem and \autoref{mainPropositionMotivicIntegralsCanonicalFantastacks} to compare $\mu^\Gor_X$ with the motivic measure $\mu_\cX$ of the canonical stack $\cX$ over $X$. 

Although we will only use \autoref{mainPropositionGorensteinMeasureToricVariety} in the case where $\sigma$ is $d$-dimensional, there is no need to make that assumption on $\sigma$ in this section.

\begin{theorem}
\label{mainPropositionGorensteinMeasureToricVariety}
Let $w \in \sigma \cap N \subset \Hom(\sigma^\vee \cap N^*, \N \cup \{\infty\})$.
\begin{enumerate}[(a)]

\item\label{GorensteinMeasureOfSubsetOfFiberOfTrop} The restriction of $\ord_{\sJ_{X,m}}$ to $\trop^{-1}(w) \subset \sL(X)$ is constant and not equal to infinity. In particular, there exists some $j_w \in \Z$ such that for any measurable subset $C \subset \trop^{-1}(w) \subset \sL(X)$,
\[
	\mu^\Gor_X(C) = (\bL^{1/m})^{j_w} \mu_X(C)  \in \widehat{\sM}_k[\bL^{1/m}].
\]

\item\label{GorensteinMeasureOfFiberOfTrop} The set $\trop^{-1}(w) \subset \sL(X)$ is measurable and
\[
	\mu^\Gor_X(\trop^{-1}(w)) = \bL^{-d}(\bL-1)^d (\bL^{1/m})^{- \langle w, q \rangle} \in \widehat{\sM}_k[\bL^{1/m}].
\]

\end{enumerate}
\end{theorem}

\begin{remark}
Summing over $w \in \sigma \cap N$, \autoref{mainPropositionGorensteinMeasureToricVariety}(\ref{GorensteinMeasureOfFiberOfTrop}) gives Batyrev's formula \cite[Theorem 4.3]{Batyrev98} for the stringy Hodge--Deligne invariant of a toric variety. Furthermore,  \autoref{mainPropositionGorensteinMeasureToricVariety}(\ref{GorensteinMeasureOfFiberOfTrop}) appears to be a special case of \cite[Lemma 4.5]{BatyrevMoreau}. For the benefit of the reader, we include a short self-contained proof.
\end{remark}

\begin{remark}
When $w$ is an integer combination of lattice points on the rays of $\sigma$, we have that $\langle w, q \rangle$ is divisible by $m$, so in that case \autoref{mainPropositionGorensteinMeasureToricVariety}(\ref{GorensteinMeasureOfFiberOfTrop}) implies
\[
	\mu^\Gor_X(\trop^{-1}(w)) \in \widehat{\sM}_k.
\]
\end{remark}

\subsection{Gorenstein measure and monomial ideals} For the remainder of this section, let $M = N^*$, let $P = \sigma^\vee \cap M$, and for each $p \in P$, let $\chi^p \in k[P]$ be the monomial indexed by $p$. 

If $\sJ$ is a nonzero ideal sheaf on $X$ generated by monomials $\{\chi^{p_i}\}_i$, then for any $\varphi \in \sL(X)$ with $\trop(\varphi) = w$,
\[
	\ord_{\sJ}(\varphi) = \min_i \langle w, p_i \rangle \in \Z_{\geq 0}.
\]
Therefore to prove \autoref{mainPropositionGorensteinMeasureToricVariety}(\ref{GorensteinMeasureOfSubsetOfFiberOfTrop}), it is sufficient to show that the ideal $\sJ_{X,m}$ is generated by monomials.

For the remainder of this subsection, fix a basis $e_1, \dots, e_d$ for $M$, and for any $p_1, \dots, p_d \in P$, set
\[
	c(p_1, \dots, p_d) = \det( (a_{i, j})_{i, j}) \in \Z,
\]
where the $a_{i, j} \in \Z$ are such that $p_j = \sum_{i = 1}^d a_{i, j} e_i$ for all $j \in \{1, \dots, d\}$. For any $(p_{i,j})_{i \in \{1, \dots, m\},j \in \{1, \dots, d\}} \in P^{md}$, set
\[
	z((p_{i,j})_{i,j}) = \chi^{-q} \prod_{i = 1}^m c(p_{i,1}, \dots, p_{i,d}) \chi^{p_{i,1} + \dots + p_{i,d}} \in k[M].
\]

\begin{lemma}
Let $(p_{i,j})_{i ,j} \in P^{md}$. Then
\[
	z((p_{i,j})_{i,j}) \in k[P].
\]
\end{lemma}

\begin{proof}
If $-q + \sum_{ i = 1}^m \sum_{j =1}^d p_{i,j} \in P$, we are done, so we may assume that there exists some first lattice point $v$ of a ray of $\sigma$ such that $\langle v, -q + \sum_{ i = 1}^m \sum_{j =1}^d p_{i,j} \rangle < 0$. Then
\[
	0 > \langle v, -q + \sum_{ i = 1}^m \sum_{j =1}^d p_{i,j}\rangle = \sum_{i =1}^m \langle v,  -\frac{1}{m} q + \sum_{j =1}^d p_{i,j} \rangle = \sum_{i=1}^m \left(-1 + \sum_{j=1}^d \langle v, p_{i,j} \rangle \right),
\]
so for some $i \in \{1, \dots, m\}$, we have $\langle v, p_{i,j} \rangle = 0$ for all $j \in \{1, \dots, d\}$, so
\[
	c(p_{i,1}, \dots, p_{i,d}) = 0,
\]
which implies that $z((p_{i,j})_{i,j}) = 0 \in k[P]$.
\end{proof}

We now prove the next proposition, which as discussed, immediately implies \autoref{mainPropositionGorensteinMeasureToricVariety}(\ref{GorensteinMeasureOfSubsetOfFiberOfTrop}). Note that because $\sigma$ is pointed, $P^\gp = M$, so the $c(p_1, \dots, p_d)$ are not all equal to 0 and the $z((p_{i,j})_{i,j})$ are not all equal to 0.

\begin{proposition}
\label{GorensteinIdealIsGeneratedByMonomials}
The ideal $\sJ_{X,m}$ is generated by the set
\[
	\{ z((p_{i,j})_{i,j}) \mid (p_{i,j})_{i ,j} \in P^{md}\}.
\]
\end{proposition}

\begin{proof}
Since $k[P]$ is generated over $k$ by the set $\{\chi^p\mid p\in P\}$, we see that $\Gamma(X,\Omega^d_X)$ is generated by the elements $\chi^{p_1}\wedge\dots\wedge\chi^{p_d}$ as $p_1,\dots,p_d$ range over elements of $P$. So, $\Gamma(X,(\Omega_X^d)^{\otimes m})$ is generated by the set
\[
	\{ \bigotimes_{i=1}^m \diff \chi^{p_{i,1}} \wedge \dots \wedge \diff \chi^{p_{i, d}} \mid (p_{i,j})_{i ,j} \in P^{md}\}.
\]

We next show that the global sections of $\omega_{X,m}$ are generated by
\[
	\chi^q \cdot \left( \diff\log \chi^{e_1} \wedge \dots \wedge \diff \log \chi^{e_d} \right)^{\otimes m};
\]
this is essentially given by \cite[Proposition 8.2.9]{CoxLittleSchenck}, however, they only state the result for $\omega_X$ instead of $\omega_{X,m}$. The proof for all $m$ works identically. To see this, let $\iota: X_\sm \hookrightarrow X$ be the inclusion of the smooth locus. By \cite[Theorem 8.2.3]{CoxLittleSchenck},
\[
\omega_{X,m} = \iota_*( (\Omega_{X_\sm}^{\dim X})^{\otimes m} )\cong\iota_*\cO_{X_\sm}(-m\sum_\rho D_\rho)=\cO_X(-m\sum_\rho D_\rho),
\]
where the sum runs over all $\rho\in\Sigma(1)$ and $D_\rho$ denotes the corresponding torus-invariant divisor. Implicit in \cite[Theorem 8.2.3]{CoxLittleSchenck} is an identification of $\Omega_{X_\sm}^{\dim X}$ with a subsheaf of $\cO_{X_\sm}$; this identification comes from the maps (8.1.3) and (8.1.5) of \cite{CoxLittleSchenck} and can be described as follows. We have an inclusion of $\Omega_{X_\sm}^{\dim X}$ into the logarithmic differentials $\Omega_{X_\sm}^{\dim X}(\log D)$ and the latter is isomorphic to $\cO_{X_\sm}$ via the map 
\[
\cO_{X_\sm}\xrightarrow{\cong}\Omega_{X_\sm}^{\dim X}(\log D)
\]
\[
f\mapsto f\cdot \diff\log \chi^{e_1} \wedge \dots \wedge \diff \log \chi^{e_d}.
\]
For each $\rho\in\Sigma(1)$, let $v_\rho$ denote the first lattice point on the ray $\rho$. Having established that $\omega_{X,m}\cong\cO_X(-m\sum_\rho D_\rho)$, \cite[Proposition 4.3.2]{CoxLittleSchenck} tells us $\Gamma(X,\omega_{X,m})$ is generated over $k$ by the sections of the form $\chi^p \cdot \left( \diff\log \chi^{e_1} \wedge \dots \wedge \diff \log \chi^{e_d} \right)^{\otimes m}$ for $p\in M$ such that $\langle p,v_\rho\rangle\geq m$ for all $\rho\in\Sigma(1)$. This condition on the inner product implies $p-q\in P$ and so $\Gamma(X,\omega_{X,m})$ is generated over $k[P]$ by
\[
	\chi^q \cdot \left( \diff\log \chi^{e_1} \wedge \dots \wedge \diff \log \chi^{e_d} \right)^{\otimes m}.
\]

Lastly, for any $p_1, \dots, p_d \in P$,
\begin{align*}
	\diff \chi^{p_1} \wedge \dots \wedge \diff \chi^{p_d} &= \chi^{p_1+\dots + p_d } \cdot \diff\log \chi^{p_1} \wedge \dots \wedge \diff \log \chi^{p_d}\\
	&= c(p_1, \dots, p_d) \chi^{p_1+\dots + p_d } \cdot \diff\log \chi^{e_1} \wedge \dots \wedge \diff \log \chi^{e_d}.
\end{align*}
Thus for any $(p_{i,j})_{i ,j} \in P^{md}$,
\[
	\bigotimes_{i=1}^m \diff \chi^{p_{i,1}} \wedge \dots \wedge \diff \chi^{p_{i, d}} = z((p_{i,j})_{i,j}) \cdot \chi^q \cdot \left( \diff\log \chi^{e_1} \wedge \dots \wedge \diff \log \chi^{e_d} \right)^{\otimes m}.
\]
The proposition then follows from the definition of $\sJ_{X,m}$.
\end{proof}

\begin{remark}
\autoref{GorensteinIdealIsGeneratedByMonomials} implies that if $\varphi \in \sL(X)$ with $\trop(\varphi)=w$, then
\[
	\ord_{\sJ_{X,m}}(\varphi) = - \langle w, q \rangle + \min_{\substack{p_1, \dots, p_d \in P\\ c(p_1, \dots, p_d) \neq 0}} m \langle w, p_1 + \dots + p_d \rangle,
\]
so if $w$ is an integer combination of lattice points on the rays of $\sigma$, then $\ord_{\sJ_{X,m}}(\varphi)$ is divisible by $m$ and $\mu^\Gor_X(C) \in \widehat{\sM}_k$ for any measurable subset $C \subset \trop^{-1}(w)$.
\end{remark}

\subsection{Gorenstein measure and toric modifications}
\label{subsec:gor-toric-mod}
In this subsection, we complete the proof of \autoref{mainPropositionGorensteinMeasureToricVariety}(\ref{GorensteinMeasureOfFiberOfTrop}). We first handle the case where $w = 0$.

\begin{proposition}
We have
\[
	\mu^\Gor_X(\trop^{-1}(0)) = \bL^{-d}(\bL - 1)^d.
\]
\end{proposition}

\begin{proof}
If $\varphi \in \sL(X)$, then $\trop(\varphi) = 0$ if and only if $\varphi^*(p)$ is a unit for all $p \in P$, which occurs if and only if $\varphi \in \sL(T)$. Thus
\[
	\trop^{-1}(0) = \sL(T) \subset \sL(X),
\]
and because $T$ is smooth, we have
\[
	\mu^\Gor_X(\trop^{-1}(0)) = \mu_T(\sL(T)) = \bL^{-\dim T} \e(T) = \bL^{-d}(\bL-1)^d
\]
where the first equality is given by \autoref{mainPropositionGorensteinMeasureToricVariety}(\ref{GorensteinMeasureOfSubsetOfFiberOfTrop}).
\end{proof}

We now only need to prove \autoref{mainPropositionGorensteinMeasureToricVariety}(\ref{GorensteinMeasureOfFiberOfTrop}) in the case where $w \neq 0$. For the remainder of this section, fix $w \in \sigma \cap N$, assume that $w \neq 0$, and let $\ell \in \Z_{>0}$ be such that $(1/\ell)w$ is the first lattice point of the ray $\tau:=\R_{\geq 0} w$. We will compute $\mu_X(\trop^{-1}(w))$ by applying the change of variables formula to a certain toric modification of $X$. Let $Y\cong\bA^1\times\bG_m^{d-1}$ be the affine $T$-toric variety whose fan is given by $\tau$, let $D$ be the (irreducible) boundary divisor of $Y$, and let $\rho: Y \to X$ be the toric morphism induced by the identity map $N \to N$. It is standard to compute the relative canonical divisor of such a birational toric morphism. In this case,
\[
	mK_{Y} - \rho^* (mK_X) = \left(\frac{1}{\ell}\langle w, q \rangle - m\right)D.
\]

For the remainder of this section, let $\cO(-D)$ be the ideal sheaf of $D$ in $Y$.

\begin{proposition}
\label{bijectionOfFibersOfTropToricModificationUsingRay}
The map $\sL(\rho): \sL(Y) \to \sL(X)$ induces a bijection 
\[
	(\ord_{\cO(-D)}^{-1}(\ell))(k') \to (\trop^{-1}(w))(k')
\]
for every field extension $k'$ of $k$.
\end{proposition}

\begin{proof}
Let $k'$ be a field extension of $k$. By construction, $\sL(\rho)$ induces a bijection $(\sL(Y) \setminus \sL(Y \setminus T))(k') \to (\trop^{-1}(N \cap \R_{\geq 0} w))(k')$. Therefore it is sufficient to show that
\[
	\sL(\rho)^{-1}((\trop^{-1}(w))(k')) = (\ord_{\cO(-D)}^{-1}(\ell))(k').
\]
By \autoref{l:trop-basic-properties}(\ref{l:trop-basic-properties::inverse-image}), it is enough to show that if $u\in\tau^\vee\cap M$, there exists $u'\in\tau^\vee\cap M$ with $u+u'\in\sigma^\vee\cap M$. Consider the quotient map $\eta\colon M\to M/(w^\perp\cap M)\cong\Z$. Since $\sigma^\vee\subset\tau^\vee$, we see $\eta(\sigma^\vee\cap M)\subset\eta(\tau^\vee\cap M)=\N$. So, for any $u\in\tau^\vee\cap M$, there exists $n\in\Z_{>0}$ such that $n\eta(u)\in\eta(\sigma^\vee\cap M)$, i.e.~for some choice of $u''\in w^\perp\cap M$, letting $u'=u''+(n-1)u$, we have $u+u'\in\sigma^\vee\cap M$.
%
\end{proof}

The next proposition completes the proof of \autoref{mainPropositionGorensteinMeasureToricVariety}(\ref{GorensteinMeasureOfFiberOfTrop}).

\begin{proposition}
We have
\[
	\mu^\Gor_X(\trop^{-1}(w)) = \bL^{-d}(\bL-1)^d (\bL^{1/m})^{- \langle w, q \rangle}.
\]
\end{proposition}

\begin{proof}
By \autoref{mainPropositionGorensteinMeasureToricVariety}(\ref{GorensteinMeasureOfSubsetOfFiberOfTrop}), there exists some $j_w \in  \Z$ such that $\ord_{\sJ_{X,m}}$ is equal to $j_w$ on $\trop^{-1}(w)$. By \autoref{bijectionOfFibersOfTropToricModificationUsingRay}, we also have that $\ord_{\sJ_{X,m}} \circ \sL(\rho)$ is equal to $j_w$ on $\ord_{\cO(-D)}^{-1}(\ell)$. Thus \cite[Chapter 7 Proposition 3.2.5]{ChambertLoirNicaiseSebag} implies that on $\ord_{\cO(-D)}^{-1}(\ell)$,
\[
	 - \ordjac_\rho\ =\ -\frac{1}{m}\ord_{\sJ_{X,m}} \circ \sL(\rho) -\frac{1}{m}\left(\frac{1}{\ell}\langle w, q \rangle - m\right) \ord_{\cO(-D)}\ =\ -\frac{j_w}{m} - \frac{1}{m}\langle w, q \rangle + \ell,
\]
where $\ordjac_\rho: \sL(Y) \to \N \cup \{\infty\}$ denotes the order function of the jacobian ideal of $\rho$. Therefore,
\begin{align*}
	\mu^\Gor_X(\trop^{-1}(w)) &\ = \int_{\trop^{-1}(w)} (\bL^{1/m})^{\ord_{\sJ_{X,m}}} \diff\mu_X
	\ =\ (\bL^{1/m})^{j_w} \mu_X(\trop^{-1}(w))\\
	& =\ (\bL^{1/m})^{j_w}\int_{\ord_{\cO(-D)}^{-1}(\ell)} \bL^{-\ordjac_\rho}\diff\mu_Y
	\ =\  (\bL^{1/m})^{j_w}\int_{\ord_{\cO(-D)}^{-1}(\ell)} \bL^{-j_w/m - \langle w, q \rangle/m + \ell}\diff\mu_Y\\
	&= (\bL^{1/m})^{-\langle w, q \rangle} \bL^\ell \mu_Y(\ord_{\cO(-D)}^{-1}(\ell))
	\ =\ \bL^{-d}(\bL-1)^d (\bL^{1/m})^{- \langle w, q \rangle},
\end{align*}
where the third equality is due to \autoref{bijectionOfFibersOfTropToricModificationUsingRay} and the motivic change of variables formula (see for example \cite[Chapter 6 Theorem 4.3.1]{ChambertLoirNicaiseSebag}), and the final equality follows from
\[
	\mu_Y(\ord_{\cO(-D)}^{-1}(\ell)) = \e(D)(\bL-1)\bL^{-d-\ell} = (\bL-1)^d\bL^{-d-\ell},
\]
which is a consequence of \cite[Chapter 7 Lemma 3.3.3]{ChambertLoirNicaiseSebag}.
\end{proof}

\section{Motivic measure and canonical stacks}\label{sectionMotivicMeasureToricStack}

Let $d \in \N$, let $N \cong \Z^d$ be a lattice, let $T = \Spec(k[N^*])$ be the algebraic torus with co-character lattice $N$, let $\sigma$ be a pointed rational cone in $N_\R$, let $X$ be the affine $T$-toric variety associated to $\sigma$, let $\cX$ be the canonical stack over $X$, and let $\pi: \cX \to X$ be the canonical map. We assume that $\sigma$ is $d$-dimensional and use the notation listed in \autoref{notationforcanonicalfantastackforfulldimensionalcone}. We assume that $X$ is $\Q$-Gorenstein and let $m \in \Z_{>0}$ and $q \in P$ be such that if $v$ is the first lattice point of any ray of $\sigma$,
\[
	\langle v, q \rangle = m.
\]

In this section, we prove the following theorem about the motivic measure $\mu_\cX$ which mirrors \autoref{mainPropositionGorensteinMeasureToricVariety} for $\mu^\Gor_X$. In \autoref{sectionStringInvariantsAndToricArtinStacks}, we will combine these two theorems to compare the measures $\mu_\cX$ and $\mu^\Gor_X$.

\begin{theorem}
\label{mainPropositionMotivicIntegralsCanonicalFantastacks}
Let $w \in \sigma \cap N \subset \Hom(P, \N \cup \{\infty\})$.
\begin{enumerate}[(a)]

\item\label{stackMeasureOfPieceOfFiberOfTrop} If $C \subset \trop^{-1}(w) \subset \sL(X)$ is measurable, then $\sL(\pi)^{-1}(C)$ is a measurable subset of $|\sL(\cX)|$. Furthermore, there exists $\Theta_w \in \widehat{\sM}_k$ such that for any measurable subset $C \subset \trop^{-1}(w) \subset \sL(X)$,
\[
	\mu_\cX(\sL(\pi)^{-1}(C)) = \Theta_w \mu_X(C).
\]

\item\label{stackMeasureOfFiberOfTrop} The set $\sL(\pi)^{-1}(\trop^{-1}(w)) \subset |\sL(\cX)|$ is measurable and
\[
	\mu_{\cX}(\sL(\pi)^{-1}(\trop^{-1}(w))) = (\# \beta^{-1}(w)) \bL^{-d}(\bL-1)^d (\bL^{1/m})^{- \langle w, q \rangle}.
\]

\end{enumerate}
\end{theorem}

\begin{remark}
A-priori we only have that
\[
	(\# \beta^{-1}(w)) \bL^{-d}(\bL-1)^d (\bL^{1/m})^{- \langle w, q \rangle} \in \widehat{\sM}_k[\bL^{1/m}] \supset \widehat{\sM_k},
\]
but by \autoref{QGorensteinsumofgeneratorsisinvariant} below, we have either $\beta^{-1}(w) = \emptyset$ or $\langle w, q \rangle$ is divisible by $m$, so
\[
	(\# \beta^{-1}(w)) \bL^{-d}(\bL-1)^d (\bL^{1/m})^{- \langle w, q \rangle} \in \widehat{\sM}_k.
\]
\end{remark}

Before we prove \autoref{mainPropositionMotivicIntegralsCanonicalFantastacks}, we show that it implies the next proposition.

\begin{proposition}
\label{summingUpStackMeasureOverFibersOfTrop}
Let $C \subset \sL(X)$ be a measurable subset. Then $\sL(\pi)^{-1}(C)$ is a measurable subset of $|\sL(\cX)|$ and
\[
	\mu_\cX(\sL(\pi)^{-1}(C)) = \sum_{w \in \sigma \cap N} \mu_\cX(\sL(\pi)^{-1}(\trop^{-1}(w) \cap C)).
\]
\end{proposition}

\begin{remark}
By \autoref{fiberOfTropIsCylinder} and \autoref{mainPropositionMotivicIntegralsCanonicalFantastacks}(\ref{stackMeasureOfPieceOfFiberOfTrop}), each $\sL(\pi)^{-1}(\trop^{-1}(w) \cap C)$ in the statement of \autoref{summingUpStackMeasureOverFibersOfTrop} is a measurable subset of $|\sL(\cX)|$.
\end{remark}

\begin{proof}
Set
\begin{align*}
	\cC &= \sL(\pi)^{-1}(C),\\
	\cC^{(0)} &= \bigcup_{w \in \sigma \cap N} \sL(\pi)^{-1}(\trop^{-1}(w) \cap C),\\
	\cC^{(\infty)} &= \cC \setminus \cC^{(0)}.
\end{align*}
For each $w \in \sigma \cap N$,
\begin{align*}
	\Vert \mu_\cX(\sL(\pi)^{-1}(\trop^{-1}(w) \cap C)) \Vert &\leq \Vert \mu_\cX(\sL(\pi)^{-1}(\trop^{-1}(w))) \Vert\\
	&\leq \Vert \bL^{-d}(\bL-1)^d \Vert \exp(-\langle w, q \rangle/m),
\end{align*}
where the first inequality is by \autoref{measureOfSubsetHasSmallerSize} and the second inequality is by \autoref{mainPropositionMotivicIntegralsCanonicalFantastacks}(\ref{stackMeasureOfFiberOfTrop}). For each $\varepsilon \in \R_{>0}$ there are only finitely many $w \in \sigma \cap N$ with $\exp(-\langle w, q \rangle / m) \geq \varepsilon$, so \autoref{measureOfCountableDisjointUnion} implies that $\cC^{(0)}$ is measurable and
\[
	\mu_\cX(\cC^{(0)}) =  \sum_{w \in \sigma \cap N} \mu_\cX(\sL(\pi)^{-1}(\trop^{-1}(w) \cap C)).
\]
By \autoref{subsetOfNegligibleIsNegligible} and \autoref{closedSubstackGivesNegligibeSet}, the set $\cC^{(\infty)}$ is measurable and
\[
	\mu_\cX(\cC^{(\infty)}) = 0.
\]
Therefore by \autoref{measureOfCountableDisjointUnion}, the set $\cC = \cC^{(0)} \sqcup \cC^{(\infty)}$ is measurable and
\[
	\mu_\cX(\cC) = \mu_\cX(\cC^{(0)}) = \sum_{w \in \sigma \cap N} \mu_\cX(\sL(\pi)^{-1}(\trop^{-1}(w) \cap C)).\qedhere
\]
\end{proof}

We will use the remainder of this section to prove \autoref{mainPropositionMotivicIntegralsCanonicalFantastacks}.

\subsection{Canonical stacks and preimages of measurable subsets}

In this subsection, we will prove \autoref{mainPropositionMotivicIntegralsCanonicalFantastacks}(\ref{stackMeasureOfPieceOfFiberOfTrop}). We begin with a couple lemmas.

\begin{lemma}
\label{liftingJetsWithCorrectField}
Let $Y$ be an irreducible finite type scheme over $k$ with smooth locus $Y_\sm \subset Y$, and let $C \subset \sL(Y)$ be a cylinder such that $C \cap \sL(Y \setminus Y_\sm) = \emptyset$. 

Then there exists some $n_C \in \N$ that satisfies the following. For any field extension $k'$ of $k$, any $n \geq n_C$, and any $\varphi_n \in \sL_n(Y)(k')$ with image in $\theta_n(C)$, there exists some $\varphi \in \sL(Y)(k')$ with image in $C$ such that $\theta_n(\varphi) = \varphi_n$.
\end{lemma}

\begin{proof}
By \cite[Chapter 5 Proposition 1.3.2(a) and Proposition 2.3.4]{ChambertLoirNicaiseSebag}, there exists a function $\ordjac_Y: \sL(Y) \to \N \cup \{\infty\}$ and some $c \in \Z_{>0}$ such that
\begin{itemize}

\item for every $n \in \N$, the set $\ordjac_Y^{-1}(n) \subset \sL(Y)$ is a cylinder,

\item the image of $\sL(Y) \setminus \sL(Y \setminus Y_\sm)$ under $\ordjac_Y$ is contained in $\N$, and

\item for every $n \in \N$, field extension $k'$ of $k$, and $\varphi_n \in \sL_n(Y)(k')$ whose image in $\sL_n(Y)$ is contained in $\theta_{n}( \ordjac_Y^{-1}(n') )$ for some $n' \leq n/c$, there exists some $\varphi \in \sL(Y)(k')$ with $\theta_n(\varphi) = \varphi_n$.

\end{itemize}
Because $C \cap \sL(Y \setminus Y_\sm) = \emptyset$, the collection $\{\ordjac_Y^{-1}(n)\}_{n \in \N}$ is a cover of the cylinder $C$ by cylinders. Thus by quasi-compactness of the constructible topology of $\sL(Y)$ (see e.g. \cite[Appendix Theorem 1.2.4(a)]{ChambertLoirNicaiseSebag}), there exists some $n'_C \in \N$ such that $C \subset \bigcup_{n = 0}^{n'_C} \ordjac_Y^{-1}(n)$. Let $n_C \in \N$ be such that $n_C \geq cn'_C$ and such that $C$ is the preimage under $\theta_{n_C}$ of a constructible subset of $\sL_{n_C}(Y)$.

Now let $k'$ be a field extension of $k$, let $n \geq n_C$, and let $\varphi_n \in \sL_n(Y)(k')$ have image in $\theta_n(C)$. Then the image of $\varphi_n$ is contained in $\theta_n(\ordjac_Y^{-1}(n'))$ for some $n' \leq n'_C \leq n_C/c \leq n/c$, so there exists some $\varphi \in \sL(Y)(k')$ with $\theta_n(\varphi) = \varphi_n$. Because $C$ is the preimage of a subset of $\sL_n(Y)$, the arc $\varphi$ has image in $C$.
\end{proof}

\begin{lemma}
\label{imageOfPreimageOfCylinderMapFromSmoothStack}
Let $Y$ be a finite type scheme over $k$, let $\cY$ be a smooth Artin stack over $k$, let $\xi: \cY \to Y$ be a morphism, let $C \subset \sL(Y)$ be a cylinder, and set $\cC = \sL(\xi)^{-1}(C) \subset | \sL(\cY) |$.

Then $\cC$ is a cylinder and there exists some $n_0 \in \N$ such that for all $n \geq n_0$,
\[
	\theta_n(\cC) = \sL_n(\xi)^{-1}(\theta_n(C)).
\]
\end{lemma}

\begin{proof}
We first note that for all $n \in \N$, we have an obvious inclusion
\[
	\theta_n(\cC) \subset \sL_n(\xi)^{-1}(\theta_n(C)).
\]
Because $C$ is a cylinder, there exists some $n_0 \in \N$ and some constructible subset $C_{n_0} \subset \sL_{n_0}(Y)$ such that $C = (\theta_{n_0})^{-1}(C_{n_0})$. Then
\[
	\cC = \sL(\xi)^{-1}( (\theta_{n_0})^{-1}(C_{n_0}) ) = (\theta_{n_0})^{-1}( \sL_n(\xi)^{-1}(C_{n_0}) )
\]
is a cylinder because $\sL_n(\xi)^{-1}(C_{n_0})$ is a constructible subset of $|\sL_n(\cY)|$. We will finish this proof by showing that for any $n \geq n_0$, we have $\theta_n(\cC) \supset \sL_n(\xi)^{-1}(\theta_n(C))$.

Let $n \geq n_0$ and $\varphi_n \in \sL_n(\xi)^{-1}(\theta_n(C))$. Because $\cY$ is smooth, there exists some $\varphi \in |\sL(\cY)|$ such that $\theta_n(\varphi) = \varphi_n$. Then $\theta_n(\sL(\xi)(\varphi)) = \sL_n(\xi)( \varphi_n ) \in \theta_n(C)$, so $\sL(\xi)(\varphi) \in (\theta_n)^{-1}(\theta_n(C))$. But $(\theta_n)^{-1}(\theta_n(C)) = C$ because $C = (\theta_n)^{-1}( (\theta^n_{n_0})^{-1}(C_{n_0}) )$ is the preimage of a subset of $\sL_n(Y)$. Thus $\sL(\xi)(\varphi) \in C$, which implies $\varphi \in \cC$ and $\varphi_n \in \theta_n(\cC)$.
\end{proof}

We may now prove the special case of \autoref{mainPropositionMotivicIntegralsCanonicalFantastacks}(\ref{stackMeasureOfPieceOfFiberOfTrop}) where $C$ is a cylinder.

\begin{proposition}
\label{motivicIntegralsFantastacksCylinderCase}
Let $w \in \sigma \cap N$. If $C \subset \trop^{-1}(w) \subset \sL(X)$ is a cylinder, then $\sL(\pi)^{-1}(C)\subset |\sL(\cX)|$ is a cylinder. Furthermore, there exists some $\Theta_w \in \widehat{\sM}_k$ such that for any cylinder $C \subset \trop^{-1}(w) \subset \sL(X)$,
\[
	\mu_\cX(\sL(\pi)^{-1}(C)) = \Theta_w \mu_X(C).
\]
\end{proposition}

\begin{proof}
Let $n_w, \Theta_w, \{\cF_n\}_{n \geq n_w}$ be as in the statement of \autoref{fibersofhtemapsofjetsmainproposition}, and let $n_{\trop^{-1}(w)}$ be as in the statement of \autoref{liftingJetsWithCorrectField} (with $Y = X$ and $C = \trop^{-1}(w)$). We show that if $n \geq \max\{n_w, n_{\trop^{-1}(w)}\}$ and $C_n \subset \theta_n(\trop^{-1}(w))$ is constructible, then
\[
	\e(\sL_n(\pi)^{-1}(C_n)) = \Theta_w \e(C_n).
\]

Let $n \geq \max\{n_w, n_{\trop^{-1}(w)}\}$, let $k'$ be a field extension of $k$, and let $\varphi_n \in \sL_n(X)(k')$ with image in $\theta_n(\trop^{-1}(w))$. Then by our choice of $n_{\trop^{-1}(w)}$, there exists some $\varphi \in \sL(X)(k')$ such that $\trop(\varphi) = w$ and $\theta_n(\varphi) = \varphi_n$. Then by our choice of $n_w$ and $\cF_n$,
\[
	\sL_n(\pi)^{-1}(\varphi_n)_\red \cong \cF_n \otimes_k k'.
\]
Therefore \autoref{existenceOfClassOfConstructibleSubsetOfStack}, \autoref{remarkPiecewiseTrivialFibrationGivesProductOfClasses}, and \autoref{piecewiseTrivialFibrationCriterion} imply that for any constructible subset $C_n \subset \theta_n(\trop^{-1}(w))$,
\[
	\e(\sL_n(\pi)^{-1}(C_n)) = \e(\cF_n)\e(C_n) = \Theta_w \e(C_n),
\]
where the second equality holds by our choice of $n_w, \Theta_w, \cF_n$.

Now let $C \subset \trop^{-1}(w)$ be a cylinder, and let $\cC = \sL(\pi)^{-1}(C) \subset |\sL(\cX)|$. Then $\cC$ is a cylinder by \autoref{imageOfPreimageOfCylinderMapFromSmoothStack}. Let $n_0$ be as in the statement of \autoref{imageOfPreimageOfCylinderMapFromSmoothStack} (with $\xi = \pi$). Then for any $n \geq \max\{n_w, n_{\trop^{-1}(w)}, n_0\}$,
\[
	\e(\theta_n( \cC )) = \e(\sL_n(\pi)^{-1}(\theta_n(C))) = \Theta_w \e(\theta_n(C)),
\]
where the first equality holds by our choice of $n_0$. Therefore
\begin{align*}
	\mu_\cX( \cC ) &= \lim_{n \to \infty} \e(\theta_n(\cC))\bL^{-(n+1)\dim\cX}\\
	&= \Theta_w \lim_{n \to \infty} \e(\theta_n(C))\bL^{-(n+1)\dim X}\\
	&= \Theta_w \mu_X(C).\qedhere
\end{align*}
\end{proof}

Now we may complete the proof of \autoref{mainPropositionMotivicIntegralsCanonicalFantastacks}(\ref{stackMeasureOfPieceOfFiberOfTrop}) in general.

\begin{proof}[Proof of  \autoref{mainPropositionMotivicIntegralsCanonicalFantastacks}(\ref{stackMeasureOfPieceOfFiberOfTrop})]
Let $w \in \sigma \cap N$, let $\Theta_w$ be as in the statement of \autoref{motivicIntegralsFantastacksCylinderCase}, and let $C \subset \trop^{-1}(w)$ be a measurable subset of $\sL(X)$.

For any $\varepsilon \in \R_{>0}$ and any cylindrical $\varepsilon$-approximation $(C^{(0)}, (C^{(i)})_{i \in I})$ of $C$, \autoref{motivicIntegralsFantastacksCylinderCase} implies that $(\sL(\pi)^{-1}(C^{(0)}),(\sL(\pi)^{-1}(C^{(i)}))_{i \in I})$ is a cylindrical $\varepsilon \Vert \Theta_w \Vert$-approximation of $\sL(\pi)^{-1}(C)$. Therefore $\sL(\pi)^{-1}(C)$ is measurable, and by another application of \autoref{motivicIntegralsFantastacksCylinderCase},
\[
	\mu_\cX(\sL(\pi)^{-1}(C)) = \Theta_w \mu_X(C).\qedhere
\]
\end{proof}

\subsection{Quotient by an algebraic torus}
In order to apply \autoref{theoremMainMotivicIntegrationForQuotientStacks} to prove \autoref{mainPropositionMotivicIntegralsCanonicalFantastacks}(\ref{stackMeasureOfFiberOfTrop}), we must rewrite our fantastack as the quotient by a torus. The following result provides an explicit way to do so.

\begin{proposition}
\label{rewriteFantastackAsQuotientByAlgebraicTorus}
Let $\widehat{N} := \widetilde{N} \oplus N$ and $\widehat{\sigma} = \widetilde{\sigma} \times \{0\}$. Letting $\widehat{X}$ be the toric variety associated to $\widehat{\sigma}$, there is a $\widetilde{T}$-action on $\widehat{X}$ and an isomorphism $[\widehat{X} / \widetilde{T}] \xrightarrow{\sim} \cX$ such that $\widehat{X} \to [\widehat{X} / \widetilde{T}] \xrightarrow{\sim} \cX \xrightarrow{\pi} X$ is the toric morphism induced by $\nu\oplus\id\colon\widehat{N} \to N$.
\end{proposition}

\begin{proof}
For ease of notation, let $\widehat{\nu}=\nu\oplus\id$ and consider the stacky fan $(\widehat{\sigma},\widehat{\nu})$. One computes that $\cok(\widehat{\nu}^*)=\widetilde{M}$ and hence $G_{\widehat{\nu}}=\widetilde{T}$. As a result, $\cX_{\widehat{\sigma},\widehat{\nu}}=[\widehat{X} / \widetilde{T}]$ for an appropriate $\widetilde{T}$-action on $\widehat{X}$.

Next note that $\cX$ is, by definition, the toric stack $\cX_{\widetilde{\sigma},\nu}$, and consider the following commutative diagram, where the vertical maps are stacky fans and the horizontal maps are morphisms between the stacky fans.
\begin{center}
\begin{tikzcd}
\widetilde{\sigma} \arrow[r] & \widehat{\sigma} \arrow[r] & \sigma\\[-15pt]
\widetilde{N} \arrow[r] \arrow[d, "\nu"] & \widehat{N} \arrow[r, "\widehat{\nu}"] \arrow[d, "\widehat{\nu}"] & N \arrow[d, equal]\\
N \arrow[r, equal] & N \arrow[r, equal] & N
\end{tikzcd}
\end{center}
This induces morphisms $\cX\to [\widehat{X} / \widetilde{T}]\to X$ of toric stacks, and the composite is the morphism $\pi$. By \cite[Lemma B.17]{GeraschenkoSatriano1}, the former map is an isomorphism of toric stacks; this is the inverse of our desired isomorphism $[\widehat{X} / \widetilde{T}] \xrightarrow{\sim} \cX$. Lastly, we see that the toric morphism $\widehat{X}\to [\widehat{X} / \widetilde{T}]\to X$ is induced by the rightmost map in the top row of the above diagram, namely $\widehat{\nu}\colon\widehat{N} \to N$.
\end{proof}

\subsection{Canonical stacks and preimages of co-characters} We end this section by proving \autoref{mainPropositionMotivicIntegralsCanonicalFantastacks}(\ref{stackMeasureOfFiberOfTrop}).

For the remainder of this section, let $r = \rk \widetilde{N}$, let $v_1, \dots, v_r$ be the first lattice points of the rays of $\sigma$, let $e_1, \dots, e_r$ be the generators of $\widetilde{N}$ indexed by the rays $\R_{\geq 0} v_1, \dots, \R_{\geq 0} v_r$, so
\[
	\beta(e_i) = v_i
\]
for all $i \in \{1, \dots, r\}$, and let $f_1, \dots, f_r$ be the basis of $\widetilde{M}$ dual to $e_1, \dots, e_r$. Thus $f_1, \dots, f_r$ are the minimal generators of the monoid $F$.

\begin{lemma}
\label{QGorensteinsumofgeneratorsisinvariant}
We have the equality
\[
	m(f_1 + \dots + f_r) = q.
\]
\end{lemma}

\begin{proof}
Because the inclusion $P \hookrightarrow F$ is dual to $\beta$,
\[
	\langle e_i, q \rangle = m
\]
for all $i \in \{1, \dots, r\}$. Then
\[
	   m(f_1 + \dots + f_r) = \langle e_1, q \rangle f_1 + \dots + \langle e_r, q \rangle f_r = q.\qedhere
\]
\end{proof}

For the remainder of this section, let $\widehat{N}$, $\widehat{\sigma}$, and $\widehat{X}$ be as in \autoref{rewriteFantastackAsQuotientByAlgebraicTorus}. Let $\rho: \widehat{X} \to \cX$ be the composition $\widehat{X} \to [ \widehat{X} / \widetilde{T} ] \xrightarrow{\sim} \cX$, where the $\widetilde{T}$-action on $\widehat{X}$ and the isomorphism $[ \widehat{X} / \widetilde{T} ] \xrightarrow{\sim} \cX$ are as in the statement of \autoref{rewriteFantastackAsQuotientByAlgebraicTorus}, let $D_1, \dots, D_r$ be the divisors of $\widehat{X}$ indexed by $e_1 \oplus 0, \dots, e_r\oplus 0 \in \widetilde{N} \oplus N = \widehat{N}$, respectively, and let $\cO(-D_1), \dots, \cO(-D_r)$ be the ideal sheaves on $\widehat{X}$ of $D_1, \dots, D_r$, respectively. Note that for all $i \in \{1, \dots, r\}$, the ideal $\cO(-D_i)$ is generated by the monomial in $k[F \oplus M]$ indexed by $f_i \oplus 0$.

\begin{proposition}
\label{pushoutPreimageOfTropFiber}
Let $w \in \sigma \cap N$. Then
\[
	\sL(\rho)^{-1} \left( \sL(\pi)^{-1}(\trop^{-1}(w)) \right) = \bigcup_{\widetilde{w} \in \beta^{-1}(w)} \left( \bigcap_{i =1}^r \ord_{\cO(-D_i)}^{-1}( \langle \widetilde{w}, f_i \rangle ) \right).
\]
\end{proposition}
\begin{proof}
By \autoref{fiberOfTropIsCylinder} and \autoref{l:trop-basic-properties}(\ref{l:trop-basic-properties::inverse-image}), it is enough to show that for any $\widehat{f}\in\widehat{\sigma}^\vee\cap\widehat{N}^*$, there exists $\widehat{f}'\in\widehat{\sigma}^\vee\cap\widehat{N}^*$ such that $\widehat{f}+\widehat{f}'$ is in the image of $\sigma^\vee\cap M$. By definition of $\widehat{\nu}$, the map
\[
P=\sigma^\vee\cap M\to\widehat{\sigma}^\vee\cap\widehat{N}^*=(\widetilde{\sigma}^\vee\cap\widetilde{M})\oplus 0=F\oplus 0
\]
is precisely $p\mapsto(p,0)$. The result then follows from \autoref{elementofFaddstogetelementofP}.
\end{proof}

We may now complete the proof of \autoref{mainPropositionMotivicIntegralsCanonicalFantastacks}(\ref{stackMeasureOfFiberOfTrop}).

\begin{proof}[Proof of \autoref{mainPropositionMotivicIntegralsCanonicalFantastacks}(\ref{stackMeasureOfFiberOfTrop})]
Let $w \in \sigma \cap N$, and set 
\[
	\widehat{C} = \sL(\rho)^{-1} \left( \sL(\pi)^{-1}(\trop^{-1}(w)) \right) \subset \sL(\widehat{X}).
\]
Then
\begin{align*}
	\mu_{\widehat{X}} (\widehat{C}) &= \mu_{\widehat{X}}\left( \bigcup_{\widetilde{w} \in \beta^{-1}(w)} \left( \bigcap_{i =1}^r \ord_{\cO(-D_i)}^{-1}( \langle \widetilde{w}, f_i \rangle ) \right) \right)\\
	&= \sum_{\widetilde{w} \in \beta^{-1}(w)} \mu_{\widehat{X}} \left( \bigcap_{i =1}^r \ord_{\cO(-D_i)}^{-1}( \langle \widetilde{w}, f_i \rangle ) \right)\\
	&= \sum_{\widetilde{w} \in \beta^{-1}(w)} (\bL-1)^{r+d} \bL^{-(r+d)-\sum_{i =1}^r \langle \widetilde{w}, f_i\rangle}\\
	&=\sum_{\widetilde{w} \in \beta^{-1}(w)} (\bL-1)^{r+d} \bL^{-(r+d) - \langle w, q \rangle/m}\\
	&= (\#\beta^{-1}(w))(\bL-1)^{r+d} \bL^{-(r+d) - \langle w, q \rangle/m},
\end{align*}
where the first equality is by \autoref{pushoutPreimageOfTropFiber}, the second equality is by \autoref{fiberabovewisfiniteset} and the fact that the union in the first line is disjoint, the third equality is by \cite[Chapter 7 Lemma 3.3.3]{ChambertLoirNicaiseSebag} and the definition of $\widehat{X}$ and $D_1, \dots, D_r$, and the fourth equality is by \autoref{QGorensteinsumofgeneratorsisinvariant}.

The set $\sL(\pi)^{-1}(\trop^{-1}(w)) \subset |\sL(\cX)|$ is a cylinder by \autoref{fiberOfTropIsCylinder} and \autoref{motivicIntegralsFantastacksCylinderCase}. Then by \autoref{theoremMainMotivicIntegrationForQuotientStacks},
\begin{align*}
	\mu_\cX( \sL(\pi)^{-1}(\trop^{-1}(w))) &= \mu_{\widehat{X}} (\widehat{C}) \e(\widetilde{T})^{-1} \bL^{\dim \widetilde{T}}\\
	&= (\#\beta^{-1}(w))(\bL-1)^{d} \bL^{-d - \langle w, q \rangle/m}.\qedhere
\end{align*}
\end{proof}

\section{Stringy invariants and toric Artin stacks: proof of \autoref{mainTheoremStackMeasureAndGorensteinMeasure}}
\label{sectionStringInvariantsAndToricArtinStacks}

We complete the proof of \autoref{mainTheoremStackMeasureAndGorensteinMeasure} in this section. Let $d \in \N$, let $N \cong \Z^d$ be a lattice, and let $T = \Spec(k[N^*])$ be the algebraic torus with co-character lattice $N$. We recall the following lemma, whose proof is standard.

\begin{lemma}
\label{GorensteinIsFaceOfGorenstein}
Let $\sigma$ be a pointed rational cone in $N_\R$, and assume that the affine $T$-toric variety associated to $\sigma$ is $\Q$-Gorenstein. Then there exists a $d$-dimensional pointed rational cone $\overline{\sigma}$ in $N_\R$ such that $\sigma$ is a face of $\overline{\sigma}$ and the $T$-toric variety associeted to $\overline{\sigma}$ is $\Q$-Gorenstein.
\end{lemma}

\begin{remark}
\label{remarkReducingToAffineAndFullDimensionalCone}
By \autoref{rmk:mainTheoremStackMeasureAndGorensteinMeasure-more-general}, \autoref{canonicalFantastackOpenInvariantSubvariet}, \autoref{measureOfCountableDisjointUnion}, \autoref{stackMeasurablesFormBooleanAlgebra}, \autoref{stackMeasureOpenSubstack}, and \autoref{GorensteinIsFaceOfGorenstein}, to prove \autoref{mainTheoremStackMeasureAndGorensteinMeasure}, it is sufficient to prove the special case where $\cX$ is the canonical stack over an affine $T$-toric variety defined by a $d$-dimensional cone in $N_\R$.
\end{remark}

Let $\sigma$ be a $d$-dimensional pointed rational cone in $N_\R$, let $X$ be the affine $T$-toric variety associated to $\sigma$, let $\cX$ be the canonical stack over $X$, let $\pi: \cX \to X$ be the canonical map, and assume that $X$ is $\Q$-Gorenstein. We will use the notation listed in \autoref{notationforcanonicalfantastackforfulldimensionalcone}.

\begin{proposition}
\label{unionOfFibersOfTrop}
Let $W \subset \sigma \cap N$. Then
\[
	\bigcup_{w \in W} \trop^{-1}(w)
\]
is a measurable subset of $\sL(X)$.
\end{proposition}

\begin{proof}
We have that $\sL(X \setminus T)$ is a measurable subset of $\sL(X)$ because $X \setminus T$ is a closed subscheme of $X$. For each $w \in \sigma \cap N$, the set $\trop^{-1}(w)$ is a measurable subset of $\sL(X)$ by \autoref{fiberOfTropIsCylinder}. Also
\[
	\sL(X \setminus T) \cup \bigcup_{w \in \sigma \cap N} \trop^{-1}(w) = \sL(X)
\]
is measurable and the above union is disjoint, so for any $\varepsilon \in \R_{>0}$, there are only finitely many $w \in \sigma \cap N$ such that $\Vert\mu_X(\trop^{-1}(w))\Vert \geq \varepsilon$. Thus $\bigcup_{w \in W} \trop^{-1}(w)$ is a measurable subset of $\sL(X)$.
\end{proof}

\begin{proposition}
\label{separatednessFunctionMeasurableFibersFullDimensionalCone}
The function $\sep_\pi: \sL(X) \to \N$ has measurable fibers.
\end{proposition}

\begin{proof}
\autoref{MainPropositionFibersOfTheMapOfArcs} implies that for any $\varphi \in \sL(X)$ with $\trop(\varphi) \in \sigma \cap N$,
\[
	\sep_\pi(\varphi) = \#\beta^{-1}(\trop(\varphi)).
\]
Thus noting that each $\beta^{-1}(w)$ is finite, we have that for any $n \in \N$,
\[
	\sep_\pi^{-1}(n) = \left( \sep_\pi^{-1}(n) \cap \sL(X \setminus T) \right) \cup \bigcup_{\substack{w \in \sigma \cap N\\ \#\beta^{-1}(w) = n}} \trop^{-1}(w)
\]
is measurable by \autoref{unionOfFibersOfTrop} and the fact that $\mu_X(\sL(X \setminus T)) = 0$, which implies that any subset of $\sL(X \setminus T)$ is a measurable subset of $\sL(X)$.
\end{proof}

For the remainder of this section, we will use that by \autoref{separatednessFunctionMeasurableFibersFullDimensionalCone}, the integral $\int_C \sep_\pi \diff\mu^\Gor_X$ is well defined for any measurable subset $C \subset \sL(X)$.

We end this section with the next proposition, which along with \autoref{remarkReducingToAffineAndFullDimensionalCone} and \autoref{separatednessFunctionMeasurableFibersFullDimensionalCone}, implies \autoref{mainTheoremStackMeasureAndGorensteinMeasure}.

\begin{proposition}
\label{stackMeasureAndGorensteinMeasureFullDimensionalCone}
Let $C$ be a measurable subset of $\sL(X)$. Then $\sL(\pi)^{-1}(C)$ is a measurable subset of $|\sL(\cX)|$ and
\[
	\mu_\cX(\sL(\pi)^{-1}(C)) = \int_C \sep_\pi \diff\mu^\Gor_X.
\]
\end{proposition}

\begin{proof}
By \autoref{countableAdditivityOfGorensteinMeasure}, \autoref{summingUpStackMeasureOverFibersOfTrop}, and the fact that
\[
	\mu_X\left(\sL(X) \setminus  \bigcup_{w \in \sigma \cap N} \trop^{-1}(w) \right) = \mu_X \left(\sL(X \setminus T)\right) = 0,
\]
it is enough to prove the statement for each $\trop^{-1}(w)\cap C$. In other words, we may fix $w$ and assume $C\subset \trop^{-1}(w)$.

We first note that $\sL(\pi)^{-1}(C) \subset |\sL(\cX)|$ is measurable by \autoref{mainPropositionMotivicIntegralsCanonicalFantastacks}(\ref{stackMeasureOfPieceOfFiberOfTrop}). Let $m \in \Z_{>0}$ and $q \in P$ be such that $\langle v, q \rangle = m$ for any first lattice point $v$ of a ray of $\sigma$. Let $j_w$ be as in the statement of \autoref{mainPropositionGorensteinMeasureToricVariety}(\ref{GorensteinMeasureOfSubsetOfFiberOfTrop}), and let $\Theta_w$ be as in the statement of \autoref{mainPropositionMotivicIntegralsCanonicalFantastacks}(\ref{stackMeasureOfPieceOfFiberOfTrop}). By \autoref{mainPropositionGorensteinMeasureToricVariety}(\ref{GorensteinMeasureOfFiberOfTrop}) and our choice of $j_w$,
\[
	(\bL^{1/m})^{j_w} \mu_X(\trop^{-1}(w)) = \mu^\Gor_X(\trop^{-1}(w)) = \bL^{-d}(\bL-1)^d (\bL^{1/m})^{- \langle w, q \rangle}.
\]
In particular, $\mu_X(\trop^{-1}(w))$ is a unit in $\widehat{\sM_k}[\bL^{1/m}]$. Then by the above equality, \autoref{mainPropositionMotivicIntegralsCanonicalFantastacks}(\ref{stackMeasureOfFiberOfTrop}), and our choice of $\Theta_w$,
\begin{align*}
	(\#\beta^{-1}(w))(\bL^{1/m})^{j_w} \mu_X(\trop^{-1}(w)) &= (\#\beta^{-1}(w))\bL^{-d}(\bL-1)^d (\bL^{1/m})^{- \langle w, q \rangle}\\
	&= \mu_\cX(\sL(\pi)^{-1}(\trop^{-1}(w)))\\
	&= \Theta_w \mu_X(\trop^{-1}(w)),
\end{align*}
so
\[
	\Theta_w = (\#\beta^{-1}(w))(\bL^{1/m})^{j_w}.
\]
Therefore by our choice of $\Theta_w$ and $j_w$,
\begin{align*}
	\mu_\cX(\sL(\pi)^{-1}(C)) &= \Theta_w\mu_X(C)\\
	&= (\#\beta^{-1}(w))(\bL^{1/m})^{j_w} \mu_X(C)\\
	&= (\#\beta^{-1}(w)) \mu^\Gor_X(C)\\
	&= \int_C \sep_\pi \diff\mu^\Gor_X,
\end{align*}
where the last equality is by \autoref{MainPropositionFibersOfTheMapOfArcs}.
\end{proof}

\section{Fantastacks with special stabilizers: proof of \autoref{mainTheoremLiftingToUntwistedArcsSpecialStabilizers}}\label{sectionFantastacksSpecialStabilizsers}

The goal of this section is to prove \autoref{theoremFantastackConnectedStabilizers} below, which characterizes when a fantastack has only special stabilizers, and then to use this characterization to prove \autoref{mainTheoremLiftingToUntwistedArcsSpecialStabilizers}. For simplicity, we only state the criterion \autoref{theoremFantastackConnectedStabilizers} in the case where the good moduli space is affine (and defined by a full dimensional cone).

Throughout this section let $d \in \N$, let $N \cong \Z^d$ be a lattice, and let $T = \Spec(k[N^*])$ be the algebraic torus with co-character lattice $N$.

\begin{theorem}
\label{theoremFantastackConnectedStabilizers}
Let $\cX=\cF_{\sigma,\nu}$ be a fantastack with dense torus $T$ and keep the notation listed in \autoref{def:fantastack}. Then the following are equivalent.


\begin{enumerate}[(i)]

\item\label{theoremFantastackConnectedStabilizers::i} The stabilizers of $\cX$ are all special groups.

\item\label{theoremFantastackConnectedStabilizers::ii} For all $I \subset \{1, \dots, r\}$, the set $\{\nu(e_i) \mid i \in I\}$ is linearly independent if and only if it can be extended to a basis for $N$.

\item\label{theoremFantastackConnectedStabilizers::iii} For some $n \in \N$, we have $\cX\cong[\bA^r_k/\bG_m^n]$ where $\bG_m^n$ acts on $\bA^r_k$ with weights $w_1,\dots,w_r\in\Z^n$ such that for all $I \subset \{1, \dots, r\}$, the set $\{w_i \mid i \in I\}$ is linearly independent if and only if it can be extended to a basis for $\Z^n$.

\end{enumerate}
\end{theorem}

The following is an immediate consequence of \autoref{theoremFantastackConnectedStabilizers}, noting that taking canonical stack is compatible with taking products of toric varieties.

\begin{corollary}
\label{corollaryCanonicalFantastackConnectedStabilizers}
Let $\sigma$ be a pointed rational cone in $N_\R$, and let $\cX$ be the canonical stack over $X_\sigma$. 
If $v_1, \dots, v_r \in N$ are the first lattice points of the rays of $\sigma$, then the following are equivalent.
\begin{enumerate}[(i)]
\item The stabilizers of $\cX$ are all special groups.
\item for all $I \subset \{1, \dots, r \}$, the set $\{v_i \mid i \in I\}$ is linearly independent if and only if it can be extended to a basis for $N$.
\end{enumerate}
\end{corollary}

Before proving \autoref{theoremFantastackConnectedStabilizers}, we use \autoref{corollaryCanonicalFantastackConnectedStabilizers} to prove \autoref{mainTheoremLiftingToUntwistedArcsSpecialStabilizers}.

\begin{proof}[Proof of \autoref{mainTheoremLiftingToUntwistedArcsSpecialStabilizers}]
Let $X$ be a toric variety over $k$, let $\pi: \cX \to X$ be its canonical stack (see \autoref{remarkOtherFantastacks}), and assume that the stabilizers of $\cX$ are all special groups. By \autoref{canonicalFantastackOpenInvariantSubvariet} and the definition of $\sep_\pi$, we may assume that $X$ is the affine $T$-toric variety defined by a pointed rational cone $\sigma$ in $N_\R$. It is easy to check, for example by using \autoref{corollaryCanonicalFantastackConnectedStabilizers}, that because the stabilizers of $\cX$ are all special groups, the cone $\sigma$ is a face of a $d$-dimensional pointed rational cone in $N_\R$ whose associated toric variety has a canonical stack with only special stabilizers. Therefore we may assume that $\sigma$ is $d$-dimensional and use the notation listed in \autoref{notationforcanonicalfantastackforfulldimensionalcone}. Then by \autoref{fiberabovewisfiniteset}, \autoref{MainPropositionFibersOfTheMapOfArcs}, and the fact that
\[
	\mu_X\left( \sL(X) \setminus \bigcup_{w \in \sigma \cap N} \trop^{-1}(w) \right) = \mu_X\left( \sL(X \setminus T) \right) = 0,
\]
it is sufficient to show that $\beta$ is surjective.

Let $w \in \sigma \cap N$. We will use \autoref{corollaryCanonicalFantastackConnectedStabilizers} to show that $w$ is in the image of $\beta$. Let $\Sigma$ be a simplicial subdivision of $\sigma$ whose rays are all rays of $\sigma$, and let $\sigma_w \in \Sigma$ be a cone containing $w$. By \autoref{corollaryCanonicalFantastackConnectedStabilizers}, the cone $\sigma_w$ is unimodular, so $w$ is a positive integer combination of first lattice points of rays of $\sigma$. Therefore $w$ is in the image of $\beta$ by the definition of $\beta$, and we are done.
\end{proof}

The remainder of this paper will be used to prove \autoref{theoremFantastackConnectedStabilizers}.

\subsection{A combinatorial criterion for special stabilizers}

We start with some preliminary results, the first of which is a standard fact.

\begin{lemma}
\label{l:dual-coker}
If $A$ and $B$ are finite rank lattices and
\[
0\to A\xrightarrow{f} B\to C\to 0
\]
is a short exact sequence, then $\cok(f^*)$ is finite. Moreover, $C$ is torsion-free if and only if $f^*$ is surjective.
\end{lemma}
\begin{proof}
Letting $C_{\tor}$ be the torsion part of $C$, applying $\Hom(-,\Z)$ to the short exact sequence
\[
0\to C_{\tor}\to C\to \overline{C}\to 0,
\]
we see $\ext^1(C,\Z)\cong\ext^1(C_{\tor},\Z)$, which is finite. Then from the exact sequence
\[
0\to C^*\to B^*\xrightarrow{f^*} A^*\to \ext^1(C,\Z)\to 0,
\]
we see $\cok(f^*)$ is finite and that $f^*$ is surjective if and only if $\ext^1(C,\Z)=0$ if and only if $C$ is a lattice.
\end{proof}

\begin{lemma}
\label{l:tf-quot-Z-basis}
Let $A$ be a lattice and suppose $v_1,\dots,v_r\in A$ span $A_\Q$. Then the following conditions are equivalent:
\begin{enumerate}
\item\label{l:tf-quot-Z-basis::tf} $A/\sum_{i\in S}\Z v_i$ is torsion-free for all $S\subseteq\{1,\dots,r\}$,
\item\label{l:tf-quot-Z-basis::Z-basis} for every $S\subseteq\{1,\dots,r\}$, if $\{v_i\mid i\in S\}$ is a $\Q$-basis for $A_\Q$, then it is a $\Z$-basis for $A$.
\end{enumerate}
\end{lemma}
\begin{proof}
To ease notation, let $L_S:=\sum_{i\in S}\Z v_i$ and $L^{\sat}_S:=A\cap(L_S)_\Q$ for all $S\subseteq\{1,\dots,r\}$. Note that $A/L^{\sat}_S$ is torsion-free, so $L^{\sat}_S$ is a direct summand of $A$. It follows $A/L_S$ is torsion-free if and only if $L^{\sat}_S/L_S$ is torsion-free, and since $L^{\sat}_S/L_S$ is finite, we see that condition (\ref{l:tf-quot-Z-basis::tf}) is equivalent to $L_S=L^{\sat}_S$.

Now, suppose that condition (\ref{l:tf-quot-Z-basis::tf}) holds and let $S\subseteq\{1,\dots,r\}$ such that $\{v_i\mid i\in S\}$ is a $\Q$-basis for $A_\Q$. Then $L_S=L^{\sat}_S=A$, so $\{v_i\mid i\in S\}$ is a $\Z$-basis for $A$.

Conversely, suppose condition (\ref{l:tf-quot-Z-basis::Z-basis}) holds, and let $S\subseteq\{1,\dots,r\}$ be any subset. Choose $S'\subseteq S$ such that $\{v_i\mid i\in S'\}$ form a $\Q$-basis for $(L_S)_\Q$. Since the $\Q$-span of $v_1,\dots,v_r$ is $A_\Q$, we can choose $S''\subseteq\{1,\dots,r\}\setminus S$ such that $\{v_i\mid i\in S'\cup S''\}$ form a $\Q$-basis for $A_\Q$. It follows that $\{v_i\mid i\in S'\cup S''\}$ is a $\Z$-basis for $A$, and hence $A\cong L_{S'}\oplus L_{S''}$.

To show condition (\ref{l:tf-quot-Z-basis::tf}) holds, i.e., $A/L_S$ is torsion-free, it thus suffices to show $L_{S'}=L_S$. To see why this equality holds, let $j\in S$ and write $v_j=\sum_{i\in S'}a'_iv_i+\sum_{i\in S''}a''_iv_i$ with $a'_i,a''_i\in\Z$. On the other hand, by definition of $S'$, we can write $v_j=\sum_{i\in S'}b'_iv_i$ with $b'_i\in\Q$. Equating our two expressions and using the fact that $\{v_i\mid i\in S'\cup S''\}$ is a $\Q$-basis for $A_\Q$, we see $a''_i=0$ for all $i\in S''$. So, $v_j\in L_{S'}$.
\end{proof}

\begin{lemma}
\label{l:torus-action-connected-stabs}
Let $\cX=[\bA^r_k/\bG_m^n]$ where $\bG_m^n$ acts with weights $w_1,\dots,w_r\in\Z^n$ which span $\Q^n$.\footnote{Note that we do not assume $\cX$ is a fantastack here.} Then $\cX$ has special stabilizers if and only if for every $S\subseteq\{1,\dots,r\}$, if $\{w_i\mid i\in S\}$ is a $\Q$-basis for $\Q^n$, then it is a $\Z$-basis for $\Z^n$.
\end{lemma}
\begin{proof}
Since the stabilizers of $\cX$ are subgroups of $\bG_m^n$, they are special if and only if they are connected. Let $w_i=(a_{i1},\dots,a_{in})\in\Z^n$. Given a point $x=(x_1,\dots,x_r)\in\bA^r_k$, let $S_x=\{i\mid x_i\neq0\}$. Then the stabilizer $G_x$ of $x$ is the set of $(\lambda_1,\dots,\lambda_n)\in\bG_m^n$ such that $\prod_{j=1}^n\lambda_j^{a_{ij}}=1$ for all $i\in S_x$. In other words, we have a short exact sequence
\[
1\to G_x\to \bG_m^n\xrightarrow{\varphi} \bG_m^{S_x}\to1
\]
where for $i\in S_x$, the $i$-th coordinate of $\varphi(\lambda_1,\dots,\lambda_n)$ is given by $\prod_{j=1}^n\lambda_j^{a_{ij}}$. Taking Cartier duals $D(-):=\Hom(-,\bG_m)$, we see $D(G_x)$ is the cokernel of the map $\Z^{S_x}\to\Z^n$ sending the $i$-th standard basis vector to $w_i$. By Cartier duality, $G_x$ is connected if and only if $D(G_x)$ is torsion-free. We have thus shown that $\cX$ has connected stabilizers if and only if $\Z^n/\sum_{i\in S}\Z w_i$ is torsion-free for all subsets $S\subseteq\{1,\dots,r\}$. \autoref{l:tf-quot-Z-basis} then finishes the proof.
\end{proof}

We now turn to \autoref{theoremFantastackConnectedStabilizers}.

\begin{proof}[{Proof of \autoref{theoremFantastackConnectedStabilizers}}]
The map $\nu\colon\Z^r\to N$ has finite cokernel, or equivalently $\nu^*$ is injective. So, we have a short exact sequence
\[
\xymatrix{
0\ar[r] & M\ar[r]^-{\nu^*} & (\Z^r)^*\ar[r]^-{\alpha} & A\ar[r] & 0.
}
\]
By the construction of fantastacks, we have $\cX=[\bA^r_k/G]$, where $G$ is the Cartier dual of $A$. Note that $G$ is the stabilizer of the origin and it is connected if and only if $A$ is torsion-free. By \autoref{l:dual-coker}, this is equivalent to surjectivity of $\nu$.

So, we may now assume $\nu$ is surjective and, in light of \autoref{l:torus-action-connected-stabs}, we need only establish the equivalence of conditions (\ref{theoremFantastackConnectedStabilizers::ii}) and (\ref{theoremFantastackConnectedStabilizers::iii}). By \autoref{l:tf-quot-Z-basis}, condition (\ref{theoremFantastackConnectedStabilizers::ii}) holds if and only if $\cok(\nu|_{\Z^S})$ is torsion-free for all subsets $S\subseteq\{1,\dots,r\}$. Letting $e_i^*\in(\Z^r)^*$ denote the dual linear functional, notice that $\alpha(e_i^*)$ is the $i$-th weight for the $G$-action on $\bA^r_k$. Given any subset $S\subseteq\{1,\dots,r\}$, we let $S'$ denote the complement of $S$. We have a natural inclusion $(\Z^{S'})^*\subseteq(\Z^r)^*$ with cokernel $(\Z^S)^*$. Another application of \autoref{l:tf-quot-Z-basis} 
shows that condition (\ref{theoremFantastackConnectedStabilizers::iii}) holds if and only if $Q_S:=A/\alpha((\Z^{S'})^*)$ is torsion-free for all $S$. We show that this latter condition is equivalent to $\cok(\nu|_{\Z^S})$ being torsion-free for all subsets $S\subseteq\{1,\dots,r\}$.

Consider the diagram
\[
\xymatrix{
& 0 & 0 & 0 & \\
0\ar[r] & B\ar[r]\ar[u] & (\Z^S)^*\ar[r]\ar[u] & Q_S\ar[r]\ar[u] & 0\\
0\ar[r] & M\ar[r]^-{\nu^*}\ar[u] & (\Z^r)^*\ar[r]^-{\alpha}\ar[u]^-{\pi} & A\ar[r]\ar[u] & 0\\
0\ar[r] & C\ar[r]\ar[u] & (\Z^{S'})^*\ar[r]\ar[u] & \alpha((\Z^{S'})^*)\ar[r]\ar[u] & 0\\
& 0\ar[u] & 0\ar[u] & 0\ar[u] & 
}
\]
where $B$ is the image of $\pi\circ\nu^*$ and $C=(\Z^{S'})^*\cap\ker\alpha$; in particular all rows and columns are exact. Note that all $\Z$-modules in the above diagram are torsion-free with the possible exception of $Q_S$. Applying $\Hom(-,\Z)$, we have the diagram
\[
\xymatrix{
0\ar[r] & Q_S^*\ar[r]\ar[d] & \Z^S\ar[r]^-{\nu|_{\Z^S}}\ar[d]^-{\pi^*} & B^*\ar[r]\ar[d] & \ext^1(Q_S,\Z)\ar[r] & 0\\
0\ar[r] & A^*\ar[r]^-{\alpha^*} & \Z^r\ar[r]^-{\nu} & N\ar[r] & 0 & 
}
\]
where all rows are exact and all vertical maps are injective. We see then that
\[
\cok(\nu|_{\Z^S})=\ext^1(Q_S,\Z).
\]
Letting $Q_{S,\tor}\subseteq Q_S$ denote the torsion part, we have $\ext^1(Q_S,\Z)=\ext^1(Q_{S,\tor},\Z)$, which is finite, so $\cok(\nu|_{\Z^S})$ is torsion-free if and only if $\ext^1(Q_S,\Z)=0$ if and only if $Q_S$ is torsion-free.
\end{proof}

\section{Quotients by $\mathrm{SL}_2$ and cones over Grassmannians}\label{quotientsbySL2andconesoverGrassmannians}

Until now, we have only considered stacks with abelian stabilizers. In this section, we verify \autoref{conjectureGorensteinMeasureAndStackMeasure} and \autoref{conjectureSpecialStabilizersLiftArcs} for non-trivial examples where the stabilizers are non-abelian. Our primary running example in this section is a stack $\cX=[\bA^8/\SL_2]$ whose good moduli space is the affine cone over the Grassmannian $\Gr(2,4)$ with respect to the Pl\"ucker embedding; we will additionally verify \autoref{conjectureSpecialStabilizersLiftArcs} and answer \autoref{questionMuXsepInfinity} for a stack $\cX=[\bA^{2r}/\SL_2]$ whose good moduli space is the affine cone over the Grassmannian $\Gr(2,r)$. The cases we consider are particularly interesting since, unlike fantastacks, we show here that $\mu_X(\sep_\pi^{-1}(\infty)) \neq 0$.

We begin by setting up some notation. Throughout this section fix some $r \in \Z_{>1}$ and let $[r]=\{1,\dots,r\}$. Let $\widetilde{X}=\bA^{2\times r}$ which we think of as $2\times r$ matrices with coordinates
\[
	\begin{pmatrix} y_{1,1} & \hdots & y_{1,r} \\ y_{2,1} & \hdots & y_{2,r} \end{pmatrix}.
\]
There is an $\SL_2$-action on $\widetilde{X}$ given by left multiplication. The invariant functions are generated by the $2\times 2$ minors, and so the quotient $X=\widetilde{X}/\SL_2\subset \bA^{[r] \choose 2}_k$ is the affine cone over the Grassmannian $\Gr(2, r)$ with respect to the Pl\"{u}cker embedding of $\Gr(2,r)$. 
Letting $\{x_{\ell,m}\mid 1\leq\ell < m\leq r\}$ be the coordinates on $\bA^{[r] \choose 2}_k$, the quotient map $\widetilde{X}\to X$ sends each $x_{\ell,m}|_X$ to the $(\ell, m)$th minor of the matrix above. Set $\cX = [ \widetilde{X} / \SL_2 ]$, let $\rho: \widetilde{X} \to \cX$ be the quotient map, and let $\pi\colon\cX\to X$ be the good moduli space map. Note that all stabilizers of $\cX$ are special: the zero matrix has stabilizer $\SL_2$, a full rank matrix has trivial stabilizer, and a rank $1$ matrix has stabilizer $\bG_a$.

\begin{remark}
\label{SL2ExampleExceptionalLocusCodimension2}
When $r \geq 3$, the exceptional locus of $\pi$ has codimension at least 2: it is straightforward to check that if $U \subset X$ is the intersection of $X$ with the complement $\bG_m^{[n] \choose 2}$ of the coordinate axes of $\bA_k^{[n] \choose 2}$, then $\pi|_{\pi^{-1}(U)}: \pi^{-1}(U)  \to U$ is an isomorphism, and when $r \geq 3$, the complement of $\pi^{-1}(U)$ in $\cX$ has codimension at least 2.
\end{remark}

The next proposition verifies \autoref{conjectureGorensteinMeasureAndStackMeasure} for an infinite collection of pairs $(\cC, C)$ such that the $C$ cover $\sL(X)$ up to  measure $0$. Our collection of $(\cC,C)$ involve non-abelian stabilizers in a non-trivial way by the fourth bullet point below. Furthermore, for every $C$ we prove \autoref{conjectureGorensteinMeasureAndStackMeasure} for multiple different choices of $\cC$, thereby illustrating the flexibility of the conjecture in choosing $\cC$.

\begin{proposition}
\label{Conjecture1forGr24}
Let $r\in\{3,4\}$ and $\mathcal{S} \subset \mathcal{S}'$ be the collections of measurable subsets of $|\sL(\cX)|\times\sL(X)$ given in \autoref{def:infinite-collection-nontriv-nonab}. Then
\begin{itemize}
\item $\mathcal{S}$ is an infinite set,
\item for all $(\cC, C) \in \mathcal{S}'$, the measurable sets $\cC$ and $C$ both have nonzero measure, 
\item the complement of $\bigcup_{(\cC,C)\in\mathcal{S}'}C$ has measure $0$,
\item for all $(\cC,C)\in\mathcal{S}$, every $\psi\in\cC$ has special point mapping to the point with $\SL_2$ stabilizer,
\item $C= \sL(\pi)(\cC)$ for all $(\cC,C)\in\mathcal{S}'$, and
\item \autoref{conjectureGorensteinMeasureAndStackMeasure} holds for all $(\cC,C)\in\mathcal{S}'$, i.e.,
\[
	\mu_\cX(\cC) = \mu_\cX(\cC \cap \sL(\pi)^{-1}(C)) =  \int_C \sep_{\pi, \cC} \diff\mu^\Gor_X.
\]
\end{itemize}
\end{proposition}

The next two propositions verify \autoref{conjectureSpecialStabilizersLiftArcs} and answer \autoref{questionMuXsepInfinity}.

\begin{proposition}
\label{Conjecture2forGr2r}
For $r\geq2$, \autoref{conjectureSpecialStabilizersLiftArcs} holds for $\pi\colon\cX\to X$.
\end{proposition}

\begin{proposition}
\label{QuestionforGr2r}
For $r\geq2$, $\mu_X(\sep_\pi^{-1}(\infty)) \neq 0$, answering \autoref{questionMuXsepInfinity} for $\pi\colon\cX\to X$.
\end{proposition}

For any $\bw \in (\N \cup \{\infty\})^{[r] \choose 2}$, let $C^{(\bw)} \subset \sL(X)$ denote the subset of arcs whose prescribed vanishing orders with respect to the Pl\"{u}cker coordinate are given by $\bw$. More precisely,
\[
	C^{(\bw)} = \bigcap_{\ell< m \in [r]} \ord_{x_{\ell,m} |_X}^{-1}(w_{\ell,m}) \subset \sL(X).
\]
Note that if $\bw \in \N^{[r] \choose 2}$, then $C^{(\bw)}$ is a cylinder.

\begin{remark}
\label{SL2ExampleRemarkAboutNegligibleSubsetsOfX}
If $\bw \in( (\N \cup \{\infty\})^{[r] \choose 2} \setminus \N^{[r] \choose 2})$, then $C^{(\bw)}$ is a measurable set with $\mu_X(C^{(\bw)}) = 0$.
\end{remark} 

With this notation, we now prove \autoref{Conjecture2forGr2r} and \autoref{QuestionforGr2r}.

\begin{proof}[{Proof of \autoref{Conjecture2forGr2r}}]
It suffices to show that away from a set of measure zero, every arc $\varphi\in\sL(X)(k')$ lifts to an arc in $\sL(\cX)(k')$ where $k'$ is an extension field of $k$. By \autoref{SL2ExampleRemarkAboutNegligibleSubsetsOfX}, we may assume $\varphi\in C^{(\bw)}$ with $\bw =(w_{i,j}) \in \N^{[r] \choose 2}$. Without loss of generality, we may additionally assume $w_{1,2}=\min(w_{i,j})$. Let the map on coordinate rings induced by $\varphi\colon\Spec(k'\llbracket t \rrbracket)\to X$ send $x_{i,j}|_X$ to $g_{i,j}\in k'\llbracket t \rrbracket$. Let $\widetilde\psi\in\sL(\widetilde{X})$ be the arc given by the matrix
\begin{equation}\label{matricesProofConjecture2forGr2r}
	\begin{pmatrix} g_{1,2} & 0 & -g_{2,3} & -g_{2,4} & \hdots & -g_{2,r} \\ 0 & 1 & g_{1,3}g_{1,2}^{-1} & g_{1,4}g_{1,2}^{-1} & \hdots & g_{1,r}g_{1,2}^{-1} \end{pmatrix}.
\end{equation}
Note that all $g_{i,j}$ are non-zero since $\bw \in \N^{[r] \choose 2}$, and all entries of the matrix are in $k'\llbracket t \rrbracket$ since $w_{1,2}=\min(w_{i,j})$. Note further that the $(i,j)$th minor of the matrix is precisely $g_{i,j}$; this is clear when $i\in\{1,2\}$ so the Pl\"{u}cker relations ensure this remains true for all $(i,j)$. As a result, $\widetilde\psi$ is a lift of $\varphi$, and hence $\psi=\sL(\rho)(\widetilde\psi)\in\sL(\cX)$ is a lift of $\varphi$.
\end{proof}

\begin{proof}[{Proof of \autoref{QuestionforGr2r}}]
As in the proof of \autoref{Conjecture2forGr2r}, it suffices to show that $\sep_\pi$ is infinite on $C^{(\bw)}$ under the assumption that $\bw =(w_{i,j}) \in \N^{[r] \choose 2}$ and $1\leq w_{1,2}=\min(w_{i,j})$. Again let $\varphi\in C^{(\bw)}$ and let $\varphi\colon\Spec(k'\llbracket t \rrbracket)\to X$ send $x_{i,j}|_X$ to $g_{i,j}\in k'\llbracket t \rrbracket$. Then for any $h\in k'\llbracket t \rrbracket$, we obtain a lift $\widetilde\psi_h\in\sL(\widetilde{X})$ of $\varphi$, where $\widetilde\psi_h$ is the arc given by the matrix
\[
	\begin{pmatrix} g_{1,2} & 0 & -g_{2,3} & -g_{2,4} & \hdots & -g_{2,r} \\ h & 1 & (g_{1,3}-hg_{2,3})g_{1,2}^{-1} & (g_{1,4}-hg_{2,4})g_{1,2}^{-1} & \hdots & (g_{1,r}-hg_{2,r})g_{1,2}^{-1} \end{pmatrix}.\]
Notice that for any $h,h'\in k'\llbracket t \rrbracket$ satisfying $h'-h$ is not divisible by $t^{w_{1,2}}$, there is no $A\in\SL_2(k'\llbracket t \rrbracket)$ for which $A\widetilde\psi_h=\widetilde\psi_{h'}$. Indeed, the unique $A\in\SL_2(k'\llparenthesis t \rrparenthesis)$ with 
\[
	A\begin{pmatrix} g_{1,2} & 0  \\ h & 1\end{pmatrix}=\begin{pmatrix} g_{1,2} & 0  \\ h' & 1\end{pmatrix}
\]
is given by 
\[
	A=\begin{pmatrix} 1 & 0  \\ (h'-h)g_{1,2}^{-1} & 1\end{pmatrix}
\]
which is not in $\SL_2(k'\llbracket t \rrbracket)$ by the assumption on $h'-h$. Noting that $w_{1,2} \geq 1$, it follows that $\sep_\pi$ is infinite on $C^{(\bw)}$.
\end{proof}

There are two features of the proof of \autoref{Conjecture2forGr2r} that we wish to highlight. First, for the matrices in (\ref{matricesProofConjecture2forGr2r}) to define arcs of $\sL(\cX)$, we did not need $w_{1,2}=\min(w_{i,j})$, but rather $w_{1,2} \leq w_{1,3}, \dots, w_{1, r}, w_{2,3}, \dots, w_{2,r}$. Second, for any $i\leq w_{1,2}$ we have many more lifts of $\varphi$ given by matrices of the form
\[
	\begin{pmatrix} g_{1,2}t^{-i} & 0 & -g_{2,3}t^{-i} & -g_{2,4}t^{-i} & \hdots & -g_{2,r}t^{-i} \\ 0 & t^i & g_{1,3}g_{1,2}^{-1}t^i & g_{1,4}g_{1,2}^{-1}t^i & \hdots & g_{1,r}g_{1,2}^{-1}t^i \end{pmatrix}.
\]

With this as motivation, we introduce the following sets. For any $i, w_{1,2} \in \N$ and $\bw_1 = (w_{1,3}, \dots, w_{1,r}), \bw_2 = (w_{2,3}, \dots, w_{2, r}) \in \N^{r-2}$ satisfying
\begin{align*}\label{ConditionOnValuationsForSL2Example}
	2i \leq w_{1,2} \leq w_{1,3}, \dots, w_{1, r}, w_{2,3}, \dots, w_{2,r}, \tag{$\star$}
\end{align*}
let $Z^{(i, w_{1,2}, \bw_1, \bw_2)}$ be the subset of $\sL(\widetilde{X})$ whose $k'$ valued points, for any extension $k'$ of $k$, are the $2 \times r$ matrices of the form
\[
	\begin{pmatrix} f_{1,1} & 0 & f_{1,3} & f_{1,4} & \hdots &  f_{1,r} \\ 0 & t^i & f_{2,3} & f_{2,4} & \hdots & f_{2,r} \end{pmatrix},
\]
where $f_{1,1}, f_{1,3}, \dots, f_{1,r}, f_{2,3}, \dots, f_{2,r} \in k'\llbracket t \rrbracket$ satisfies
\begin{align*}
	\ord_t(f_{1,1}) &= w_{1,2} - i,\\
	\ord_t(f_{1,j}) &= w_{2,j} - i, &\text{for all $j = 3, \dots, r$},\\
	\ord_t(f_{2,j}) &= w_{1,j} - w_{1,2} + i, &\text{for all $j = 3, \dots, r$}.
\end{align*}
Note that the condition (\ref{ConditionOnValuationsForSL2Example}) guarantees that these vanishing orders are nonnegative. For any $i, w_{1,2} \in \N$ and $\bw_1, \bw_2 \in \N^{r-2}$ satisfying (\ref{ConditionOnValuationsForSL2Example}), set 
\[
	\cC^{(i, w_{1,2}, \bw_1, \bw_2)} = \sL(\rho)(Z^{(i, w_{1,2}, \bw_1, \bw_2)}) \subset |\sL(\cX)|
\]
and set
\[
	C^{(w_{1,2}, \bw_1, \bw_2)} = \bigcup_{\bw} C^{(\bw)},
\]
where $\bw$ varies over all elements of $(\N \cup \{\infty\})^{[r] \choose 2}$ whose $(\ell, m)$th entry is equal to $w_{\ell, m}$ for all $(\ell, m) = (1,2), (1,3), \dots, (1,r), (2,3), \dots, (2,r)$.

Lastly, although (for ease of notation) we have chosen to define all of our sets above with $w_{1,2}$ playing a special role, by symmetry of the Pl\"{u}cker coordinates, we obtain many analogous sets as follows. Note that there is a natural $S_r$-action on $\cX$ given by permuting columns of the matrices in $\widetilde{X}$; by functoriality, this descends to an $S_r$-action on the good moduli space $X$. Let
\[
	\cC^{(i, w_{1,2}, \bw_1, \bw_2)}_\sigma = \sigma(\cC^{(i, w_{1,2}, \bw_1, \bw_2)}) \subset |\sL(\cX)|\quad\textrm{and}\quad C^{(w_{1,2}, \bw_1, \bw_2)}_\sigma=\sigma(C^{(w_{1,2}, \bw_1, \bw_2)})\subset\sL(X).
\]

\begin{definition}
\label{def:infinite-collection-nontriv-nonab}
Let $\mathcal{S}'$ be the collection of pairs
\[
(\cC^{(i, w_{1,2}, \bw_1, \bw_2)}_\sigma,C^{(w_{1,2}, \bw_1, \bw_2)}_\sigma)
\]
with $i, w_{1,2} \in \N$ and $\bw_1, \bw_2 \in \N^{r-2}$ satisfying (\ref{ConditionOnValuationsForSL2Example}), and $\sigma\in S_r$. Let $\mathcal{S}\subset\mathcal{S}'$ be the subset consisting of pairs where $i>0$.

\end{definition}

\begin{remark}
\label{rmk:SL2specialPoint}
By definition, for every $(\cC,C)\in\mathcal{S}$ and $\psi\in\cC$, the matrix corresponding to $\psi$ specializes to the $0$ matrix. As a result, the special point of $\psi$ maps to the point in $\cX$ with $\SL_2$ stabilizer, justifying the fourth bullet point of \autoref{Conjecture1forGr24}.
\end{remark}

\begin{remark}
\label{rmk:reduceToSigmaId}
Since the $S_r$-actions on $\sL(\cX)$ and $\sL(X)$ are measure-preserving, to prove \autoref{conjectureGorensteinMeasureAndStackMeasure} for $(\cC^{(i, w_{1,2}, \bw_1, \bw_2)}_\sigma,C^{(w_{1,2}, \bw_1, \bw_2)}_\sigma)$, it is enough to prove the conjecture for $(\cC^{(i, w_{1,2}, \bw_1, \bw_2)},C^{(w_{1,2}, \bw_1, \bw_2)})$.
\end{remark}

The remainder of this section is concerned with the proof of \autoref{Conjecture1forGr24}. To compute $\mu^{\Gor}_X(C^{(w_{1,2}, \bw_1, \bw_2)})$, it is enough to compute $\mu^{\Gor}_X(C^{(\bw)})$, which is done in the next proposition. This is the only result in this section where we impose that $r\in\{3,4\}$ as opposed to $r\geq2$.

\begin{proposition}
\label{SL2ExampleComputationOfGorensteinMeasure}
\begin{enumerate}[(a)]

\item Suppose $r = 3$, and let $\bw = (w_{1,2}, w_{1,3}, w_{2,3}) \in \N^{[r] \choose 2}$. Then $X$ is smooth, and
\[
	\mu^{\Gor}_X(C^{(\bw)}) = (\bL-1)^3\bL^{-3-w_{1,2} - w_{1,3} - w_{2,3}}.
\]

\item Suppose $r = 4$, let $\bw = (w_{1,2}, w_{3,4}, w_{1,3}, w_{2,4}, w_{1,4}, w_{2,3}) \in \N^{[r] \choose 2}$, and set 
\begin{align*}
	m &= \min(w_{1,2} + w_{3,4}, w_{1,3} + w_{2,4}, w_{1,4} + w_{2,3}),\\
	j &= -5+ m - w_{1,2} - w_{3,4} - w_{1,3} - w_{2,4} - w_{1,4} - w_{2,3}.
\end{align*}
Then $X$ has log-terminal singularities, and
\[
	\mu^{\Gor}_X(C^{(\bw)}) = \begin{cases} 0, & \text{exactly one of $w_{1,2} + w_{3,4}, w_{1,3} + w_{2,4}, w_{1,4} + w_{2,3}$ equals $m$} \\ (\bL-1)^5 \bL^{j}, & \text{exactly two of $w_{1,2} + w_{3,4}, w_{1,3} + w_{2,4}, w_{1,4} + w_{2,3}$ equal $m$} \\ (\bL-1)^4(\bL-2) \bL^{j}, & \text{all three of $w_{1,2} + w_{3,4}, w_{1,3} + w_{2,4}, w_{1,4} + w_{2,3}$ equal $m$} \end{cases}
\]

\end{enumerate}
\end{proposition}

\begin{proof}
\begin{enumerate}[(a)]

\item This is immediate from the fact that $X = \bA_k^{[3] \choose 2}$.

\item We only sketch the proof, as the details are somewhat extensive but do not involve any novel ideas. We have $X$ is defined in $\bA_k^{[4] \choose 2}$ by the vanishing of a single polynomial $x_{1,2}x_{3,4} - x_{1,3}x_{2,4} + x_{1,4}x_{2,3}$. This polynomial is non-degenerate with respect to its Newton polyhedron, so the claim follows from standard techniques: a certain toric modification of $\bA_k^{[4] \choose 2}$ gives a resolution of $X$, see e.g., \cite{SchepersVeys}, especially the discrepancy computation in the proof of \cite[Proposition 2.3]{SchepersVeys}. Note that to get our claim from these techniques, one needs to verify that $\e(X \cap \bG_m^{[4] \choose 2}) = (\bL-1)^3(\bL-2)$, which for example, follows from the relationship between $X \cap \bG_m^{[4] \choose 2}$ and the realization space of the rank 2 uniform matroid on 4 elements (see e.g., \cite[Proposition 9.7]{Katz}), which has class $\bL-2$.\qedhere
\end{enumerate}
\end{proof}

We next turn to the computation of $\mu_\cX(\cC^{(i, w_{1,2}, \bw_1, \bw_2)})$. This is the technical heart of this section.

\begin{proposition}
\label{SL2ExampleVolumeOfStackCylinders}
Let $i, w_{1,2} \in \N$ and $\bw_1 = (w_{1,3}, \dots, w_{1,r})$, $\bw_2 = (w_{2,3}, \dots, w_{2, r})$ satisfy $($\ref{ConditionOnValuationsForSL2Example}$)$. Then $\cC^{(i, w_{1,2}, \bw_1, \bw_2)}$ is a cylinder in $|\sL_n(\cX)|$ and
\[
	\mu_\cX(\cC^{(i, w_{1,2}, \bw_1, \bw_2)}) = (\bL-1)^{2r-3} \bL^{-(2r-3)+(r-4)w_{1,2}-(w_{1,3} + \dots + w_{1,r} + w_{2,3} + \dots + w_{2,r})}.
\]
\end{proposition}

\begin{proof}
To ease notation, set $Z = Z^{(i, w_{1,2}, \bw_1, \bw_2)}$ and $\cC = \cC^{(i, w_{1,2}, \bw_1, \bw_2)}$. For any $n \in \N$, let $Z_n$ be the locally closed subscheme of $\sL_n(\widetilde{X})$ whose $A$ valued points, for any $k$-algebra $A$, are the $2 \times r$ matrices of the form
\[
	\begin{pmatrix} u_{1,1} t^{w_{1,2} - i} & 0 & u_{1,3} t^{w_{2,3} - i} & u_{1,4} t^{w_{2,4} - i} & \hdots &  u_{1,r} t^{w_{2,r} - i} \\ 0 & t^i & u_{2,3}t^{w_{1,3} - w_{1,2} + i} & u_{2,4}t^{w_{1,4} - w_{1,2} + i} & \hdots & u_{2,r}t^{w_{1,r} - w_{1,2} + i} \end{pmatrix},
\]
where $u_{1,1}, u_{1,3}, \dots, u_{1, r}, u_{2,3}, \dots, u_{2, r}$ are all units in $A[t]/(t^{n+1})$, so by construction
\[
	Z_n = \theta_n(Z).
\]
Thus for any $n \in \N$,
\begin{equation}\label{EquationTruncationOfStackSetIsImageOfTruncationOfCoverSet}
	\theta_n(\cC) = \sL_n(\rho)(Z_n).
\end{equation}

Now set $n_0 = \max(w_{1,2}, i, w_{2,3}, \dots, w_{2, r}, w_{1,3} + i, \dots, w_{1,r} + i)$. We will show that
\[
	\cC = \theta_{n_0}^{-1}\left( \sL_{n_0}(\rho)(Z_{n_0}) \right).
\]
One inclusion is clear from (\ref{EquationTruncationOfStackSetIsImageOfTruncationOfCoverSet}).
To show the other inclusion, let $k'$ be a field extension of $k$ and $\psi$ be a $k'$-point of $\theta_{n_0}^{-1}\left( \sL_{n_0}(\rho)(Z_{n_0}) \right)$. For the sake of showing the desired inclusion we may extend $k'$, so we may assume there exists a $k'$-point $\widetilde{\psi}_{n_0}$ of $Z_{n_0}$ such that $\sL_n(\rho)(\widetilde{\psi}_{n_0}) \cong \theta_{n_0}(\psi)$. Because $\rho: \widetilde{X} \to \cX$ is smooth, there exists a $k'$-point $\widetilde{\psi}$ of $\sL(\widetilde{X})$ such that $\theta_{n_0}(\widetilde{\psi}) = \widetilde{\psi}_{n_0}$ and $\sL(\rho)(\widetilde{\psi}) \cong \psi$. We want to show that there exists some $g \in \SL_2(k'\llbracket t \rrbracket)$ such that $g \cdot \widetilde{\psi} \in Z$, as this would imply that $\psi \cong \sL(\rho)(\widetilde{\psi}) \cong \sL(\rho)(g \cdot \widetilde{\psi}) \in \sL(\rho)(Z) = \cC$. By our choice of $n_0$ and the fact that $\theta_{n_0}(\widetilde{\psi}) \in Z_{n_0}$, the arc $\tilde{\psi}$ is equal to
\[
	\begin{pmatrix} f_{1,1} & h_{1,2}t^{n_0+1} & f_{1,3} & f_{1,4} & \hdots &  f_{1,r} \\ h_{2,1}t^{n_0+1} & t^i + h_{2,2}t^{n_0+1} & f_{2,3} & f_{2,4} & \hdots & f_{2,r} \end{pmatrix},
\]
for some $h_{1,2}, h_{2,1}, h_{2,2} \in k'\llbracket t \rrbracket$ and some $f_{1,1}, f_{1,3}, \dots, f_{1,r}, f_{2,3}, \dots, f_{2,r} \in k'\llbracket t \rrbracket$ satisfying
\begin{align*}
	\ord_t(f_{1,1}) &= w_{1,2} - i,\\
	\ord_t(f_{1,j}) &= w_{2,j} - i, &\text{for all $j = 3, \dots, r$},\\
	\ord_t(f_{2,j}) &= w_{1,j} - w_{1,2} + i, &\text{for all $j = 3, \dots, r$}.
\end{align*}
Let $h \in k\llbracket t \rrbracket$ be the $(1,2)$th minor of $\widetilde{\psi}$. We have $h \neq 0$ because $h$ and $f_{1,1}t^i$  have the same image in $k'[t]/(t^{n_0+1})$ and the latter image is nonzero by $n_0 \geq w_{1,2}$. Thus there exists some $g \in \SL_2(k'\llparenthesis t \rrparenthesis)$ such that
\[
	g \cdot \begin{pmatrix} f_{1,1} & h_{1,2}t^{n_0+1} \\ h_{2,1}t^{n_0+1} & t^i + h_{2,2}t^{n_0+1} \end{pmatrix} = \begin{pmatrix} ht^{-i} & 0 \\ 0 & t^i \end{pmatrix}.
\]
Then
\[
	g^{-1} = \begin{pmatrix} f_{1,1} & h_{1,2}t^{n_0+1} \\ h_{2,1}t^{n_0+1} & t^i + h_{2,2}t^{n_0+1} \end{pmatrix} \begin{pmatrix} ht^{-i} & 0 \\ 0 & t^i \end{pmatrix}^{-1} = \begin{pmatrix} f_{1,1}t^ih^{-1} & h_{1,2}t^{n_0+1-i} \\ h_{2,1}t^{n_0+1+i}h^{-1} & t^i + h_{2,2}t^{n_0+1-i} \end{pmatrix}.
\]
Because $h$ and $f_{1,1}t^i$ have the same image in $k'[t]/(t^{n_0+1})$ and $n_0 \geq w_{1,2} =\ord_t(f_{1,1}t^i)$, we have that $\ord_t(h) = w_{1,2}$. Together with the fact that $n_0 \geq i$, this implies that the entries of $g^{-1}$ are all elements of $k'\llbracket t \rrbracket$, so $g \in \SL_2(k'\llbracket t \rrbracket)$. We will now show that $g \cdot \widetilde{\psi} \in Z$. We see that $g \cdot \widetilde{\psi}$ is equal to
\[
	\begin{pmatrix} ht^{-i} & 0 & h_{1,3} & h_{1,4} & \hdots &  h_{1,r} \\ 0 & t^i & h_{2,3} & h_{2,4} & \hdots & h_{2,r} \end{pmatrix}
\]
for some $h_{1,3}, \dots, h_{1,r}, h_{2,3}, \dots, h_{2,r} \in k'\llbracket t \rrbracket$, and we have already shown that $\ord_t(ht^{-i}) = w_{1,2} - i$. For any $j = 3, \dots, r$ let $q_{1,j} \in k'\llbracket t \rrbracket$ (resp. $q_{2,j} \in k'\llbracket t \rrbracket$) be the $(1,j)$th (resp. $(2,j)$th) minor of $\psi$. Then for any $j = 3, \dots, r$ we have that $-f_{1,j}t^{i}$ and $q_{2,j}$ (resp. $f_{1,1}f_{2,j}$ and $q_{1,j}$) have the same image in $k'[t]/(t^{n_0+1})$, so because $n_0 \geq w_{2,j} = \ord_t(-f_{1,j}t^i)$ (resp. $n_0 \geq w_{1,j} = \ord_t(f_{1,1}f_{2,j})$), we have $\ord_t(q_{2,j}) = w_{2,j}$ (resp. $\ord_t(q_{1,j}) = w_{1,j}$). Because $g \in \SL_2(k'\llbracket t \rrbracket)$, the matrices $\widetilde{\psi}$ and $g \cdot \widetilde{\psi}$ have the same minors, so
\begin{align*}
	\ord_t(h_{1,j}) &= \ord_t(-q_{2,j}t^{-i}) = w_{2,j} - i, &\text{for all $j = 3, \dots, r$},\\
	\ord_t(h_{2,j}) &= \ord_t(q_{1,j}h^{-1}t^i) =w_{1,j} - w_{1,2} + i, &\text{for all $j = 3, \dots, r$}.
\end{align*}
We have thus verified that $g \cdot \widetilde{\psi} \in Z$, so $\psi \in \cC$, and thus we have finished showing
\[
	\cC= \theta_{n_0}^{-1}\left( \sL_{n_0}(\rho)(Z_{n_0}) \right).
\]
Because $\sL_{n_0}(\rho)(Z_{n_0})$ is a constructible subset of $\sL_n(\cX)$ by Chevalley's Theorem for Artin stacks \cite[Theorem 5.2]{HallRydh2017}, this implies that $\cC$ is a cylinder. We will use the remainder of this proof to compute $\mu_\cX(\cC)$.

Set $n_1 = \max(n_0, 2i-1)$. For any $n \geq n_1$, let $H_n$ be the closed subscheme of $\sL_n(\bA^{2 \times 2}_k)$ whose $A$-valued points, for any $k$-algebra $A$, are the 2 by 2 matrices of the form
\[
	\begin{pmatrix} 1+g_{1,1}t^{n+1-i} & g_{1,2}t^{n+1-i} \\ g_{2,1}t^{n+1-w_{1,2}+i} & 1-g_{1,1}t^{n+1-i} \\  \end{pmatrix}
\]
for some $g_{1,2}, g_{2,1}, g_{2,2} \in A[t]/(t^{n+1})$. By our choice of $n_1$,
\[
	\det \begin{pmatrix} 1+g_{1,1}t^{n+1-i} & g_{1,2}t^{n+1-i} \\ g_{2,1}t^{n+1-w_{1,2}+i} & 1-g_{1,1}t^{n+1-i} \\  \end{pmatrix} = 1,
\]
and
\begin{align*}
	\begin{pmatrix} 1+g_{1,1}t^{n+1-i} & g_{1,2}t^{n+1-i} \\ g_{2,1}t^{n+1-w_{1,2}+i} & 1-g_{1,1}t^{n+1-i} \\  \end{pmatrix}& \begin{pmatrix} 1+g'_{1,1}t^{n+1-i} & g'_{1,2}t^{n+1-i} \\ g'_{2,1}t^{n+1-w_{1,2}+i} & 1-g'_{1,1}t^{n+1-i} \\  \end{pmatrix} \\
	&= \begin{pmatrix} 1+(g_{1,1}+g'_{1,1})t^{n+1-i} & (g_{1,2}+g'_{1,2})t^{n+1-i} \\ (g_{2,1}+g'_{2,1})t^{n+1-w_{1,2}+i} & 1-(g_{1,1}+g'_{1,1})t^{n+1-i} \\  \end{pmatrix},
\end{align*}
so the inclusion $H_n \hookrightarrow \sL_n(\bA^{2 \times 2}_k)$ factors through a closed immersion $H_n \hookrightarrow \sL_n(\SL_2)$ that gives $H_n$ the structure of a closed subgroup of $\sL_n(\SL_2)$, and $H_n \cong \bG_a^{w_{1,2} + i}$ as algebraic groups.

For any $n \in \N$, set $\widetilde{C}_n = \sL_n(\rho)^{-1}\left(\theta_n(\cC)\right) \subset \sL_n(\widetilde{X})$. We will show that for $n \geq n_1$,
\begin{equation}\label{ComputingClassOfTildeCinExampleSL}
	\e(\widetilde{C}_n) = \e(Z_n)\e(\sL_n(\SL_2))\e(H_n)^{-1} = \e(Z_n)\e(\sL_n(\SL_2))\bL^{-w_{1,2}-i} \in K_0(\Stack_k).
\end{equation}
By (\ref{EquationTruncationOfStackSetIsImageOfTruncationOfCoverSet}) and \autoref{JetSchemeOfQuotientStackIsQuotientOfJetSchemes}, $\widetilde{C}_n$ is equal to the image of the morphism

\[
	\sL_n(\SL_2) \times_k Z_n \to \sL_n(\widetilde{X})
\]
induced by the action of $\sL_n(\SL_2)$ on $\sL_n(\widetilde{X})$. Thus to show (\ref{ComputingClassOfTildeCinExampleSL}), it is sufficient to show that for any field extension $k'$ of $k$ and any $\widetilde{\psi}_n \in \widetilde{C}_n(k')$, the fiber of $\sL_n(\SL_2) \times_k Z_n \to \sL_n(\widetilde{X})$ over $\widetilde{\psi}_n$ is isomorphic to $H_n \otimes_k k'$. Because $H_n$ is special and thus $H_n$-torsors over $k'$ are trivial, it is sufficient to show that for any field extension $k'$ of $k$ and $\widetilde{\psi}_n \in Z_n(k')$,
\begin{itemize}

\item the stabilizer of $\widetilde{\psi}_n$ under the $\sL_n(\SL_2)$ action is $H_n \otimes_k k'$, and

\item if $g_n \in \sL_n(\SL_2)(k')$ is such that $g_n \cdot \widetilde{\psi}_n \in Z_n(k')$ then $g_n \cdot \widetilde{\psi}_n = \widetilde{\psi}_n$.

\end{itemize}
To show both of these items, it is sufficient to show that for any $g_n \in \sL_n(\SL_2)(k')$,
\[
	g_n \cdot \widetilde{\psi}_n \in Z_n(k') \implies g_n \in H_n(k') \implies g_n \cdot \widetilde{\psi}_n = \widetilde{\psi}_n,
\]
where we note that because $k$ has characteristic 0, we only need to show that $H_n \otimes_k k'$ and the stabilizer of $\widetilde{\psi}_n$ have the same underlying subset of $\sL_n(\SL_2) \otimes_k k'$. Write
\[
	\widetilde{\psi}_n = \begin{pmatrix} u_{1,1} t^{w_{1,2} - i} & 0 & \hdots & u_{1,j} t^{w_{2,j} - i} & \hdots \\ 0 & t^i & \hdots & u_{2,j}t^{w_{1,j} - w_{1,2} + i} & \hdots \end{pmatrix},
\]
and write
\[
	g_n = \begin{pmatrix} a & b \\ c & d \\  \end{pmatrix}.
\]
with $a,b,c,d \in k'[t]/(t^{n+1})$. Then
\begin{equation}\label{ExampleSLstabilizerMatrixMultiplicationFormula}
	g_n \cdot \widetilde{\psi}_n = \begin{pmatrix} au_{1,1} t^{w_{1,2} - i} & bt^i & \hdots & au_{1,j}t^{w_{2,j}-i} + bu_{2,j}t^{w_{1,j}-w_{1,2}+i} & \hdots \\ cu_{1,1}t^{w_{1,2}-i} & dt^i & \hdots & cu_{1,j}t^{w_{2,j}-i} + du_{2,j}t^{w_{1,j}-w_{1,2}+i} & \hdots \end{pmatrix}.
\end{equation}
We start with the first above implication, so suppose $g_n \cdot \widetilde{\psi}_n \in Z_n(k')$. Then $bt^i = cu_{1,1}t^{w_{1,2}-i} = 0$ and $dt^i = t^i$, so $b = g_{1,2}t^{n+1-i}$, $c = g_{2,1}t^{n+1 - w_{1,2}+i}$, $d = 1 - g_{1,1}t^{n+1-i}$ for some units $g_{1,2}, g_{2,1}, g_{1,1}$ in $k'[t]/(t^{n+1})$. By our choice of $n_1$,
\[
	1 = \det g_n = a(1 - g_{1,1}t^{n+1-i}) - g_{1,2}t^{n+1-i}g_{2,1}t^{n+1 - w_{1,2}+i} = a(1 - g_{1,1}t^{n+1-i}),
\]
so
\[
	a = (1 - g_{1,1}t^{n+1-i})^{-1} = 1 + g_{1,1}t^{n+1-i},
\]
and we have $g_n \in H_n(k')$. The second above implication is a straightforward application of (\ref{ExampleSLstabilizerMatrixMultiplicationFormula}) and (\ref{ConditionOnValuationsForSL2Example}). We have therefore finished showing that $\e(\widetilde{C}_n) = \e(Z_n)\e(\sL_n(\SL_2))\bL^{-w_{1,2}-i}$.

Now we can complete the computation of $\mu_\cX(\cC)$. Set
\[
	s = w_{1,3} + \dots + w_{1,r} + w_{2,3} + \dots + w_{2,r}.
\]
By \autoref{JetSchemeOfQuotientStackIsQuotientOfJetSchemes} and the fact that $\sL_n(\SL_2)$ is a special group, for any $n \geq n_1$
\[
	\e(\theta_n(\cC)) = \e(\widetilde{C}_n) \e(\sL_n(\SL_2))^{-1} = \e(Z_n)\bL^{-w_{1,2}-i} =  (\bL-1)^{2r-3}\bL^{n(2r-3) + (r-4)w_{1,2} - s},
\]
where the last equality is by construction of $Z_n$. Therefore noting that $\dim\cX = \dim(\bA^{2 \times r}_k) - \dim\SL_2 = 2r-3$,
\[
	\mu_{\cX}(\cC) = \lim_{n \to \infty} \e(\theta_n(\cC)) \bL^{-(n+1)\dim\cX} = (\bL-1)^{2r-3} \bL^{-(2r-3)+(r-4)w_{1,2}-s}.\qedhere
\]
\end{proof}

We next compute the value of $\sep_\pi$ on our sets of interest.

\begin{proposition}
\label{SL2ExampleOneToOne}
Let $i, w_{1,2} \in \N$ and $\bw_1, \bw_2 \in \N^{r-2}$ satisfy $($\ref{ConditionOnValuationsForSL2Example}$)$. Then the image of $\cC^{(i, w_{1,2}, \bw_1, \bw_2)}$ under $\sL(\pi)$ is equal to $C^{(w_{1,2}, \bw_1, \bw_2)}$, and $\sep_{\pi, \cC^{(i, w_{1,2}, \bw_1, \bw_2)}}$ is equal to 1 on all of $C^{(w_{1,2}, \bw_1, \bw_2)}$.
\end{proposition}

\begin{proof}
The image of $\cC^{(i, w_{1,2}, \bw_1, \bw_2)}$ is contained in $C^{(w_{1,2}, \bw_1, \bw_2)}$ by construction of $Z^{(i, w_{1,2}, \bw_1, \bw_2)}$. We now only need to show that $\sep_{\pi, \cC^{(i, w_{1,2}, \bw_1, \bw_2)}}$ is equal to 1 on all of $C^{(w_{1,2}, \bw_1, \bw_2)}$. Let $k'$ be a field extension of $k$ and let $\varphi \in C^{(w_{1,2}, \bw_1, \bw_2)}(k')$. Then for any $\{\ell, m\} \in {[r] \choose 2}$, let $h_{\ell, m} \in k'\llbracket t \rrbracket$ be such that $\varphi: \Spec(k'\llbracket t \rrbracket) \to X$ pulls back $x_{\ell,m}|_X$ to $h_{\ell,m}$. Then $\varphi$ is the image of the element $\widetilde{\psi}$ of $Z^{(i, w_{1,2}, \bw_1, \bw_2)}(k')$ given by the matrix
\[
	\begin{pmatrix} h_{1,2}t^{-i} & 0 & -h_{2,3}t^{-i} & -h_{2, 4}t^{-i} & \hdots & -h_{2, r}t^{-i} \\ 0 & t^i  & h_{1,3} h_{1,2}^{-1} t^i & h_{1,4} h_{1,2}^{-1} t^i & \hdots & h_{1,r} h_{1,2}^{-1} t^i\end{pmatrix},
\]
so $\sep_{\pi, \cC^{(i, w_{1,2}, \bw_1, \bw_2)}} \geq 1$ on all of $C^{(w_{1,2}, \bw_1, \bw_2)}$. Let $\psi \in \sL(\cX)(k')$ be the image of $\widetilde{\psi}$. To finish this proof, we only need to show that if $\psi' \in \sL(\cX)(k')$ has class in $|\sL(\cX)|$ contained in $\cC^{(i, w_{1,2}, \bw_1, \bw_2)}$ and satisfies $\sL(\pi)(\psi') = \varphi$, then $\psi' \cong \psi$. Because $\SL_2$ is a special group and thus has only trivial torsors over $\Spec(k'\llbracket t \rrbracket)$, there exists $\widetilde{\psi}' \in \sL(\widetilde{X})(k')$ whose image is isomorphic to $\psi'$. By construction, there exists some field extension $k''$ of $k'$ and $g \in \SL_2(k''\llbracket t \rrbracket)$ such that $g \cdot \widetilde{\psi}'_{k''} \in Z^{(i, w_{1,2}, \bw_1, \bw_2)}(k'')$, where $\widetilde{\psi}'_{k''}$ is the composition of $\Spec(k''\llbracket t \rrbracket) \to \Spec(k' \llbracket t \rrbracket)$ with $\widetilde{\psi}': \Spec(k'\llbracket t \rrbracket) \to \widetilde{X}$. By construction of $Z^{(i, w_{1,2}, \bw_1, \bw_2)}$, the entries of $g \cdot \widetilde{\psi}'_{k''}$ are all determined by its minors and thus $g \cdot \widetilde{\psi}'_{k''}$ is equal to the composition of $\Spec(k''\llbracket t \rrbracket) \to \Spec(k' \llbracket t \rrbracket)$ with $\widetilde{\psi}: \Spec(k'\llbracket t \rrbracket) \to \widetilde{X}$. Because $\widetilde{\psi}$ has nonzero (1,2) minor, this implies $g \in \SL_2(k'\llbracket t \rrbracket)$, so $g \cdot \widetilde{\psi'} = \widetilde{\psi}$ and $\psi' \cong \psi$ as desired.
\end{proof}

\autoref{SL2ExampleExceptionalLocusCodimension2} and \autoref{SL2ExampleOneToOne} tell us that if $r \geq 3$, $X$ has log-terminal singularities, and $i, w_{1,2} \in \N$ and $\bw_1, \bw_2 \in \N^{r-2}$ satisfy (\ref{ConditionOnValuationsForSL2Example}), then \autoref{conjectureGorensteinMeasureAndStackMeasure} implies that $\mu_\cX(\cC^{(i, w_{1,2}, \bw_1, \bw_2)}) = \mu^{\Gor}_X(C^{(w_{1,2}, \bw_1, \bw_2)})$. For $r \in \{3,4\}$, we verify this unconditionally by combining the above results, thereby proving \autoref{Conjecture1forGr24}.


\begin{proof}[{Proof of \autoref{Conjecture1forGr24}}]
Let $r \in \{3,4\}$. It is clear that $\mathcal{S}$ is an infinite set. The third and fourth bullet points are justified by \autoref{SL2ExampleRemarkAboutNegligibleSubsetsOfX} and \autoref{rmk:SL2specialPoint}. Next, a straightforward computation using \autoref{SL2ExampleRemarkAboutNegligibleSubsetsOfX}, \autoref{SL2ExampleComputationOfGorensteinMeasure}, and \autoref{SL2ExampleVolumeOfStackCylinders} shows
\[
0 \neq \mu_\cX(\cC^{(i, w_{1,2}, \bw_1, \bw_2)}) = \mu^{\Gor}_X(C^{(w_{1,2}, \bw_1, \bw_2)}).
\]
By \autoref{rmk:reduceToSigmaId} and \autoref{SL2ExampleOneToOne}, this implies
\[
0 \neq \mu_\cX(\cC) = \mu_\cX(\cC \cap \sL(\pi)^{-1}(C)) =  \int_C \sep_{\pi, \cC} \diff\mu^\Gor_X
\]
for all $(\cC,C)\in\mathcal{S}'$.
\end{proof}

\bibliographystyle{alpha}
\bibliography{SITAS}

\end{document}